\documentclass[11pt,reqno]{amsart}

\usepackage{amsmath,amssymb,amsthm}
\usepackage[normalem]{ulem}
\usepackage{graphicx}
\usepackage{enumerate}
\usepackage{tikz,tikz-cd}
\usepackage[letterpaper,left=1.5cm,right=3.5cm]{geometry}
\usepackage{bbm}
\usepackage{bm}
\definecolor{darkred}{rgb}{0.7,0,0} 
\newcommand{\defncolor}{\color{darkred}}
\newcommand{\defn}[1]{{\defncolor\textit{#1}}} 
\usetikzlibrary{calc}
\usetikzlibrary{arrows}
\usepackage{tikzsymbols}
\usetikzlibrary{backgrounds}

\usepackage{algorithm,algpseudocode}
\usetikzlibrary{math}

\newcommand{\PP}{\mathbb{P}}
\newcommand{\NN}{\mathbb{N}}
\newcommand{\zz}{\mathbf{z}}
\newcommand{\V}{\mathbf{V}}

\newtheorem{theorem}{Theorem}
\numberwithin{theorem}{section} 

\newtheorem{prop}[theorem]{Proposition}
\newtheorem{lemma}[theorem]{Lemma}

\newtheorem{coro}[theorem]{Corollary}

\newtheorem{definition}[theorem]{Definition}

\theoremstyle{remark}

\newtheorem{remark}{Remark}
\newtheorem{example}{Example}

\newtheorem{question}{Question}

\newenvironment{mygraph}[1][]{%
\begin{tikzpicture}[#1,every loop/.style={distance=10mm},every node/.style={circle, draw, fill=black,
                    inner sep=0pt, minimum width=3.5pt,font=\small}]}
{\end{tikzpicture}}
\newenvironment{mygraph2}[1][]{%
\begin{tikzpicture}[#1,every node/.style={circle, draw, fill=black,
                    inner sep=0pt, minimum width=3pt,font=\small}]}
{\end{tikzpicture}}

\newcommand{\ZZ}{\mathbb{Z}}

\newcommand{\QQ}{\mathbb{Q}}

\newcommand{\MM}{\mathbb{M}}
\newcommand{\x}{\mathbf{x}}
\newcommand{\st}{\mathfrak{st}}

\title[Marked Graphs and the chromatic symmetric function]{Marked graphs and the chromatic symmetric function}

\author[Aliste-Prieto]{Jos\'e Aliste-Prieto}
\address[J. Aliste-Prieto]{Departamento de Matem\'aticas, Facultad de Ciencias Exactas, Universidad Andres Bello, Chile}
\email{jose.aliste@unab.cl}
\urladdr{{http://www.mat-unab.cl/~jaliste/}}

\author[De Mier] {Anna de Mier}
\address[A. de Mier]{Departament de Matem\`atiques and Institut de Matem\`atiques de la UPC IMTech, Universitat Polit\`ecnica de Catalunya - BarcelonaTech, Barcelona 08034, Spain}
\email{anna.de.mier@upc.edu}
\urladdr{https://web.mat.upc.edu/anna.de.mier}

\author[Orellana]{Rosa Orellana}
\address[R. Orellana]{Mathematics Department, Dartmouth College, 
Hanover, NH 03755, U.S.A.}
\email{Rosa.C.Orellana@dartmouth.edu}
\urladdr{{https://math.dartmouth.edu/~orellana/}}

\author[Zamora] {Jos\'e Zamora}
\address[J. Zamora]{Departamento de Matem\'aticas, Facultad de ciencias exactas, Universidad Andres Bello, Chile}
\email{josezamora@unab.cl}
\urladdr{{http://www.mat-unab.cl/~jzamora/}}

\date{\today}

\begin{document}

\begin{abstract}
The main result of this paper is the introduction of marked graphs and the marked graph polynomials ($M$-polynomial) associated with them. These polynomials can be defined via a deletion-contraction operation.  These polynomials are a generalization of the $W$-polynomial introduced by Noble and Welsh and a specialization of the $\mathbf{V}$-polynomial introduced by Ellis-Monaghan and Moffatt. In addition, we describe an important specialization of $M$-polynomial which we call the $D$-polynomial.  Furthermore, we give an efficient algorithm for computing the chromatic symmetric function of a graph in the \emph{star-basis} of symmetric functions. As an application of these tools, we prove that  proper trees of diameter at most 5 can be reconstructed from its chromatic symmetric function. 

\end{abstract}
\maketitle

\section{Introduction}

In the mid 1990s Stanley \cite{stanley95symmetric} introduced a symmetric function {generalization of the chromatic polynomial}. For any graph $G$, if $x_1, x_2, \ldots$ are commuting variables then 
\[X_G = \sum_{\kappa}x_{\kappa(v_1)}x_{\kappa(v_2)}\cdots x_{\kappa(v_n)}~, \]
where $v_1, \ldots, v_n$ are the vertices of $G$ and the sum runs over all proper colorings $\kappa$ of $G$ with positive integers. The function $X_G$ is known as the \emph{chromatic symmetric function} of $G$.  If we set $x_1= x_2= \ldots = x_k = 1$ and all other variables to zero, $X_G$ specializes to the evaluation  $\chi_G(k)$ of the one variable chromatic polynomial, which counts the number of proper colorings of $G$ with $k$ colors.  In his seminal paper Stanley expressed $X_G$ using the classical bases of symmetric 
functions, proved many results and made several conjectures related to $X_G$. This function has generated a lot of research, for example \cite{aliste2014proper, crew2020deletion, dahlberg2017lollipop, gasharov1999stanley, guay2013modular,  heil2019algorithm, loebl2018isomorphism, martin2008distinguishing, orellana2014graphs,shareshian2016chromatic}.  

One open question is \emph{Stanley's tree isomorphism problem} which asks whether the chromatic symmetric function distinguishes non-isomorphic trees. This conjecture is known to hold for trees with less than 30 vertices  \cite{heil2019algorithm} and it has been verified for several subclasses of trees \cite{aliste2017trees,aliste2014proper,loebl2018isomorphism,
martin2008distinguishing,orellana2014graphs}.

The chromatic symmetric function  $X_G$ does not satisfy the deletion-contraction rule that is used to compute the one variable polynomial $\chi_G(k)$.  Nevertheless, it admits several linear modular relations: Guay-Paquet \cite{guay2013modular} and Orellana and Scott \cite{orellana2014graphs} found a four term modular relation for graphs with triangles. This relation was generalized to graphs with bigger cycles by Dalhberg and van Willigenburg \cite{dahlberg2017lollipop}.
The chromatic symmetric function is a morphism from the Hopf algebra of simple graphs to the Hopf algebra of symmetric functions \cite[Section 7.3]{grinberg2020hopf}; in  \cite{penaguiao2018kernel}, Penaguiao proved that the modular relations of Guay-Paquet, Orellana and Scott together with the isomorphism relation on graphs span  the  kernel of  the  chromatic  symmetric  function. 
In this paper we introduce the \emph{deletion-near-contraction relation} which can be interpreted as a faster version of the relation of Guay-Paquet,  Orellana and Scott.

In \cite{noble99weighted}, Noble and Welsh studied a larger class of graphs known as \emph{weighted graphs}, these are graphs with the vertices labeled by positive integers, called \emph{weights}. In this setting, 
they introduced the $W$-polynomial using a deletion-contraction rule, and showed that via a simple substitution one can obtain $X_G$ from the $W$-polynomial as a linear combination of the power sum basis of symmetric functions.  Recently, Crew and Spirkl \cite{crew2020deletion} introduced the weighted chromatic symmetric function $X_{(G,\omega)}$, for any weighed graph $(G,\omega)$, which satisfies a deletion-contraction relation and whose power sum expansion is a specialization of the $W$-polynomial.

In \cite{Scott-thesis}, Scott postulated that working with different bases of symmetric functions arising from graphs might lead to an approach to Stanley's tree isomorphism problem. In fact Scott uses a multiplicative basis constructed from paths in his work. Cho and van Willigenburg \cite{cho2016chromatic} generalized this  construction and defined multiplicative bases of symmetric functions from any sequence of connected graphs $(G_n)_n$ such that for each $n$ the graph $G_n$ has $n$ vertices. In this article we are interested in working with the basis constructed from star graphs, that is, when $G_n$ is the tree with one vertex of degree $n-1$ and all other vertices having degree $1$.  The path-basis and star-basis have been used in the study of the chromatic symmetric function in several works \cite{aliniaeifard2020extended, Scott-thesis,  chan2021tree}.
   
In this paper we give an efficient algorithm for computing $X_G$ in the star basis.  This algorithm relies on the deletion-near-contraction relation. To further exploit this relation, we work with a larger class of graphs which we call \emph{marked graphs}, these are graphs where each vertex is labeled by a pair of integers $(w,d)$. 
The number $w$ represents a vertex weight as in the $W$-polynomial, and $d$ is a non-negative integer that keeps track of edge information.
We show that we can define a graph polynomial for marked graphs that we call the \emph{Marked graph polynomial} or \emph{M-polynomial}. This polynomial is defined via a deletion-contraction rule and can be seen as a specialization of the $\mathbf{V}$-polynomial of Ellis-Monaghan and Moffatt \cite{ellis2011tutte}.
By applying a substitution that we call the \emph{undotting substitution}, we obtain an important specialization of the $M$-polynomial. We call this specialization the $D$-polynomial and from it we obtain the star expansion of the chromatic symmetric function  via a substitution of the variables. 

This theory leads to a new technique to study Stanley's tree isomorphism problem. In particular, we use the $D$-polynomial to show that Stanley's tree isomorphism problem has a positive answer for proper trees with diameter less than or equal to $5$. A tree is \emph{proper} if each internal vertex is adjacent  at least one leaf. The numbers of trees of diameter 4 and 5 are given by sequences A000094 and A00147 in the OEIS~\cite{OEIS}. Note for instance that the number of trees of diameter $4$ on $n$ vertices grows as the number of partitions of the integer $n$. 

This paper is organized as follows: In Section~\ref{sec:prelim} we recall the main definitions and results that will be needed in the paper. Next, in Section~\ref{sec:algo} we introduce the deletion-near-contraction rule and give our algorithm. In Section~\ref{sec:Mpoly} we introduce marked graphs, the $M$-polynomial and prove some of its properties. In Section~\ref{sec:Dpoly} we prove our main result that $X_G$ is a specialization of the $M$-polynomial, to do this we introduce the $D$-polynomial and the \emph{leaf-absorption} operation. In addition, we characterize when the substitution providing the specialization of the $M$-polynomial into the $D$-polynomial produces a cancellation-free expression. In Section~\ref{sec:Mp} we study the relationship between the $M$-polynomial and the $D$-polynomial. 

Finally, in Section~\ref{sec:app} we discuss how the results of this paper can be used to further study Stanley's tree isomorphism problem and apply this discussion to show that all proper trees up to diameter $5$ are distinguished up to isomorphism by the chromatic symmetric function.

\subsection*{Acknowledgements}
 JAP was partially supported by FONDECYT REGULAR 1160975 CONICYT Chile. AdM was partially supported by PID2020-113082GB-I00 funded by MCIN/AEI. RO was partially supported by NSF grant DMS--300512. JZ was partially supported by FONDECYT REGULAR 1180994 CONICYT Chile.

\section{Preliminaries: Chromatic symmetric functions and graph polynomials}\label{sec:prelim}
\subsection{Graph theory background}\label{graphtheory} For the basic notions in graph theory we refer the reader to any introductory graph theory text such as \cite{west2001introduction}.  For convenience we include some of the definitions here. 

A \defn{graph} $G$ is a pair $(V,E)$, where $V$ is a finite set of \defn{vertices} and $E$ is a multiset of \emph{edges}.
Each edge $e$  is a (unordered) pair of two (possibly the same) vertices $u$ and $v$. { The vertices $u,v$ are called the \defn{endpoints} of $e$ and we also say that $e$ is \emph{incident} with $u$ (and $v$)}.
Observe that our definition of a graph allows for multiple (also called parallel) edges and loops.
A \defn{loop} of $G$ is an edge connecting a vertex to itself and we say that two edges are \defn{parallel} (or multiple) if they connect the same vertices.  If a graph $G$ does not contain any parallel edges nor loops, then we say that $G$ is \defn{simple}. Given a graph $G$ we denote by $G^s$ the simple graph obtained by replacing each maximal class of parallel edges by a single edge and removing all loops.

Let $G = (V,E)$ be a graph.  A \defn{pendant edge} of $G$ is an edge incident to a vertex of degree $1$, which is called a \defn{leaf} of $G$. A \emph{non-loop} edge is called \defn{internal} if both its endpoints have degree greater than 1. 

Let $\PP$ be the set of positive integers. 
A \defn{proper coloring} of $G$ is a function $\kappa:V\rightarrow\PP$ such that $\kappa(u)\neq \kappa(v)$ whenever $uv\in E$. If the image of $\kappa$ is a subset of $\{1,2,\ldots,k\}$, we say that $\kappa$ is a $k$-coloring.  The number of $k$-colorings of $G$ is denoted  $\chi_G(k)$. It is well-known that {$\chi_G(k)$} is a polynomial in $k$ and it is called the
the \defn{chromatic polynomial} of $G$.

{ One key fact about the chromatic polynomial is that it satisfies a \defn{deletion-contraction} formula}: If $e$ is a non-loop edge of $G$ incident to vertices $u$ and $v$ then 
\begin{equation}\label{DC-rel}
\chi_G(k) = \chi_{G\setminus e}(k)-\chi_{G/e}(k),
\end{equation}
where $G\setminus e$ is the graph obtained by deleting $e$ and $G/e$ is the graph obtained by \emph{contracting} $e$, that is, $G{/ e}$ is the graph obtained by deleting $e$, $u$ and $v$, then adding a new vertex $v_e$ and replacing each edge $e'$ in $G$ incident with $u$ or $v$ with an edge $e''$ that is the same edge as $e'$ but with $u$ and $v$ replaced by $v_e$. We say that $G/e$ contracts $e$ into the \defn{contracted vertex} $v_e$.

\begin{figure}[hbt!]
\begin{center}
\begin{tikzpicture}[scale = 0.5,thick, baseline={(0,-1ex/2)}]
\tikzstyle{vertex} = [black, shape = circle, fill=black, minimum size = 3.5pt, inner sep = 1pt]
\node[vertex] (G7) at (12.0, -2) [shape = circle, draw] {};
\node[vertex] (G6) at (10.0, 0) [shape = circle, draw] {};
\node[vertex] (G5) at (12.0, 2) [shape = circle, draw] {};
\node[vertex] (G4) at (14.0, 2) [shape = circle, draw] {};
\node[vertex] (G3) at (16.0, 1.5) [shape = circle, draw] {};
\node[vertex] (G2) at (14.0, -2) [shape = circle, draw] {};
\node[vertex] (G1) at (16.0, -1.5) [shape = circle, draw] {};
\draw[thick] (G7)--(G6);
\draw[thick] (G7) -- (G2);
\draw[thick] (G6) -- (G5);
\draw[thick] (G5) -- (G4);
\draw[thick] (G4) -- (G3);
\draw[thick] (G3) -- (G1);
\draw[thick] (G1) -- (G2);
\draw[thick] (G7) -- (G5) node[midway,left] {$e$};
\draw[thick] (G5) -- (G2);
\draw[thick] (G4)  .. controls +(0.5, -0.5) and +( 0.5, 0.5) .. (G2);
\draw[thick] (G2)  .. controls +(-0.5, 0.5) and +( -0.5, -0.5) .. (G4);
\node[] at (12,-2.5) {$u$};
\node[] at (12,2.5) {$v$};
\end{tikzpicture}\qquad \qquad \quad \quad
\begin{tikzpicture}[scale = 0.5,thick, baseline={(0,-1ex/2)}]
\tikzstyle{vertex} = [black, shape = circle, fill=black, minimum size = 3.5pt, inner sep = 1pt]
\node[vertex] (G6) at (10.0, 0) [shape = circle, draw] {};
\node[vertex] (G5) at (12.0, 0) [shape = circle, draw] {};
\node[vertex] (G4) at (14.0, 2) [shape = circle, draw] {};
\node[vertex] (G3) at (16.0, 1.5) [shape = circle, draw] {};
\node[vertex] (G2) at (14.0, -2) [shape = circle, draw] {};
\node[vertex] (G1) at (16.0, -1.5) [shape = circle, draw] {};
\draw[thick] (G5)  .. controls +(-0.5, 0.5) and +( 0.5,0.5) .. (G6);
\draw[thick] (G6)  .. controls +(0.5, -0.5) and +(-0.5,-0.5) .. (G5);
\draw[thick] (G2)  .. controls +(-0.5, -0.5) and +( -0.5,0.5) .. (G5);
\draw[thick] (G5)  .. controls +(0.5, 0.5) and +( 0.5,-0.5) .. (G2);
\draw[thick] (G4)  .. controls +(0.5, -0.5) and +( 0.5,0.5) .. (G2);
\draw[thick] (G2)  .. controls +(-0.5, 0.5) and +( -0.5,-0.5) .. (G4);
\draw[thick] (G5) -- (G4);
\draw[thick] (G4) -- (G3);
\draw[thick] (G3) -- (G1);
\draw[thick] (G1) -- (G2);
\node[] at (12,.7) {$v_e$};
\end{tikzpicture}
\end{center}
\caption{Classical contraction of edges}
\label{fig:contraction}
\end{figure}
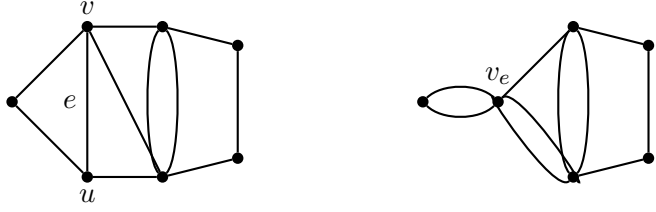

\begin{remark}
If $G$ is a graph with loops, then it cannot be { properly} colored and hence $\chi_G(k)=0$. { On the other hand,} if $G$ does not have loops but has parallel edges, removing duplicated edges does not affect the proper colorings of $G$. Thus, in this case $\chi_G(k)=\chi_{G^s}(k)$. 
Because of this {fact}, the chromatic polynomial is usually defined for simple graphs. Here we choose to define it for (general) graphs for consistency with the other invariants being consider in this article. 
In addition, we will apply the contraction operation \emph{only} to non-loop edges. Thus, the contraction operation will always reduce the number of vertices.
\end{remark}

\subsection{The chromatic symmetric function}
\label{CSF}
Before introducing the chromatic symmetric function, we review some basic facts about symmetric functions that will be needed in this paper. For more details, see \emph{e.g.} \cite{macdonald1998symmetric, sagan2001symmetric, stanley1999enumerative}.

A \defn{partition} is a sequence $\lambda = (\lambda_1,\lambda_2,\cdots,\lambda_l)$ of positive integers such that $\lambda_1\geq \lambda_2\geq\cdots\geq\lambda_l$. Each $\lambda_i$ is called a \defn{part} of $\lambda$ and we say that $\lambda$ is a partition of $n$, denoted $\lambda\vdash n$, if $\sum_i\lambda_i = n$. The \defn{length} of $\lambda$
is the number of parts. 

Let $\mathbf{x} = x_1,x_2,\ldots$ be a countably infinite family of commuting indeterminates and consider the formal power series ring $\QQ[[\mathbf{x}]]$.  A monomial 
$x_{i_1}^{\lambda_1}x_{i_2}^{\lambda_2}\cdots x_{i_l}^{\lambda_l}$ has \defn{degree} $n$ if $n = \sum_{i}\lambda_i$ and $f(\mathbf{x})$ in $\QQ[[\mathbf{x}]]$ is \defn{homogeneous of degree $n$} if all monomials in $f(\mathbf{x})$ have degree $n$. 

Denote by $S_n$ the symmetric group with $n$ elements. There is a natural action of $\bigcup_n S_n$ on $\QQ[[\mathbf{x}]]$, which is defined  by 
\[\pi f(x_1,x_2,x_3,\ldots) = f(x_{\pi(1)}, x_{\pi(2)}, x_{\pi(3)},\ldots)\]
where $\pi\in\cup_n S_n$, that is, there is a positive integer $k$ such that $\pi(i)=i$ for all $i>k$. The ring of \defn{symmetric functions}, denoted $\mathsf{Sym}$, is the subset of $\QQ[[\mathbf{x}]]$ invariant under this action. We denote by $\mathsf{Sym}_n$ the homogeneous symmetric functions of degree $n$.
It is well-known that the dimension of $\mathsf{Sym}_n$ is equal to the number of partitions of $n$. In addition, there are several linear bases for this ring.  One of the best know basis is the \defn{power sum basis} defined as follows: for any positive integer $k$ let $p_k = \sum_{i\geq 1}x_{i}^k$
and for a partition $\lambda= (\lambda_1,\lambda_2,\cdots,\lambda_l)$ of $n$  define
\[p_\lambda = p_{\lambda_1}p_{\lambda_2}\cdots p_{\lambda_l}.\]
The family $\{p_\lambda\mid \lambda \vdash n\}$ is a basis for 
$\mathsf{Sym}_n$. 
The \defn{chromatic symmetric function}  \cite{stanley95symmetric} of a graph $G$ is the symmetric function defined by
\[X_G = \sum_{\kappa} x_{\kappa(v_1)}x_{\kappa(v_2)}\cdots x_{\kappa(v_n)},\]
where the sum ranges over all proper colorings $\kappa$ of $G$ and $\{v_1,v_2,\ldots,v_n\}$ denote the vertices of $G$. As it is the case for the chromatic polynomial, the chromatic symmetric function of a graph with loops is zero, and removing duplicated edges from $G$ does not change its chromatic symmetric function. 

The chromatic symmetric function $X_G$ is a homogeneous symmetric function of degree  the order of $G$.  In \cite{stanley95symmetric}, Stanley gave the expansion of $X_G$ in several bases of symmetric functions.  Here we describe the expansion in the {power sum} basis.  Given $A\subseteq E$, the \defn{spanning subgraph} of $G$ with edge set $A$ is the graph $G|_A$  with 
the same vertex set as $G$ and edge set $A$.  If $G$ has $n$ vertices, we let $\lambda(A)$ be the partition of $n$ whose parts are the orders of the connected components of $G|_A$.

\begin{theorem}[{\cite[Theorem 2.5]{stanley95symmetric}}]
\label{stanley1}
We have 
\[X_G = \sum_{A\subseteq E} (-1)^{|A|}p_{\lambda(A)}.\]
\end{theorem}
It is known that when $G$ is a tree, the partition $\lambda(A)$ has length $n-|A|$. Thus, in this case, given a  partition $\lambda$ of $n$, the absolute value of the coefficient of $p_\lambda$ in $X_G$ counts the number of edge sets $A$ such that the orders of the connected components of $G|_A$ induce the partition $\lambda$.

In \cite{cho2016chromatic}, Cho and van Willigenburg constructed new bases for the space of symmetric functions by using chromatic symmetric functions. 
Given two graphs $G$ and $H$, we use $G\sqcup H$ to denote the disjoint union of $G$ and $H$.
\begin{theorem}[\cite{cho2016chromatic}]
\label{thm:cho}
For any positive integer $k$, let $G_k$ denote a connected graph with $k$ vertices and $\{G_k\}_{k\in\NN}$ be a family of these graphs.  Given $\lambda\vdash n$ of length $l$, define $G_\lambda = G_{\lambda_1}\sqcup G_{\lambda_2}\sqcup\cdots\sqcup G_{\lambda_l}$. The set $\{X_{G_{\lambda}}\mid \lambda\vdash n\}$ is a basis for $\mathsf{Sym}_n$.
\end{theorem}
The bases constructed in Theorem~\ref{thm:cho} are called \defn{chromatic bases}. Among these, the one constructed from star graphs will play a fundamental role in this paper: Given $k$, denote by  $St_k$ the star graph with $k$ vertices, for example,

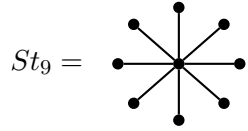
\begin{center}
$St_9 = \quad$\begin{tikzpicture}
[scale = 0.3,thick, baseline={(0,-1ex/2)}]
\tikzstyle{vertex} = [black, shape = circle, fill=black, minimum size = 3.5pt, inner sep = 1pt]
\node[vertex] (G7) at (12.0, -1.75) [shape = circle, draw] {};
\node[vertex] (G6) at (11.3, 0) [shape = circle, draw] {};
\node[vertex] (G5) at (12.0, 1.75) [shape = circle, draw] {};
\node[vertex] (G4) at (14.0, 2.5) [shape = circle, draw] {};
\node[vertex] (G3) at (16.0, 1.75) [shape = circle, draw] {};
\node[vertex] (G8) at (16.0, -1.75) [shape = circle, draw] {};
\node[vertex] (G2) at (14.0, -2.5) [shape = circle, draw] {};
\node[vertex] (G1) at (16.7, 0) [shape = circle, draw] {};
\node[vertex] (G0) at (14.0, 0) [shape = circle, draw] {};
\draw[thick] (G0) -- (G1);
\draw[thick] (G0) -- (G2);
\draw[thick] (G0) -- (G3);
\draw[thick] (G0) -- (G4);
\draw[thick] (G0) -- (G5);
\draw[thick] (G0) -- (G6);
\draw[thick] (G0) -- (G7);
\draw[thick] (G0) -- (G8);
\end{tikzpicture}
\end{center}

For a partition $\lambda = (\lambda_1, \ldots, \lambda_l)$, we denote by $St_\lambda=St_{\lambda_1}\sqcup\cdots \sqcup St_{\lambda_l}$. 
For a positive integer $k$, we define
$\st_k := X_{St_k}$ and for a partition $\lambda= (\lambda_1, \ldots, \lambda_k)$ we have $\st_\lambda := \st_{\lambda_1} \st_{\lambda_2}\cdots \st_{\lambda_l}$.  By Theorem~\ref{thm:cho}, $\{\st_{\lambda}\mid \lambda \vdash n\}$ is a basis for $\mathsf{Sym}_n$. The following result gives the change of basis between the power sum and the star chromatic basis.  
\begin{prop}
\label{involution}
We have 
\begin{equation}
    \label{eq:star_p_exp}
\st_{n+1} = \sum_{r=0}^n(-1)^r\binom{n}{r}p_{(r+1,1^{n-r})}
\end{equation}
and 
\begin{equation}
    \label{eq:p_star_exp}
{p}_{n+1} = \sum_{r=0}^n(-1)^r\binom{n}{r}\st_{(r+1,1^{n-r})}.
\end{equation}
In other words, there is a linear involution that sends $\st_\lambda$ into $p_\lambda$ and vice versa. 
\end{prop}
This proposition is a consequence of the \emph{chromatic reciprocity theorem} recently discovered 
by Aliniaeifard, Wang  and van Willigenburg \cite{aliniaeifard2020extended}. By combining this proposition (more precisely, \eqref{eq:p_star_exp})  with Theorem~\ref{stanley1} we obtain a formula for computing the $\st$-expansion of $X_{G}$ for every graph $G$. However, as seen in Example~\ref{ex:1}, many terms may cancel out. Hence, if we want to compute the $\st$-expansion of $X_G$, we need a better algorithm. This algorithm will be described in Section~\ref{sec:algo}.
\begin{example}
\label{ex:1}
If $G$ is the star graph with $n+1$ vertices, then Theorem~\ref{stanley1} combined with \eqref{eq:p_star_exp} yields
\[
X_G = \sum_{r=0}^n\sum_{k=0}^r (-1)^{r+k}\binom{n}{r}\binom{r}{k}\st_{(k+1,1^{n-k})}.
\]
Since $X_G=\st_{n+1}$, we see that in the right-hand side of the last equation all terms but one must cancel out. 

\end{example}

\subsection{Weighted graphs, their graph polynomials and symmetric functions}
\label{sec:Vpoly}

Given a graph $G = (V,E)$,  a \defn{weight function} is a map $\omega:V\rightarrow \PP$. A \defn{weighted graph} is a pair $(G,\omega)$, where $G$ is a graph and $\omega$ is a weight function. Given two weighted graphs $(G,\omega)$ and $(H,\omega')$, we say that they  are \defn{$\omega$-isomorphic} if there is an isomorphism $\phi$ from $G$ to $H$ such that \[\omega'(\phi(v))=\omega(v)\quad\text{for all $v$ in $V(G)$}.\]
Let $(G,\omega)$ be a weighted graph. The \defn{simplification of} $(G,\omega)$ is defined as $(G,\omega)^s = (G^s,\omega)$. Given an edge $e=v_1v_2$ of $G$, then $(G,\omega)\setminus e := (G\setminus e,\omega)$ is the weighted graph obtained by deleting $e$ and leaving $\omega$ unchanged.  The contraction is $(G/e,\omega/e)$, where $G/e$ is the graph obtained by contracting $e$ into the {contracted vertex} $v_e$ and the weight function  $\omega/e$ is defined by setting
\[(\omega/e) (v)=\begin{cases}\omega(v_1)+ \omega(v_2),&\text{if } v = v_e,\\
  \omega(v),& \text{ otherwise.}
  \end{cases}
  \]

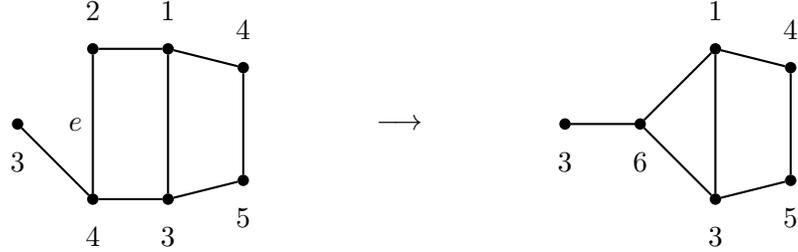
\begin{center}
\begin{tikzpicture}[scale = 0.5,thick, baseline={(0,-1ex/2)}]
\tikzstyle{vertex} = [black, shape = circle, fill=black, minimum size = 3.5pt, inner sep = 1pt]
\node[vertex=left 1] (G7) at (12.0, -2) [shape = circle, draw] {};
\node[] at (12,-3) {$4$};
\node[vertex] (G6) at (10.0, 0) [shape = circle, draw] {};
\node[] at (10,-1) {$3$};
\node[vertex] (G5) at (12.0, 2) [shape = circle, draw] {};
\node[] at (12,3) {$2$};
\node[vertex] (G4) at (14.0, 2) [shape = circle, draw] {};
\node[] at (14,3) {$1$};
\node[vertex] (G3) at (16.0, 1.5) [shape = circle, draw] {};
\node[] at (16,2.5) {$4$};
\node[vertex] (G2) at (14.0, -2) [shape = circle, draw] {};
\node[] at (14,-3) {$3$};
\node[vertex] (G1) at (16.0, -1.5) [shape = circle, draw] {};
\node[] at (16,-2.5) {$5$};
\draw[thick] (G7)--(G6);
\draw[thick] (G7) -- (G2);
\draw[thick] (G5) -- (G4);
\draw[thick] (G4) -- (G3);
\draw[thick] (G3) -- (G1);
\draw[thick] (G1) -- (G2);
\draw[thick] (G7) -- (G5) node[midway,left] {$e$};
\draw[thick] (G4) -- (G2);
\end{tikzpicture}\qquad  \quad \quad  $\longrightarrow$ \qquad \quad \quad   
\begin{tikzpicture}[scale = 0.5,thick, baseline={(0,-1ex/2)}]
\tikzstyle{vertex} = [black, shape = circle, fill=black, minimum size = 3.5pt, inner sep = 1pt]
\node[vertex] (G6) at (10.0, 0) [shape = circle, draw] {};
\node[] at (10,-1) {$3$};
\node[vertex] (G5) at (12.0, 0) [shape = circle, draw] {};
\node[] at (12,-1) {$6$};
\node[vertex] (G4) at (14.0, 2) [shape = circle, draw] {};
\node[] at (14,3) {$1$};
\node[vertex] (G3) at (16.0, 1.5) [shape = circle, draw] {};
\node[] at (16,2.5) {$4$};
\node[vertex] (G2) at (14.0, -2) [shape = circle, draw] {};
\node[] at (14,-3) {$3$};
\node[vertex] (G1) at (16.0, -1.5) [shape = circle, draw] {};
\node[] at (16,-2.5) {$5$};
\draw[thick] (G5)  -- (G6);
\draw[thick] (G2)  -- (G5);
\draw[thick] (G4) -- (G2);
\draw[thick] (G5) -- (G4);
\draw[thick] (G4) -- (G3);
\draw[thick] (G3) -- (G1);
\draw[thick] (G1) -- (G2);
\end{tikzpicture}
\end{center}

In \cite{crew2020deletion}, Crew and Spirkl extended the definition of the chromatic symmetric function to weighted graphs.  The \defn{weighted chromatic symmetric function} of a weighted graph $(G,\omega)$ 
is defined as   
\begin{equation}\label{eq:weightedchromatic}
X_{(G,\omega)}= \sum_{\kappa} x_{\kappa(v_1)}^{\omega(v_1)}x_{\kappa(v_2)}^{\omega(v_2)}\cdots x_{\kappa(v_l)}^{\omega(v_l)},
\end{equation}
where the sum ranges over all proper colorings $\kappa$ of the vertices of $G$.  If $(G,\omega)$ is \defn{unweighted}, that is, if $\omega(v)=1$ for every vertex, then we have that $X_{G}=X_{(G,\omega)}$.  If $e$ is an edge of $G$, then the weighted chromatic symmetric function satisfies the  deletion-contraction formula: 
\begin{equation}
\label{w_delcont}
X_{(G,\omega)} = X_{(G\setminus e,\omega)}-X_{(G/e,\omega/e)}.
\end{equation}

{As in the case of the chromatic symmetric function, if $(G,\omega)$ has a loop, then $X_{(G,\omega)}$ is zero. It is also easy to note that if $G$ has no loops, then 
\[ X_{(G,\omega)} = X_{(G^s,\omega)}.\]
This can bee seen from \eqref{eq:weightedchromatic} and also from \eqref{w_delcont}. Indeed, if $e$ and $e'$ are two parallel edges, then $G/e$ contains a loop and thus $X_{(G/e,\omega/e)}=0$. Therefore, after applying deletion-contraction to $e$ we get $X_{(G,\omega)}=X_{(G\setminus e,\omega)}$.  It follows that if $(G,\omega)$ is simple, then 
\[X_{(G,\omega)} = X_{G\setminus e,\omega} - X_{((G/e)^s,\omega/e)}.\]}

\begin{example} In this example we have a weighted graph $(G,\omega)$ such that contracting the edge $e$ produces a graph with parallel edges.  On the right, we show the simple weighted graph $((G/e)^s,\omega/e)$ which results after removing the parallel edges.
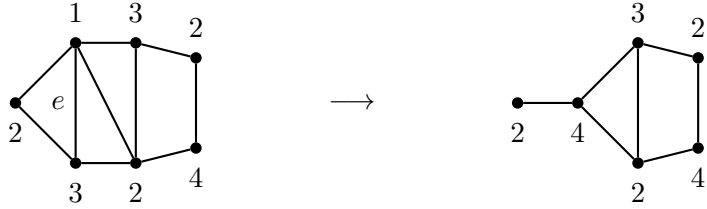
\begin{center}
\begin{tikzpicture}[scale = 0.4,thick, baseline={(0,-1ex/2)}]
\tikzstyle{vertex} = [black, shape = circle, fill=black, minimum size =3.5pt, inner sep = 1pt]
\node[vertex] (G7) at (12.0, -2) [shape = circle, draw] {};
\node[] at (12,-3) {$3$};
\node[vertex] (G6) at (10.0, 0) [shape = circle, draw] {};
\node[] at (10,-1) {$2$};
\node[vertex] (G5) at (12.0, 2) [shape = circle, draw] {};
\node[] at (12,3) {$1$};
\node[vertex] (G4) at (14.0, 2) [shape = circle, draw] {};
\node[] at (14,3) {$3$};
\node[vertex] (G3) at (16.0, 1.5) [shape = circle, draw] {};
\node[] at (16,2.5) {$2$};
\node[vertex] (G2) at (14.0, -2) [shape = circle, draw] {};
\node[] at (14,-3) {$2$};
\node[vertex] (G1) at (16.0, -1.5) [shape = circle, draw] {};
\node[] at (16,-2.5) {$4$};
\draw[thick] (G7)--(G6);
\draw[thick] (G7) -- (G2);
\draw[thick] (G6) -- (G5);
\draw[thick] (G5) -- (G4);
\draw[thick] (G4) -- (G3);
\draw[thick] (G3) -- (G1);
\draw[thick] (G1) -- (G2);
\draw[thick] (G7) -- (G5) node[midway,left] {$e$};
\draw[thick] (G5) -- (G2);
\draw[thick] (G4) -- (G2);
\end{tikzpicture}\qquad  \quad \quad $\longrightarrow$ \qquad \quad \quad
\begin{tikzpicture}[scale = 0.4,thick, baseline={(0,-1ex/2)}]
\tikzstyle{vertex} = [black, shape = circle, fill=black, minimum size = 3.5pt, inner sep = 1pt]
\node[vertex] (G6) at (10.0, 0) [shape = circle, draw] {};
\node[] at (10,-1) {$2$};
\node[vertex] (G5) at (12.0, 0) [shape = circle, draw] {};
\node[] at (12,-1) {$4$};
\node[vertex] (G4) at (14.0, 2) [shape = circle, draw] {};
\node[] at (14,3) {$3$};
\node[vertex] (G3) at (16.0, 1.5) [shape = circle, draw] {};
\node[] at (16,2.5) {$2$};
\node[vertex] (G2) at (14.0, -2) [shape = circle, draw] {};
\node[] at (14,-3) {$2$};
\node[vertex] (G1) at (16.0, -1.5) [shape = circle, draw] {};
\node[] at (16,-2.5) {$4$};
\draw[thick] (G5)  -- (G6);
\draw[thick] (G2)  -- (G5);
\draw[thick] (G4) -- (G2);
\draw[thick] (G5) -- (G4);
\draw[thick] (G4) -- (G3);
\draw[thick] (G3) -- (G1);
\draw[thick] (G1) -- (G2);
\end{tikzpicture}
\end{center}
\end{example}

Let $\mathbf{z}=z_1,z_2,\ldots$ be a family of commuting indeterminates and $y$ another indeterminate that commutes with all the $z_i$. The \defn{$W$-polynomial} introduced by Noble and Welsh in  \cite{noble99weighted}, is defined for every  weighted graph $(G,\omega)$  by the following rules: 
\begin{enumerate}[a)]
\item If $e$ is a loop in $G$, then 
\[ W_{(G,\omega)}(\mathbf z,y) = y W_{(G{\setminus e},\omega)}(\mathbf z,y).\]
\item If $e$ is an edge with different endpoints, then $W_G$ satisfies the \emph{deletion-contraction} formula: 
\begin{equation*}
W_{(G,\omega)}(\mathbf z,y) = W_{(G{\setminus e},\omega)}(\mathbf z,y) + W_{(G{/e},\omega/e)}(\mathbf z,y).
\end{equation*}
\item If $G$ consists of $k$ isolated vertices with corresponding weights $\omega_1$,$\omega_2$,\ldots,$\omega_k$, 
then 
\[W_{(G,\omega)}(\mathbf z,y) = z_{\omega_1}z_{\omega_2}\cdots z_{\omega_k}.\]
\end{enumerate}

These rules yield a well-defined polynomial \cite[Proposition 4.1]{noble99weighted} in the sense that the order in which we apply rules a) and b) does not affect its value. 

We need some basic definitions before stating the next theorem. Let $(G,\omega)$ be a weighted graph. Given a subset $U$ of vertices of $G$, the \defn{weight} of $U$, denoted by $\omega(U)$, is the sum of the weights of all the vertices in $U$. The \defn{total weight} of $(G,\omega)$, denoted $\omega(G)$, is the weight of $V(G)$. The \emph{number of connected components} of any graph $H$ is denoted by $k(H)$. For any subset $A$ of edges, the \defn{rank} of $A$, denoted by $r_G(A)$, is defined as 
\begin{equation}
    \label{rankDef}
r_G(A) = |V(G)| - k(G|_A).
\end{equation}
The \emph{partition} induced by $A$ in $(G,\omega)$, denoted by $\lambda(G,\omega,A)$, is the integer partition  determined by the total weights of the connected components of $(G|_A,\omega)$. If $\lambda$ is an integer partition
of length  $l$,  then $\mathbf{z}_\lambda :=  z_{\lambda_1}z_{\lambda_2}\cdots z_{\lambda_l}$.

\begin{theorem}[{\cite[Theorem 4.3]{noble99weighted}}]
\label{teoUtoX}
We have 
\begin{equation}
\label{Wstates}
W_{(G,\omega)}({\mathbf{z}},y) = 
\sum_{A\subseteq E(G)}
\mathbf{z}_{\lambda(G,\omega,A)}(y-1)^{|A|-r_G(A)}.
\end{equation}
\end{theorem}

It is known that if $G$ is a forest, then $|A|=r_G(A)$ for all $A\subseteq E(G)$. Thus the right hand side of \eqref{Wstates} does not depend on $y$. It follows in this case that given  $\lambda\vdash N$, where $N$ is the total weight of $(G,\omega)$, the coefficient of $\mathbf{z}_\lambda$ in $W_{(G,\omega)}(\mathbf{z})$ counts the number of edge sets $A$ such that $\lambda(G,\omega,A)=\lambda$.

Observe that the expansion of the chromatic symmetric functions in terms of the power sum basis, Theorem~\ref{stanley1}, is very similar to the expansion given by Theorem~\ref{teoUtoX}. In fact, the chromatic symmetric function of a graph $G$ is a specialization of the $W$-polynomial of the weighted graph assigning weight $1$ to every vertex \cite[Theorem 6.1]{noble99weighted} where we substitute each variable $z_i$ by $-p_i$ and let $y=0$. Actually this also holds for the weighted chromatic symmetric function: The power sum expansion of $X_{(G,\omega)}$ is given in \cite[Lemma 3]{crew2020deletion}. 
Thus, by performing the same substitution of variables in the $W$-polynomial for weighted graphs yields 
\begin{theorem}
\label{thm:noble1}
Let $(G,\omega)$ be a weighted graph. Then, 
\[X_{(G,\omega)} = (-1)^{|V(G)|}W_{(G,\omega)}(z_i = -p_i, y=0).\]
\end{theorem}

\subsection{Weighted graphs with weights in semigroups}

From the proof of \cite[Proposition 4.1]{noble99weighted} one can generalize the definition of the $W$-polynomial to weighted graphs with weights belonging to any commutative semigroup. In fact, Ellis-Monaghan and Moffatt \cite{ellis2011tutte} went further and introduced the $\V$-polynomial for weighted graphs also having weights in the edges. We review this polynomial now.

Fix a commutative semigroup $(S,\star)$. 
An $S$-weighted graph is a pair $(G,\mathsf{s})$ where $G$ is a graph and $\mathsf{s}:V\rightarrow S$ is a function. Given an edge $e=uv$ of $G$, the deletion and contraction operations extend naturally to $(G,\mathsf{s})$:
\begin{itemize}
\item {(deletion)} The $S$-weighted graph $(G\setminus e,\mathsf{s})$ is obtained by deleting the edge $e$ from $G$ and leaving $\mathsf{s}$ unchanged.
\item (contraction) If $e$ is not a loop, we consider $(G{/ e},\mathsf{s}/e)$, where $G/e$ is the contraction of $G$ by $e$ into the \emph{contracted vertex} $v_e$, and $\mathsf{s}/e$ is defined as 
\[\mathsf{s}/ e(u')=\begin{cases}\mathsf{s}(u)\star \mathsf{s}(v),&\text{if }u'=v_e,\\
										\mathsf{s}(u'),&\text{otherwise.}
\end{cases}
\]
\end{itemize}

Given an $S$-weighted graph $(G,\mathsf{s})$, edge weights $\pmb{\gamma}=\{\gamma_e\}_{e\in E}$, and a family of commuting indeterminates $\mathbf{z}=\{z_{s}\}_{s\in S}$, the \defn{$\V$-polynomial} of  $(G,\mathsf{s})$  is defined by the following rules:

\begin{enumerate}[a)]
\item If $(G,\mathsf{s})$ consists only of $k$ isolated vertices with weights  $s_1,s_2,\ldots,s_k$ in $S$, 
then 
\[\V_{(G,\mathsf{s})}(\mathbf{z},\pmb{\gamma}) = z_{s_1}z_{s_2}\cdots z_{s_k}.\]
\item If $e$ is a loop of $G$, then 
\[ \V_{(G,\mathsf{s})}(\mathbf z,\pmb{\gamma}) = (\gamma_e+1) \V_{(G\setminus e,\mathsf{s})}(\mathbf z,\pmb{\gamma}).\]
\item If $e$ is a non-loop edge of $G$, then 
\begin{equation}
\label{especial}
\V_{(G,\mathsf{s})}(\mathbf z,\pmb{\gamma}) = \V_{(G\setminus e,\mathbf{s})}(\mathbf z,\pmb{\gamma}) + \gamma_e \V_{(G/e,\mathsf{s}/e)}(\mathbf z,\pmb{\gamma}).
\end{equation}
\end{enumerate}
It is proved in \cite[Proposition 3.2]{ellis2011tutte} that the $\V$-polynomial is well-defined. Now we recall the main results from \cite{ellis2011tutte} that are needed in the sequel.

\begin{theorem}[{\cite[Theorem 3.4]{ellis2011tutte}}]
\label{recipeV}
Let $f$ be a complex valued function on $S$-weighted graphs defined recursively by the following rules: for $\alpha_e,\beta_e\in \mathbb{C}$ and each $\alpha_e\neq 0$
\begin{enumerate}
    \item if $e$ is a non-loop edge, then
    \[f(G,\mathsf{s}) = \alpha_e f(G{\setminus e},\mathsf{s}) + \beta_e f(G{/e},\mathsf{s}/e);\]
    \item if $e$ is a loop, then
    \[f(G,\mathsf{s}) = (\alpha_e+\beta_e)f(G{ \setminus e},\mathsf{s});\]
    \item if $(G,\mathsf{s})$ consists of isolated vertices with weights $s_1,\ldots,s_m$,
    \[f(G,\mathsf{s}) = \prod_{i=1}^m z_{s_i}.\]
\end{enumerate}
Then $f$ is well-defined and 
\[f(G,\mathsf{s}) = \left(\prod_{e\in E}\alpha_e\right) \V_{(G,\mathsf{s})}(\mathbf{z},\{\gamma_e = \beta_e/\alpha_e\}_{e\in E}).\]
\end{theorem}
\begin{coro}
Suppose $S=\PP$, the positive integers, and let $(G,\omega)$ be a weighted graph. Then 
\[ W_{(G,\omega)}(\mathbf{z},y)=(y-1)^{-|V(G)|}\V_{(G,\omega)}((y-1)z_1,(y-1)z_2,\ldots,\{\gamma_e=(y-1)\}_{e\in E}).\]
\end{coro}
The $\V$-polynomial also admits a representation as a sum over spanning subgraphs. Given $\mathsf{s}$ in $S$, we say that a multiset  $\{\mathsf{s}_1,\mathsf{s}_2,\ldots,\mathsf{s}_l\}$ is a \defn{partition} of $\mathsf{s}$ if $\mathsf{s}_1\star\mathsf{s}_2\star\cdots\star\mathsf{s}_k=\mathsf{s}$. An \defn{$S$-partition} is a partition of $\mathsf{s}$ for some $\mathsf{s}$ in $S$.
If $\lambda=\{\mathsf{s}_1,\mathsf{s}_2,\ldots,\mathsf{s}_l\}$ is an $S$-partition,
then $\mathbf{z}_\lambda :=  z_{\mathsf{s}_1}z_{\mathsf{s}_2}\cdots z_{\mathsf{s}_l}$.
The $S$-\emph{partition} induced by $A$ in $(G,\mathsf{s})$, denoted by
$\lambda(G,\mathsf{s},A)$, is the $S$-partition  determined by the total weights of the connected components of the spanning subgraph $(G|_A,\mathsf{s})$.

\begin{theorem}[{\cite[Theorem 3.5]{ellis2011tutte}}]
\label{thm:Vspanning}
{The $\V$-polynomial can be represented as follows
\[ \V_{(G,\mathsf{s})}(\mathbf{z},{\pmb{\gamma}})=\sum_{A\subseteq E(G)}\mathbf{z}_{\lambda(G,\mathsf{s},A)}\prod_{e\in E(G)}\gamma_e
.\]
}\end{theorem}

\section{On the $\st$-expansion of the chromatic symmetric function}
\label{sec:algo}
The main objective of this section is to describe an algorithm for the expansion of the chromatic symmetric function in the $\st$-basis for any graph.  This algorithm is based on a relation that we call \defn{deletion-near-contraction}. Understanding and simplifying the computation of this algorithm is one of the main motivations for the introduction of marked graphs and their associated polynomials in the following sections.

The main reason that the chromatic symmetric function does not satisfy the classical deletion-contraction relation is that when we contract an edge the resulting graph has one fewer vertex; therefore the resulting difference of the chromatic symmetric functions is no longer homogeneous. In contrast, when contracting an edge in a weighted graph the total weight is preserved (since the weights of the endpoints of the contracted edge are added), and thus the weighted chromatic symmetric functions of the original and the contracted graph are homogeneous of the same degree. 

We now introduce a new relation that also preserves the degree for unweighted graphs, i.e., graphs with all vertex weights equal to 1, and we will show that the  chromatic symmetric function satisfies this relation.  We will first define the \defn{near-contraction operation} on weighted graphs:  
  Given a weighted graph $(G,\omega)$ and a non-loop edge $e$ in $G$, we denote by $v'$ the contracted vertex in $(G/e,\omega/e)$ and define  $(G,\omega)\odot e$ to be the weighted graph $(G/e\cup \{v'v''\},\omega')$ where $v''$ is a new vertex, and $\omega'$ is given  by
\[
\omega'(u)=\begin{cases}
\omega/e(v')-1,&\text{if $u=v'$},\\
1,&\text{if $u=v''$},\\
\omega/e(u),&\text{otherwise.}
\end{cases}
\]
In other words, $(G,\omega)\odot e$ is obtained by first contracting $e$ into the contracted vertex $v'$ and then attaching a leaf $v''$ of weight $1$ to $v'$ and decreasing the weight of $v'$ by $1$ to keep the total weight constant. The pendant edge  $\ell_e=v'v''$ is called the \defn{near-contracted edge} of $(G,\omega)\odot e$. We will refer to $(G,\omega)\odot e$ as the \emph{near-contraction} of $e$.

\begin{example} We illustrate the near-contraction operation on a graph, here the $B_i$ are graphs and they do not need to be disjoint from each other. 

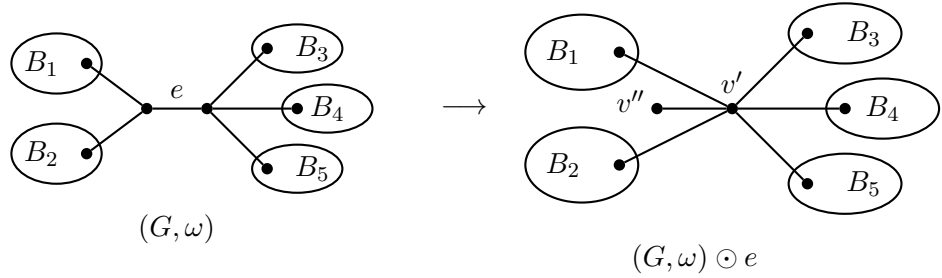
\begin{center}
\begin{tikzpicture}[scale = 0.4,thick, baseline={(0,-1ex/2)}]
\tikzstyle{vertex} = [black, shape = circle, fill=black, minimum size = 3.5pt, inner sep = 1pt]
\node[vertex] (G7) at (8.0, -1.5) [shape = circle, draw] {};
\draw (7,-1.5) ellipse (1.5cm and 1cm);
\node[] at (6.5,1.5) {$B_1$};
\node[vertex] (G6) at (10.0, 0) [shape = circle, draw] {};
\node[vertex] (G5) at (8.0, 1.5) [shape = circle, draw] {};
\draw (7,1.5) ellipse (1.5cm and 1cm);
\node[] at (6.5,-1.5) {$B_2$};
\node[vertex] (G4) at (12.0, 0) [shape = circle, draw] {};
\node[vertex] (G3) at (14.0, 2) [shape = circle, draw] {};
\draw (15,2) ellipse (1.5cm and .8cm);
\node[] at (15.5,-2) {$B_5$};
\node[vertex] (G2) at (14.0, -2) [shape = circle, draw] {};
\draw (15,-2) ellipse (1.5cm and .8cm);
\node[] at (15.5,2) {$B_3$};
\node[vertex] (G1) at (15.0, 0) [shape = circle, draw] {};
\draw (16,0) ellipse (1.5cm and .8cm);
\node[] at (16,0) {$B_4$};
\draw[thick] (G7) -- (G6);
\draw[thick] (G5)  -- (G6);
\draw[thick] (G4)  -- (G6) node[midway,above] {$e$};
\draw[thick] (G4) -- (G2);
\draw[thick] (G3) -- (G4);
\draw[thick] (G1) -- (G4);
\node[] at (11,-4) {$(G,\omega)$};
\end{tikzpicture}\qquad  
\begin{tikzpicture}[scale = 0.5,thick, baseline={(0,-1ex/2)}]
\tikzstyle{vertex} = [black, shape = circle, fill=black, minimum size = 3.5pt, inner sep = 1pt]
\node[] at (5, 0) {$\longrightarrow$ \qquad};
\node[vertex] (G7) at (9.0, -1.5) [shape = circle, draw] {};
\draw (8,-1.5) ellipse (1.5cm and 1cm);
\node[] at (7.5,1.5) {$B_1$};
\node[vertex] (G6) at (10.0, 0) [shape = circle, draw] {};
\node[] at (9.3, .2) {$v''$};
\node[vertex] (G5) at (9.0, 1.5) [shape = circle, draw] {};
\draw (8,1.5) ellipse (1.5cm and 1cm);
\node[] at (7.5,-1.5) {$B_2$};
\node[vertex] (G4) at (12.0, 0) [shape = circle, draw] {};
\node[] at (12,.7) {$v'$};
\node[vertex] (G3) at (14.0, 2) [shape = circle, draw] {};
\draw (15,2) ellipse (1.5cm and .8cm);
\node[] at (15.5,-2) {$B_5$};
\node[vertex] (G2) at (14.0, -2) [shape = circle, draw] {};
\draw (15,-2) ellipse (1.5cm and .8cm);
\node[] at (15.5,2) {$B_3$};
\node[vertex] (G1) at (15.0, 0) [shape = circle, draw] {};
\draw (16,0) ellipse (1.5cm and .8cm);
\node[] at (16,0) {$B_4$};
\draw[thick] (G7) -- (G4);
\draw[thick] (G5)  -- (G4);
\draw[thick] (G4)  -- (G6);
\draw[thick] (G4) -- (G2);
\draw[thick] (G3) -- (G4);
\draw[thick] (G1) -- (G4);
\node[] at (11,-4) {$(G,\omega)\odot e$};
\end{tikzpicture}
\end{center}
\end{example}

\begin{prop}
\label{prop:near_DC}
Let $(G,\omega)$ be a weighted graph and $e$ be a non-loop edge of $G$. Then we have 
\begin{equation}
    \label{eq:lemma_near_contraction}
X_{(G/e,\omega/e)} = X_{((G,\omega)\odot e)\setminus \ell_e} - X_{(G,\omega)\odot e}.
\end{equation}
Moreover, the {weighted} chromatic symmetric function satisfies the \emph{deletion near-contraction} formula:
\begin{equation}
    \label{eq:deletion_near_contraction}
X_{(G,\omega)}=X_{(G\setminus e,\omega)} - X_{((G,\omega)\odot e)\setminus \ell_e} + X_{(G,\omega)\odot e}.
\end{equation}
{Moreover, if $G$ is simple, then 
\[X_{(G,\omega)}=X_{(G\setminus e,\omega)} - X_{((G,\omega)\odot e)^s\setminus \ell_e} + X_{((G,\omega)\odot e)^s}.\]
}
\end{prop}

\begin{proof}
It is easy to see that contracting the near-contracted edge $\ell_e$ in  $(G,\omega)\odot e$  yields $(G/e,\omega/e)$ and thus the first assertion follows directly from the deletion-contraction formula~\eqref{w_delcont} applied to $(G,\omega)\odot e$ and $\ell_e$. Thus, \eqref{eq:deletion_near_contraction} follows directly from substituting \eqref{eq:lemma_near_contraction} into the  deletion-contraction formula now applied to $(G,\omega)$ and $e$.  Note that deleting or near-contracting an edge does not create loops, unless the near-contraction is applied to a parallel edge. Then, when $G$ is simple all graphs that appear on the right-hand side of~\eqref{eq:deletion_near_contraction} are loopless, and the last statement follows since the weighted chromatic function of a loopless weighted graph equals that of its simplification.
\end{proof}

\begin{figure}[hbt!]
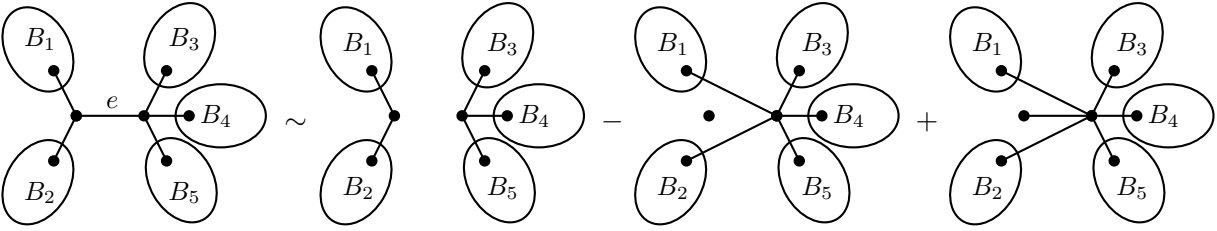

\[
\begin{mygraph}[baseline=-5pt,thick,scale=0.6]
\draw (-0.5,1) node {} -- (0,0) 
               node {} -- (-0.5,-1) node {};
\draw (2,1) node {} -- (1.5,0) 
            node {} -- (2,-1) node {};
\draw (2.5,0) node {} -- (1.5,0);
\draw (0,0) node {} -- (1.5,0);
\draw[rotate=-60] (2.4,1.3) ellipse (1 and 0.7);
\draw[rotate=-60] (-1.7,0) ellipse (1 and 0.7);
\draw[rotate=60] (-1.7,0) ellipse (1 and 0.7);
\draw (3.2,0) ellipse (1 and 0.7);
\draw[rotate=60] (2.5,-1.2) ellipse (1 and 0.7);
\node[fill=none,draw=none] at (-0.8,1.7)  {$B_1$};
\node[fill=none,draw=none] at (-0.8,-1.7)  {$B_2$};
\node[fill=none,draw=none] at (2.4,1.7)  {$B_3$};
\node[fill=none,draw=none] at (3.1,0)  {$B_4$};
\node[fill=none,draw=none] at (2.4,-1.7)  {$B_5$};
\node[fill=none,draw=none] at (.8,.3)  {$e$};
\end{mygraph}
\ \sim 
\begin{mygraph}[baseline=-5pt,thick,scale=0.6]
\draw (-0.5,1) node {} -- (0,0);
\draw (0,0)  node {} -- (-0.5,-1) node {};
\draw (2,1) node {} -- (1.5,0) 
            node {} -- (2,-1) node {};
\draw (2.5,0) node {} -- (1.5,0);
\draw[rotate=-60] (2.4,1.3) ellipse (1 and 0.7);
\draw[rotate=-60] (-1.7,0) ellipse (1 and 0.7);
\draw[rotate=60] (-1.7,0) ellipse (1 and 0.7);
\draw (3.2,0) ellipse (1 and 0.7);
\draw[rotate=60] (2.5,-1.2) ellipse (1 and 0.7);
\node[fill=none,draw=none] at (-0.8,1.6)  {$B_1$};
\node[fill=none,draw=none] at (-0.8,-1.6)  {$B_2$};
\node[fill=none,draw=none] at (2.4,1.6)  {$B_3$};
\node[fill=none,draw=none] at (3.1,0)  {$B_4$};
\node[fill=none,draw=none] at (2.4,-1.6)  {$B_5$};
\end{mygraph}
\ -
\begin{mygraph}[baseline=-5pt,thick,scale=0.6]
\draw (0,0) node {};
\draw (-0.5,1) node {} -- (1.5,0) 
               node {} -- (-0.5,-1) node {};
\draw (2,1) node {} -- (1.5,0) 
            node {} -- (2,-1) node {};
\draw (2.5,0) node {} -- (1.5,0);
\draw[rotate=-60] (2.4,1.3) ellipse (1 and 0.7);
\draw[rotate=-60] (-1.7,0) ellipse (1 and 0.7);
\draw[rotate=60] (-1.7,0) ellipse (1 and 0.7);
\draw (3.2,0) ellipse (1 and 0.7);
\draw[rotate=60] (2.5,-1.2) ellipse (1 and 0.7);
\node[fill=none,draw=none] at (-0.8,1.6)  {$B_1$};
\node[fill=none,draw=none] at (-0.8,-1.6)  {$B_2$};
\node[fill=none,draw=none] at (2.4,1.6)  {$B_3$};
\node[fill=none,draw=none] at (3.1,0)  {$B_4$};
\node[fill=none,draw=none] at (2.4,-1.6)  {$B_5$};
\end{mygraph}
\ +
\begin{mygraph}[baseline=-5pt,thick,scale=0.6]
\draw (-0.5,1) node {} -- (1.5,0);
\draw (1.5,0) node {}  -- (-0.5,-1) node {};
\draw (2,1) node {} -- (1.5,0) 
            node {} -- (2,-1) node {};
\draw (2.5,0) node {} -- (1.5,0);
\draw (0,0) node {} -- (1.5,0);
\draw[rotate=-60] (2.4,1.3) ellipse (1 and 0.7);
\draw[rotate=-60] (-1.7,0) ellipse (1 and 0.7);
\draw[rotate=60] (-1.7,0) ellipse (1 and 0.7);
\draw (3.2,0) ellipse (1 and 0.7);
\draw[rotate=60] (2.5,-1.2) ellipse (1 and 0.7);
\node[fill=none,draw=none] at (-0.8,1.6)  {$B_1$};
\node[fill=none,draw=none] at (-0.8,-1.6)  {$B_2$};
\node[fill=none,draw=none] at (2.4,1.6)  {$B_3$};
\node[fill=none,draw=none] at (3.1,0)  {$B_4$};
\node[fill=none,draw=none] at (2.4,-1.6)  {$B_5$};
\end{mygraph}
\]
\caption{The deletion-near-contraction relation applied to edge $e$.}
\label{fig:moves3}
\end{figure}
\begin{remark}
Observe that if $(G,\omega)$ is unweighted, i.e., all the weights are equal to $1$, then $(G\setminus e,\omega)$, $(G,\omega)\odot e$ and $((G,\omega)\odot e)\setminus \ell_e$ are also unweighted. Thus, Proposition~\ref{prop:near_DC} holds as well for the (unweighted) chromatic symmetric function.
\end{remark}

Observe that if $e$ is a pendant edge of $G$ and the weight of the associated leaf 
is $1$, then $(G,\omega)\odot e$ is $\omega$-isomorphic to $(G,\omega)$ by an isomorphism sending $e$ onto $\ell_e$. In this case 
Equation \eqref{eq:deletion_near_contraction} does not give any information. Thus, it only 
makes sense to apply Proposition~\ref{prop:near_DC} to edges that are either 
internal or have an incident leaf of weight larger than one. Moreover, it is easy to see that 
deleting or near-contracting internal edges decreases the number of internal edges in a graph. Thus, after applying Proposition~\ref{prop:near_DC} recursively only to internal edges the procedure must always stop. 

Let $G$ be a simple graph. Recall that in this case \defn{internal} edges are those with both endpoints having degree greater than 1.
Notice that if $G$ has no internal
edges, then all the components of $G$ are star graphs since $G$ is simple.  Thus, simple graphs with no internal edges are \emph{star forests}. The following algorithm recursively builds a deletion-near-contraction tree to compute the expansion of the chromatic symmetric function of $G$ in the star-basis.

\vskip .1in
\noindent {\bf Algorithm: } Star-Expansion 

\noindent{\bf Input:} A  simple graph $G$.

\noindent{\bf Idea: } Construct a rooted tree with root labeled $G$, each new vertex is obtained by applying the deletion-near-contraction relation to the parent vertex if the parent graph has at least one internal edge.  The edges of the tree are labeled by $+$ or $-$.

\noindent{\bf Initialization: } Let $\mathcal{T}$ be a tree with root labeled $G$ and no edges. 

\noindent{\bf Iteration: } If $H$ is the label of a leaf vertex (has no children) 
of $\mathcal{T}$ and $H$ has some internal edge $e$, then extend $\mathcal{T}$ by adding three child vertices to $H$ labeled $H \setminus e$, {$(H\odot e)^s$ and $(H\odot e)^s\setminus \ell_e$}, and label the three new edges with $+$ or $-$ according to the corresponding coefficient in the deletion-near-contraction relation.  The algorithm terminates when no leaf of $\mathcal{T}$ contains internal edges. 

\noindent{\bf Output: } A rooted tree $\mathcal{T}(G )$ with leaves labeled by star forests. 
\vskip .1in 

Before we state our next theorem we introduce some notation.  Given any graph $G$, let $\iota(G)$ denote the number of isolated vertices of $G$, and for a star forest $H$, let $\lambda(H)$ be the partition whose parts are the orders of the connected components of $H$. 

\begin{theorem}
\label{thm:stfree}
For every simple graph $G$, let $\mathcal{T}(G)$ be an output of the Star-Expansion Algorithm and $L(\mathcal{T}(G))$ be the multiset of leaf labels of $\mathcal{T}(G)$. Then 
$$ X_G = \sum_{H\in L(\mathcal{T}(G))} (-1)^{\iota(H)-\iota(G)} \st_{\lambda(H)}.$$

In addition, no cancellations occur in the computation; that is, for any partition $\lambda$, all terms $\st_{\lambda}$ appear with the same sign.
\end{theorem}

\begin{proof}
The proof is by induction on the number of internal edges of $G$. If $G$ has no internal edges, then $G$ is a star forest, and hence the algorithm stops after the initialization step and outputs a tree with its only vertex labeled by $G$. Of course in this case $X_G=\st_{\lambda(G)}$ and thus the assertion holds. 
For the inductive step, suppose that  $G$ is a  simple graph with $n+1$ internal edges. Let $e$ be an internal edge of $G$ and apply the first iteration of the Star-Expansion algorithm to $G$ and $e$. On one hand, this yields three children, labeled by $G\setminus e$, $(G\odot e)^s$ and $(G\odot e)^s\setminus \ell_e$, with corresponding signs $+$,$+$ and $-$ in their edges. The multiset of leaf labels $L(\mathcal{T}(G))$ corresponds to the union of the multisets $L(\mathcal{T}({G\setminus e})$, $L(\mathcal{T}({(G\odot e)^s}))$ and $L(\mathcal{T}({(G\odot e)^s\setminus \ell_e}))$. 
On the other hand, by deletion-near-contraction, 
we get 
\[ X_G = X_{G\setminus e} + X_{(G\odot e)^s} - X_{{(G\odot e)}^s \setminus \ell_e}.\]
 Since these graphs have at most $n$ internal edges, by the induction hypothesis, we get
 that 

 \begin{align*}
     X_G  = &
     \sum_{H\in L(\mathcal{T}(G\setminus e))} (-1)^{\iota(H)-\iota(G\setminus e)}\st_{\lambda(H)}+
         \sum_{H\in L(\mathcal{T}(G\odot e)^s)} (-1)^{\iota(H)-\iota({(G\odot e)}^s)}\st_{\lambda(H)}
         \\
         & - \sum_{H\in L(\mathcal{T}(G\odot e)^s\setminus \ell_e)} (-1)^{\iota(H)-\iota((G\odot e)^s\setminus \ell_e)}\st_{\lambda(H)}.
         \end{align*}

Since $e$ is an internal edge, we have 
\[\iota(G)=\iota(G\setminus e)=\iota({(G\odot e)}^s)=\iota((G\odot e)^s\setminus \ell_e)-1.  \]
This gives
\[X_G = \sum_{H\in L(\mathcal{T}(G))} (-1)^{\iota(H)-\iota(G)} \st_{\lambda(H)},\]
which finishes the induction.  

Note that two star forests $H_1, H_2$ are isomorphic if and only if $\lambda(H_1)= \lambda(H_2)$. Thus, for a given partition $\lambda$, all terms $\st_{\lambda}$ arise from isomorphic star forests, so they all have the same sign (as the sign depends only on the isomorphism type of $G$ and $H$).
\end{proof}

\begin{remark}
The proof of Theorem~\ref{thm:stfree} also shows that for a leaf of $\mathcal{T}(G)$ labeled $H$, the difference $\iota(H)-\iota(G)$ is the number of edges in the path in  $\mathcal{T}(G)$ between the root and the leaf that are labeled $-$. Thus the coefficient of the corresponding term $\st_{\lambda(H)}$ in $X_G$ can be read directly from $\mathcal{T}(G)$. \end{remark}

\begin{example} In Figure~\ref{fig:star-expansion}, we show an example of how the algorithm computes the star-expansion of a graph $G$. This says that for the graph at the top we have $X_G = 2\st_4-2\st_{3,1}+\st_{2,2}$.
\begin{figure}[hbt!]
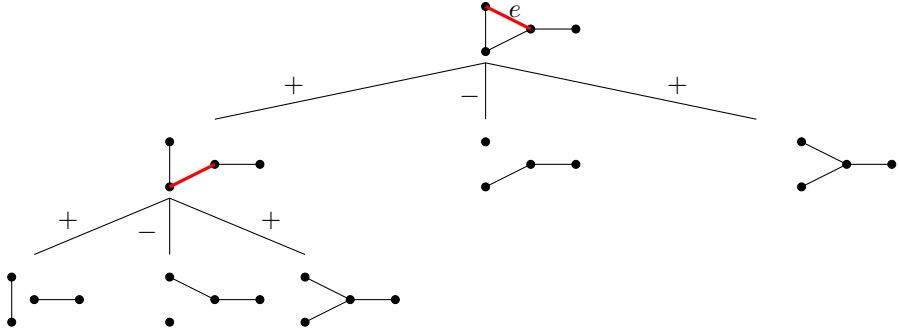

    \centering
    \begin{mygraph2}[scale=0.3]
    \draw (0,1) -- (0,-1) node {}
                -- (2,0) node {}
                -- (0,1) node {};
    \draw (2,0) -- (4,0) node {};
    \draw[very thick, red] (0,1) -- (2,0);
    \node[fill=none,draw=none] at (1.3,.8)  {$e$};
    \node[fill=none,draw=none] at (-8.5,-2.5)  {$+$};
    \node[fill=none,draw=none] at (-.7,-3)  {$-$};
    \node[fill=none,draw=none] at (8.5,-2.5)  {$+$};
    \draw (-12,-4)--(0,-1.5)--(12,-4);
    \draw (0,-1.5)--(0,-4);
    \pgftransformshift{\pgfpoint{-14cm}{-6cm}}
    \draw (0,1) -- (0,-1) node {}
                -- (2,0) node {}
                 (0,1) node {};
    \draw[very thick, red] (0,-1) -- (2,0);
    \draw (2,0) -- (4,0) node {};
    \node[fill=none,draw=none] at (-4.5,-2.5)  {$+$};
    \node[fill=none,draw=none] at (-1,-3)  {$-$};
    \node[fill=none,draw=none] at (4.5,-2.5)  {$+$};
    \draw (-6,-4)--(0,-1.5)--(6,-4);
    \draw (0,-1.5)--(0,-4);
    \pgftransformshift{\pgfpoint{14cm}{0cm}}
   \draw (2,0) node {} to  (0,-1) node {};
    \draw (2,0) -- (4,0) node {};
    \draw (0,1) node {};
    \pgftransformshift{\pgfpoint{14cm}{0cm}}
    \draw (2,0) node {} to (0,-1) node {};
    \draw (2,0)  -- (0,1) node {};
    \draw (2,0) -- (4,0) node {};
    \pgftransformshift{\pgfpoint{-34cm}{-6cm}}
    \draw (-1,1) -- (-1,-1) node {}
                 (0,0) node {}
                 (-1,1) node {};
   
    \draw (0,0) -- (2,0) node {};
    \pgftransformshift{\pgfpoint{6cm}{0cm}}
    \draw (0,1) -- (2,0) node {}
                 (0,-1) node {}
                 (0,1) node {};
    \draw (2,0) -- (4,0) node {};
    \pgftransformshift{\pgfpoint{6cm}{0cm}}
    \draw (0,1) -- (2,0) node {}
                -- (0,-1) node {}
                 (0,1) node {};
    \draw (2,0) -- (4,0) node {};
    \end{mygraph2}
    \caption{The tree $\mathcal{T}(G)$}
    \label{fig:star-expansion}
\end{figure}

\end{example}

We observe that a weighted version of the algorithm Star-Expansion can be defined in a completely analogous manner. In that case, for an input graph $(G, \omega)$ we obtain a tree with leaves labeled by weighted star forests. Then, the analogue of Theorem~\ref{thm:stfree} would express $X_{(G,\omega)}$ in terms of the weighted symmetric functions of weighted star forests, but in this case weighted stars do not lead to a basis. We omit the straightforward details. 

\section{Marked graphs and the $M$-polynomial}\label{sec:Mpoly}
In this section define the $M$-polynomial and introduce its properties.
Let $\NN$ be the set of non-negative integers. A \emph{mark} is a pair $(w,d)$ where $w\in\PP$, $d\in \NN$ and $w\geq d+1$. 
We say that $w$ is the \emph{weight} of the mark and $d$ is its \emph{number of dots}. 

The set of all marks 
{\[\MM = \{(w,d)\mid w\in \PP, d\in \NN,w\geq d+1\}\]}
is endowed with the \defn{dot-sum} operation, denoted  $\dotplus$, defined by 
\[ (w,d)\dotplus (w',d') := (w+w',d+d'+1).\]
It is easy to check that $(\MM,\dotplus)$ is a semigroup. 
 Following Section~\ref{sec:Vpoly}, we will consider weighted graphs (that is, with weights in $\PP$) and  $\MM$-weighted graphs.
In order to avoid confusions we will refer to the latter as \defn{marked graphs}. 
That is, a \defn{marked graph} is a pair $(G,\mathsf{m})$ where $G$ is a graph and $\mathsf{m}:V(G)\rightarrow \MM$ is a function.
Two marked graphs $(G,\mathsf{m})$ and $(G',\mathsf{m}')$ are \defn{mark-isomorphic} if there is an isomorphism $\phi:G\rightarrow G'$ 
such that $\mathsf{m}(v)=\mathsf{m}'(\phi(v))$ for all $v\in V(G)$. 

Let $(G,\mathsf{m})$ be a marked graph. The deletion and contraction operations have already been defined in Section~\ref{sec:Vpoly}. Now we introduce the $M$-polynomial.

\begin{definition}
Let $y$ and $\mathbf{z}=\{z_{w,d}\mid (w,d)\in \MM\}$ be commuting indeterminates.
The \defn{marked graph polynomial} or \defn{$M$-polynomial} of a marked graph $(G,\mathsf{m})$ is defined by the following rules: 
\begin{enumerate}[a)]
\item If $G$ consists only of isolated vertices with corresponding marks $(w_1,d_1)$,$(w_2,d_2)$,$\ldots$,$(w_k,d_k)$, 
then 
\[M_{(G,\mathsf{m})}(\mathbf z,y) = z_{w_1,d_1}z_{w_2,d_2}\cdots z_{w_k,d_k}.\]
\item If $e$ is a loop of $G$, then 
\[ M_{(G,\mathsf{m})}(\mathbf z,y) = y M_{(G\setminus e,\mathsf{m})}(\mathbf z,y).\]
\item If $e$ is a non-loop edge of $G$, then 
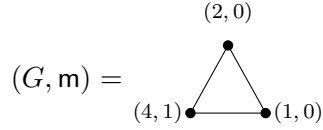
\begin{equation*}
M_{(G,\mathsf{m})}(\mathbf z,y) = M_{(G\setminus e,\mathsf{m})}(\mathbf z,y) +  M_{(G/e,\mathsf{m}/e)}(\mathbf z,y).
\end{equation*}
\end{enumerate}
\end{definition}
\begin{example}\label{ex:M-poly}
\label{ex:triangle}
Suppose
$$(G,\mathsf{m}) =\begin{mygraph}[baseline=10pt, every label/.append style={font=\scriptsize}]
\draw (0,0) node[label=left:${(4,1)}$] {} -- (1,0)
            node[label=right:${(1,0)}$] {} -- (0.5,0.9)
            node[label=above:${(2,0)}$] {} -- (0,0);
\draw (-0.05,-0.25) node[draw=none,fill=none] {};
\draw (0.5,-0.2) node[draw=none,fill=none] {};
\end{mygraph}
$$
then $M_{(G,\mathsf{m})} = z_{2,0}z_{4,1}z_{1,0}+z_{4,1}z_{3,1}+z_{1,0}z_{6,2}+z_{5,2}z_{2,0}+z_{7,3}(2+y)$ {(see Figure~\ref{fig:Mtriangle}).}

\begin{figure}[hbt!]
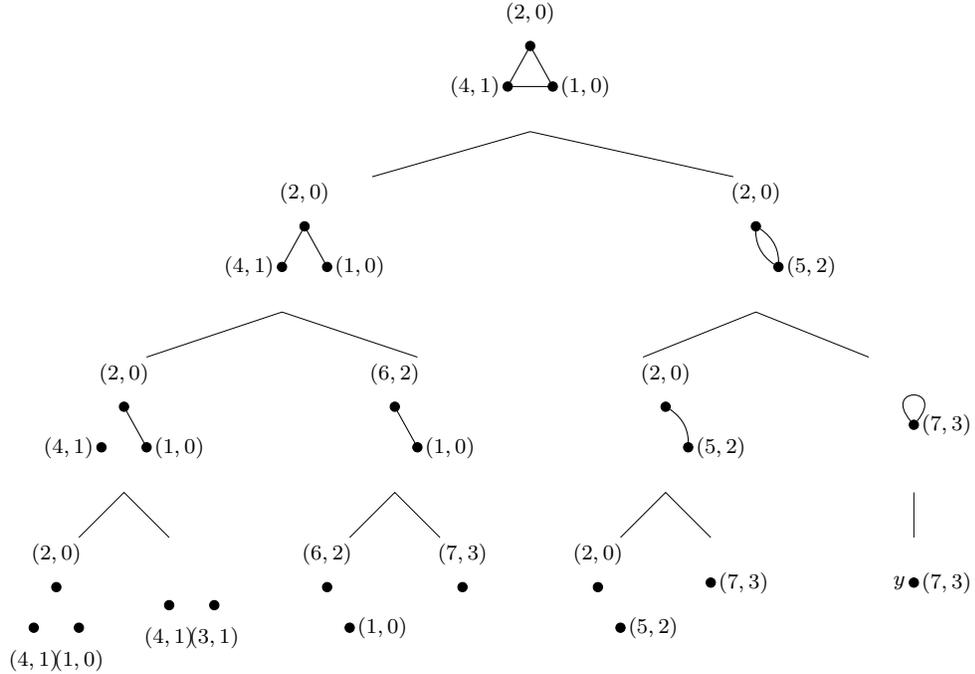

\begin{mygraph}[scale=0.6,baseline=10pt,every label/.append style={font=\scriptsize}]
\draw (0,0) node[label=left:${(4,1)}$] {} -- (1,0)
            node[label=right:${(1,0)}$] {} -- (0.5,0.9)
            node[label=above:${(2,0)}$] {} -- (0,0);
\draw[very thin] (0.5,-1) -- (5,-2)
                 (0.5,-1) -- (-3,-2)
                 (5.5,-5) -- (3,-6)
                 (5.5,-5) -- (8,-6)
                 (-5,-5) -- (-8,-6)
                 (-5,-5) -- (-2,-6)
                 (-7.5,-10)-- (-8.5,-9) -- (-9.5,-10)
                 (-3.5,-10)--(-2.5,-9)--(-1.5,-10)
                 (4.5,-10)--(3.5,-9)--(2.5,-10)
                 (9,-9)--(9,-10);
\begin{scope}[xshift=-5cm,yshift=-4cm]
\draw (0,0) node[label=left:${(4,1)}$] {}  (1,0)
            node[label=right:${(1,0)}$]{} -- (0.5,0.9)
            node[label=above:${(2,0)}$] {} --  (0,0);
\end{scope}
\begin{scope}[xshift=5cm,yshift=-4cm]
\draw (1,0) node[label=right:${(5,2)}$] {} edge [bend left] (0.5,0.9);
\draw (0.5,0.9) node[label=above:${(2,0)}$] {} edge [bend left]  (1,0);
\end{scope}
\begin{scope}[xshift=-9cm,yshift=-8cm]
\draw (0,0) node[label=left:${(4,1)}$] {}  (1,0)
            node[label=right:${(1,0)}$]{} -- (0.5,0.9)
            node[label=above:${(2,0)}$] {}  (0,0);
\end{scope}
\begin{scope}[xshift=-3cm,yshift=-8cm]
\draw 
(1,0) node[label=right:${(1,0)}$]{} -- (0.5,0.9)
      node[label=above:${(6,2)}$] {};
\end{scope}
\begin{scope}[xshift=3cm,yshift=-8cm]
\draw (1,0) node[label=right:${(5,2)}$] {};
\draw (0.5,0.9) node[label=above:${(2,0)}$] {} edge [bend left]  (1,0);
\end{scope}
\begin{scope}[yshift=-8cm,xshift=9cm]
\draw (0,0.5) node[label=right:${(7,3)}$] {} edge [in=50,out=130,loop] (0.5,0.9);
\end{scope}
\begin{scope}[xshift=-10.5cm,yshift=-12cm]
\draw (0,0) node[label=below:${(4,1)}$] {}  (1,0)
            node[label=below:${(1,0)}$]{}  (0.5,0.9)
            node[label=above:${(2,0)}$] {}  (0,0);
\end{scope}
\begin{scope}[xshift=-7.5cm,yshift=-12cm]
\draw (0,0.5) node[label=below:${(4,1)}$] {}  (1,0.5)
            node[label=below:${(3,1)}$] {}  (0,0);
\end{scope}
\begin{scope}[xshift=-4.5cm,yshift=-12cm]
\draw 
(1,0) node[label=right:${(1,0)}$]{} (0.5,0.9)
      node[label=above:${(6,2)}$] {};
\end{scope}
\begin{scope}[xshift=-1.5cm,yshift=-12cm]
\draw 
(0.5,0.9)  node[label=above:${(7,3)}$] {};
\end{scope}
\begin{scope}[xshift=1.5cm,yshift=-12cm]
\draw (1,0) node[label=right:${(5,2)}$] {};
\draw (0.5,0.9) node[label=above:${(2,0)}$] {}  (1,0);
\end{scope}
\begin{scope}[xshift=4.5cm,yshift=-11cm]
\draw (0,0) node[label=right:${(7,3)}$] {};
\end{scope}
\begin{scope}[xshift=9cm,yshift=-11cm]
\draw (0,0) node[label=left:$y$,label=right:${(7,3)}$] {};
\end{scope}
\end{mygraph}
\caption{Computation of the $M$-polynomial with the definition (Example~\ref{ex:M-poly}).}
\label{fig:Mtriangle}
\end{figure}
\end{example}

Observe that the  $M$-polynomial is a generalization of the $W$-polynomial to marked graphs. It is direct to see that is a mark-isomorphism invariant. It is also a specialization of the $\V$-polynomial described in Section~\ref{sec:Vpoly} (see Proposition~\ref{prop:smr}). Thus, many of the results known about the $W$-polynomial can be extended to the $M$-polynomial. 

Given a marked graph $(G,\mathsf{m})$ and a subset $U$ of vertices of $G$, the \defn{total mark} 
of $U$, denoted $\mathsf{m}(U)$, is the dot-sum of the marks of all  vertices in $U$. 
Given a mark $(w,d)\in\MM$, an $\MM$-partition of $(w,d)$ is a multiset 
$(w_1,d_1),\ldots,(w_l,d_l)$ such that $\sum_{i=1}^{l}w_i=w$ and $\sum_{i=1}^{l}d_i=d - l + 1$. The following proposition is the analog of Theorem~\ref{teoUtoX} for the $M$-polynomial.

\begin{prop}
\label{prop:smr}
Let $(G,\mathsf{m})$ be a marked graph. Then, the $M$-polynomial is  well-defined.
Moreover, it has the following \emph{states model representation}
\begin{equation}
\label{Mformula1}
M_{(G,\mathsf{m})}({\mathbf{z}},y) = \sum_{A\subseteq E(G)}
\mathbf{z}_{\lambda(G,\mathsf{m},A)}(y-1)^{|A|-r_G(A)},
\end{equation}
{where $r_G(A)$ is the rank of $A$ defined in Equation  \eqref{rankDef}} and $\lambda(G,\mathsf{m},A)$ is the $\MM$-partition induced by the total marks of the connected components of the spanning subgraph $G|_A$.
\end{prop}

\begin{proof}
Consider $f(G;\mathbf{z},y)={(y-1)^{|V(G)|}}M_{(G, \mathsf{m})}(\frac{\mathbf{z}}{y-1},y)$ where $\frac{\mathbf{z}}{y-1}$ means that we replace each variable $z_{i,k}$ by $\frac{z_{i,k}}{y-1}$. By the deletion-contraction applied to $M$, we may apply Theorem~\ref{recipeV} with $\alpha_e=1$, $\beta_e=y-1$. This ensures that $M$ is well-defined and satisfies
\[ \V_G(\mathbf{z},\gamma_e=(y-1)) = (y-1)^{|V(G)|} M_{(G,\mathsf{m})}\left(\frac{\mathbf{z}}{y-1},y\right).\]
By a change of variables we get
\[ M_{(G,\mathsf{m})}(\mathbf{z},y)=(y-1)^{-|V(G)|}\V_G((y-1)\mathbf{z},\gamma_e=(y-1)).\]
Then, combining the last equation with Theorem~\ref{thm:Vspanning}  yields 
\[M_{(G,\mathsf{m})}(\mathbf{z},y)=(y-1)^{-|V(G)|}\sum_{A\subseteq E(G)} (y-1)^{k(A)}\mathbf{z}_{\lambda(G,\mathsf{m},A)}(y-1)^{|A|},
\]
from which Equation \eqref{Mformula1} follows and thus this finishes the proof.
\end{proof}

The states model representation of the $M$-polynomial in Proposition ~\ref{prop:smr} implies that it is \emph{multiplicative} over connected components, that is, if $(G_1,\mathsf{m}_1)$ and $(G_2,\mathsf{m}_2)$ are two marked graphs, then \[ M_{(G_1,\mathsf{m}_1)\sqcup(G_2,\mathsf{m}_2)}(\mathbf{z},y)=M_{(G_1,\mathsf{m}_1)}(\mathbf{z},y)M_{(G_2,\mathsf{m}_2)}(\mathbf{z},y).\]

The $M$-polynomial has another states model representation analogous to 
\cite[Theorem 4.7]{noble99weighted} that is given in terms of the bond 
lattice of $G$ and the Tutte polynomial of $G$. Given a graph $G$, a set 
partition $\pi=\{V_1,V_2,\ldots,V_k\}$ of $V(G)$ is \emph{connected} if 
each induced subgraph $G[V_i]$ is connected. The \defn{bond lattice} $L_G$ 
of $G$ is the set of all connected partitions of $G$, partially ordered by 
refinement. It is well-known that $L_G$ is a geometric lattice with the 
rank of each $\pi$ being {given by $|V(G)|-|\pi|$, where 
$|\pi|$ denotes the number of blocks of $\pi$, see  \cite[Theorem 1.1 and 
2.6]{stanley95symmetric} and also \cite{thatte2020connected}}.
The \emph{Tutte polynomial} of $G$ is the  polynomial defined by
\[T_G(x,y)=\sum_{A\subseteq E}(x-1)^{r(E)-r_G(A)}(y-1)^{|A|-r_G(A)}.\]

\begin{prop}
\label{prop:42}
The $M$-polynomial of $(G,\mathsf{m})$ admits the following representation: 
\[M_{(G,\mathsf{m})}(\mathbf{z},y) = \sum_{\pi\in L_G}{z}_{\mathsf{m}(V_1)}{z}_{\mathsf{m}(V_2)}\cdots {z}_{\mathsf{m}(V_k)} T_{G_1}(1,y)T_{G_2}(1,y)\cdots T_{G_k}(1,y),
\]
where $\pi=\{V_1,V_2,\ldots,V_k\}$ and $G_i=G[V_i]$ for each $i$ in $\{1,\ldots,k\}$.
\end{prop}
\begin{proof}
The proof is analogous to the proof of \cite[Theorem 4.7]{noble99weighted} since the $M$-polynomial has a similar definition as the $W$-polynomial but with marks instead of weights.
\end{proof}
{Note that if $G$ is a forest, then the representations given by Proposition~\ref{prop:smr} and Proposition~\ref{prop:42} agree. This follows from the fact that for forests,  $L_G$ is isomorphic to the lattice of subsets of edges and that the Tutte polynomial evaluated at $(1,y)$ equals $1$.}

If $(G,\mathsf{m})$ is a marked graph, then we let 
$\omega_{\mathsf{m}}:V(G)\rightarrow \PP$ be the weight function obtained by \emph{forgetting the dots} in $\mathsf{m}$. This means that for every $v\in V(G)$ we set \defn{$\omega_\mathsf{m}(v):=w$} where $\mathsf{m}(v)=(w,d)$. 
Conversely, we regard a weighted graph $(G,\omega)$ as a marked graph $(G,\mathsf{m}_\omega)$ by setting \defn{$\mathsf{m}_\omega(v):=(\omega(v),0)$} for each vertex $v$.  Similarly, a graph $G$ can be seen as a marked graph $(G, \mathbbm{1})$ where $\mathbbm{1}(v) = (1,0)$ for each vertex $v$.

Our next result, which follows from Proposition~\ref{prop:smr} and Theorem~\ref{teoUtoX},  shows that we can recover the $W$-polynomial from the $M$-polynomial.
\begin{prop}
\label{prop:WfromM}
Let $(G,\omega)$ be a weighted graph. Then
\[ W_{(G,\omega)}(z_w,y)=M_{(G,\mathsf{m}_\omega)}(z_{w,k}=z_w,y),\]
where $z_{w,k}=z_w$ means that we substitute each variable $z_{w,k}$ by $z_w$.
\end{prop}

\begin{example}
{
\label{ex:M-poly2}
Consider the following weighted graph $(G,\omega)$: 
\begin{mygraph}[baseline=-20pt,thick,scale=0.6]
\draw (0,-1) node[label=below:$4$] {} --   (1.2,-1)
             node[label=below:$1$] {} -- (2.5,-1) 
             node[label=below:$2$] {};
\end{mygraph}. Then 
\[
(G,\mathsf{m}_{ \omega}) = 
\begin{mygraph}[baseline=-20pt,thick,scale=0.6]
\draw (0,-1) node[label=below:${(4,0)}$] {} --   (1.3,-1)
             node[label=below:${(1,0)}$] {} -- (2.6,-1) 
             node[label=below:${(2,0)}$] {};
\end{mygraph}.
\]}

Proposition~\ref{prop:smr} allows us to compute the $M$-polynomial of $(G,\mathsf{m}_{\omega})$. Indeed, the Boolean lattice of $E(G)$ is
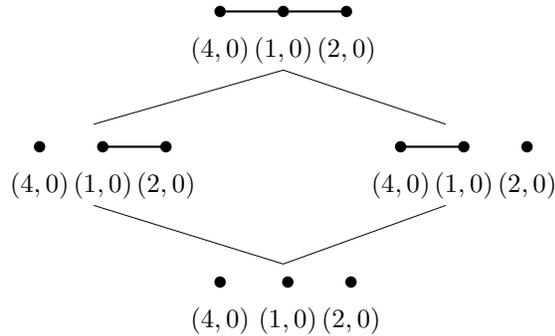
\begin{center}
\begin{mygraph}[baseline=-20pt,thick,scale=0.6]
\newcommand{\xx}{0}\newcommand{\yy}{0}
\draw[ultra thin]  (5,0.7) node[fill=none,draw=none]{} -- 
       (1.4,-0.6) node[fill=none,draw=none]{} -- (-4+1.2,0.7);
\draw[ultra thin]  (5,2.5) node[fill=none,draw=none]{} -- 
       (1.4,3.7) node[fill=none,draw=none]{} -- (-4+1.2,2.5);
\draw (0+\xx,-1+\yy) node[label=below:${(4,0)}$] {}  (1.5+\xx,-1+\yy)
                   node[label=below:${(1,0)}$] {}  (2.9+\xx,-1+\yy) 
                   node[label=below:${(2,0)}$] {};
\renewcommand{\xx}{-4}\renewcommand{\yy}{3}
\draw (0+\xx,-1+\yy) node[label=below:${(4,0)}$] {}  (1.4+\xx,-1+\yy)
                   node[label=below:${(1,0)}$] {} -- (2.8+\xx,-1+\yy) 
                   node[label=below:${(2,0)}$] {};
\renewcommand{\xx}{4}\renewcommand{\yy}{3}
\draw (0+\xx,-1+\yy) node[label=below:${(4,0)}$] {} -- (1.4+\xx,-1+\yy)
                   node[label=below:${(1,0)}$] {} (2.8+\xx,-1+\yy) 
                   node[label=below:${(2,0)}$] {};
\renewcommand{\xx}{0}\renewcommand{\yy}{6}                   
\draw (0+\xx,-1+\yy) node[label=below:${(4,0)}$] {} -- (1.4+\xx,-1+\yy)
                   node[label=below:${(1,0)}$] {} -- (2.8+\xx,-1+\yy) 
                   node[label=below:${(2,0)}$] {};
\end{mygraph}
\end{center}
Thus, applying \eqref{Mformula1} yields
\[M_{(G,\mathsf{m}_{\omega})}(\mathbf{z},y)=z_{4,0}z_{1,0}z_{2,0}+z_{4,0}z_{3,1}+z_{5,1}z_{2,0}+z_{7,2}.
\]
Finally note that Proposition~\ref{prop:WfromM}  applied to this example yields
\[ W_{(G,\omega)}(\mathbf{z},y) = z_4z_1z_2+z_4z_3+z_5z_2+z_7.\]

\end{example}

For the next result, we extend the near-contraction operation to marked graphs. Let $(G,\mathsf{m})$ be a marked graph and $e$ be a non-loop edge of $G$. Let $u,v$ be the endpoints of $e$. Then $(G,\mathsf{m})\odot e$ is the graph obtained by first contracting $e$ into the contracted vertex $v_e$, and then attaching a leaf $v'$. The mark of $v'$ is $(1,0)$, the mark of $v_e$ is $(w_u+w_v-1,d_u+d_v)$ where $(w_u,d_u)$ and $(w_v,d_v)$ are the marks of $u$ and $v$ respectively. The marks of the other vertices remain unchanged. The new pendant edge $v'v_e$ is denoted $\ell_e$.
\begin{prop}
\label{prop:Mnearcontract}
For every marked graph $(G,\mathsf{m})$ and $e$ a non-loop edge of $G$, we have
\[
M_{(G/e,\mathsf{m}/e)}(\mathbf{z},y) = M_{(G,\mathsf{m})\odot e}(\mathbf{z},y) - M_{((G,\mathsf{m})\odot e)\setminus \ell_e}(\mathbf{z},y).\]
Moreover, the $M$-polynomial satisfies the deletion-near-contraction formula
\[  M_{(G,\mathsf{m})}(\mathbf{z},y) = M_{(G\setminus e,\mathsf{m})}(\mathbf{z},y)- M_{(G,\mathsf{m})\odot e\setminus l_e}(\mathbf{z},y)+M_{(G,\mathsf{m})\odot e}(\mathbf{z},y).\]
\end{prop}
\begin{proof}
{Since by definition the $M$-polynomial satisfies the deletion-contraction relation, this proof is the same as that for Proposition~\ref{prop:near_DC} and thus we omit it here. }
\end{proof}
\begin{remark}\label{rem:My0}
From the example in Figure~\ref{fig:Mtriangle} we can 
observe that even if $(G,\mathsf{m})$ is simple, it is not true that 
\[  M_{(G,\mathsf{m})}(\mathbf{z},y) = M_{(G\setminus e,\mathsf{m})}(\mathbf{z},y)- M_{(G,\mathsf{m})^s\odot e\setminus l_e}(\mathbf{z},y)+M_{(G,\mathsf{m})^s\odot e}(\mathbf{z},y)~,\]
because if we removed repeated edges and loops, we would not get all the terms on the right-hand side of the tree in Figure~\ref{fig:Mtriangle}.
However, it follows from the definition that if we set $y=0$ the $M$-polynomial does not distinguish between a graph and its simplification. That is, for any marked graph $(G,\mathsf{m})$,
\[M_{(G,\mathsf{m})}(\mathbf{z},y=0)=M_{(G^s,\mathsf{m})}(\mathbf{z},y=0).\]
Thus, if $y=0$, we have a deletion near-contraction formula for simple graphs.
\end{remark}

\section{The $D$-polynomial}
\label{sec:Dpoly}
In this section we introduce a specialization of the $M$-polynomial that we call the $D$-polynomial. We will prove properties of this specialization that will allow us to make connections to the star-expansion of the chromatic symmetric function.  

\begin{definition} The \defn{$D$-polynomial} of a marked graph $(G,\mathsf{m})$ is the polynomial in $\ZZ[\mathbf{z},y]$ defined by
\[D_{(G,\mathsf{m})}(\mathbf{z},y) = M_{(G,\mathsf{m})}\left(z_{w,d} = D_{\bullet_{w,d}},y\right),\]
where $\bullet_{w,d}$ is the one-vertex graph with mark $(w,d)$ and 
\begin{equation}
    \label{def:bullet}
D_{\bullet_{w,d}}=D_{\bullet_{w,d}}(\mathbf{z},y) := \sum_{i=0}^d(-1)^{i}\binom{d}{i}z_{w-i,0}z_{1,0}^{i}.
\end{equation}
Note that this definition is consistent since $M_{\bullet_{w,d}}(\mathbf{z},y)=z_{w,d}$. In particular, note that for any weight $w$,  $D_{\bullet_{w,0}}=z_{w,0}$.
\end{definition}
We call the substitution $z_{w,d}=D_{\bullet_{w,d}}$ the \defn{undotting substitution} since the result is a polynomial where all the variables $z_{w,d}$ have $d=0$. 

\begin{example}
Consider the marked graph $(G,\mathsf{m})$ of Example~\ref{ex:triangle}. Its $M$-polynomial is 
$$M_{(G,\mathsf{m})}(\mathbf{z},y) = z_{2,0}z_{4,1}z_{1,0}+z_{4,1}z_{3,1}+z_{1,0}z_{6,2}+z_{5,2}z_{2,0}+z_{7,3}(2+y).$$
Then its $D$-polynomial is given by 

\begin{align*}
D_{(G,\mathsf{m})}(\mathbf{z},y)=& D_{\bullet_{2,0}}D_{\bullet_{4,1}}D_{\bullet_{1,0}}+D_{\bullet_{4,1}}D_{\bullet_{3,1}}+D_{\bullet_{1,0}}D_{\bullet_{6,2}}+D_{\bullet_{5,2}}D_{\bullet_{2,0}}+D_{\bullet_{7,3}}(2+y)\\
= & z_{2,0}(z_{4,0}-z_{3,1}z_{1,0})z_{1,0}+(z_{4,0}-z_{3,1}z_{1,0})(z_{3,0}-z_{2,1}z_{1,0})\\
& +z_{1,0}(z_{6,0}-2z_{5,0}z_{1,0}+z_{4,0}z_{1,0}^2) + (z_{5,0}-2z_{4,0}z_{1,0}+z_{3,0}z_{1,0}^2)z_{2,0} \\
& + (z_{7,0}-3z_{6,0}z_{1,0}+3z_{5,0}z_{1,0}^2-z_{4,0}z_{1,0}^3)(2+y) \\
=& z_{3,0}z_{4,0}-z_{3,0}^2z_{1,0}  +z_{1,0}(z_{6,0}-2z_{5,0}z_{1,0}+z_{4,0}z_{1,0}^2)\\ 
&+ (z_{5,0}-2z_{4,0}z_{1,0}+z_{3,0}z_{1,0}^2)z_{2,0}\\ &+(z_{7,0}-3z_{6,0}z_{1,0}+3z_{5,0}z_{1,0}^2-z_{4,0}z_{1,0}^3)(2+y)~. \\
\end{align*}
\end{example}

Since the $D$-polynomial is a specialization of the $M$-polynomial, many results extend directly to the $D$-polynomial. 
\begin{prop}
\label{prop:dproperties}
Let $(G,\mathsf{m})$ be a marked graph and $e$ a non-loop edge of $G$.
Then the $D$-polynomial satisfies the \emph{deletion-contraction} relations:
\[ 
D_{(G,\mathsf{m})}(\mathbf{z},y) = D_{(G\setminus e,\mathsf{m})}(\mathbf{z},y) + 
D_{(G /  e,\mathsf{m}/e)}(\mathbf{z},y)
\]
and
\[
D_{(G/e,\mathsf{m}/e)}(\mathbf{z},y) = D_{(G,\mathsf{m})\odot e}(\mathbf{z},y) - D_{((G,\mathsf{m})\odot e)\setminus \ell_e}(\mathbf{z},y).\]
Moreover, the $D$-polynomial satisfies the \emph{deletion-near-contraction formula}
\[  D_{(G,\mathsf{m})}(\mathbf{z},y) = D_{(G\setminus e,\mathsf{m})}(\mathbf{z},y)- D_{(G,\mathsf{m})\odot e\setminus \ell_e}(\mathbf{z},y) +  D_{(G,\mathsf{m})\odot e}(\mathbf{z},y).\]
\end{prop}
\begin{proof}
This is a direct consequence of Proposition~\ref{prop:Mnearcontract} and the fact that $D$-polynomial is a specialization of the $M$-polynomial.
\end{proof}

In the process of computing the $D$-polynomial from the $M$-polynomial by doing the substitution $z_{w,d}= D_{\bullet_{w,d}}$ some cancellations may occur. We now seek to better understand these cancellations, to do this we begin with the following key definition. 

\begin{definition}
Let $(G,\mathsf{m})$ be a marked graph. A vertex $v$ is \defn{absorbable} if $v$ has degree one and mark $(1,0)$. An edge $e$ is \defn{absorbable} if it is incident to an absorbable vertex $v$. If $e=uv$ is an absorbable edge where the vertex $v$ is  absorbable and $\mathsf{m}(u)=(w,d)$, the \defn{absorption} of $e$ is the marked graph $(G/e, \mathsf{m}')$, where the marks of all vertices of $V(G/e)\setminus\{u\}$ are the same as in $G$, and $u$ has mark $\mathsf{m}'(u)=(w+1,d)$. The \defn{core} 
of $(G,\mathsf{m})$ is the marked graph obtained after absorbing all the absorbable edges of $(G,\mathsf{m})$. If $(G,\mathsf{m})$ is clear from the context, we denote its core by $(K,\mathsf{m}_K)$.
\end{definition} 

\begin{example} 
\label{ex:ejemplo8}
The graph below, $(G, \mathsf{m})$, is an unmarked graph, $\mathsf{m}= \mathbbm{1}$, with six  absorbable vertices colored red. Its core appears on the right:  
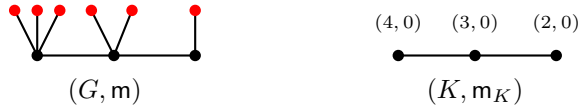
\begin{center}
\begin{mygraph}[baseline=-10pt,thick,scale=0.6]
\draw \foreach \x in {-0.5,0,0.5}
{ (\x,0) node[draw=red, fill=red] {} -- (0,-1) };
\draw \foreach \x in {1.2,2.1}
{ (\x,0) node[draw=red, fill=red] {} -- (1.7,-1)};
\draw \foreach \x in {3.5}
{ (\x,0) node[draw=red, fill=red] {} -- (3.5,-1)};
\draw (0,-1) node {} -- (1.7,-1)
             node {} -- (3.5,-1) 
             node {};
              \node[draw=none, fill=none] at (1.5,-1.8) {$(G,\mathsf{m})$};
\pgftransformshift{\pgfpoint{8cm}{0cm}}
\draw (0,-1) node[label=above:\tiny{${(4,0)}$}] {} -- (1.7,-1)
             node[label=above:\tiny{${(3,0)}$}] {} -- (3.5,-1)
             node[label=above:\tiny{${(2,0)}$}] {};
\draw (0,-1.2);
 \node[draw=none, fill=none] at (1.7,-1.8) {$(K,\mathsf{m}_K)$};
\end{mygraph}
\end{center}    
\end{example}

One of the main results of this section is that the $D$-polynomials of a marked graph and its core are equal. We first need an easy lemma.

\begin{lemma}
\label{lemma:Pascal}
For every $(w,d)\in\MM$ such that $w>1$ and $d>0$ we have
\[ D_{\bullet_{w,d}} = D_{\bullet_{w,d-1}}-z_{1,0}D_{\bullet_{w-1,d-1}}.
\]
\end{lemma}
\begin{proof}
The proof follows from Equation \eqref{def:bullet} and Pascal's identity. Indeed,
\begin{align*}
D_{\bullet_{w,d-1}}(\mathbf{z})
-z_{1,0} D_{\bullet_{w-1,d-1}}(\mathbf{z})
=& \sum_{i=0}^{d-1}(-1)^{i}\binom{d-1}{i}z_{w-i,0}z_{1,0}^i
-
z_{1,0}\sum_{i=0}^{d-1}(-1)^{i}\binom{d-1}{i}z_{w-1-i,0}z_{1,0}^i
\\
= & \binom{d-1}{0}z_{w,0} + \sum_{i=1}^{d-1}(-1)^{i}\left(\binom{d-1}{i}+\binom{d-1}{i-1}\right)
z_{w-i,0}z_{1,0}^i \\
& - (-1)^{d-1}\binom{d-1}{d-1}z_{w-d,0}z_{i,0}^d
\\
=& \sum_{i=0}^d(-1)^{i}\binom{d}{i}
z_{w-i,0}z_{1,0}^i,
\end{align*}
which is equal to $D_{\bullet_{w,d}}(\mathbf{z})$ again by Equation \eqref{def:bullet}.
\end{proof}

\begin{example}
 Let  $(G,\mathsf{m})$ be an edge whose endpoints have marks $(w-1,d-1)$ and $(1,0)$ so the total mark of $(G,\mathsf{m})$ is $(w,d)$.
 By Proposition \ref{prop:smr}, we have $D_{(G,\mathsf{m})}=z_{1,0}D_{\bullet_{w-1,d-1}}+D_{\bullet_{w,d}}$. That is, before cancellations $D_{(G,\mathsf{m})}$ has $3\cdot 2^{d-1}$ many terms. But by Lemma \ref{lemma:Pascal}, $D_{(G,\mathsf{m})}=D_{\bullet_{w,d-1}}$, which has $2^{d-1}$ terms. This means that there are $2^d$ many terms that cancel out and that by absorbing the leaf with mark $(1,0)$ we can avoid these cancellations.
\end{example}
In general, we have the following result.

\begin{theorem} \label{thm:Dcore}
Let $(G,\mathsf{m})$ be a marked graph and $(K,\mathsf{m}_K)$ be its core.
Then
\[D_{(G,\mathsf{m})}(\mathbf{z},y) = D_{(K,\mathsf{m}_K)}(\mathbf{z},y).\]
\end{theorem}

\begin{proof}
We proceed by induction on the number of absorbable edges of $(G,\mathsf{m})$. If $(G,\mathsf{m})$ does not have absorbable edges, the assertion is clear.
Now we assume the assertion is true for all marked graphs  with $k-1$ absorbable edges and suppose that $(G,\mathsf{m})$ has $k$ absorbable edges. 
Let $e$ be an absorbable edge of $(G,\mathsf{m})$. We will apply Lemma~\ref{lemma:Pascal} to prove 
\begin{equation}
    \label{lem:absortion}
D_{(G/e,\mathsf{m}/e)}(\mathbf{z},y)=D_{(G',\mathsf{m}')}(\mathbf{z},y)-D_{(G\setminus e,\mathsf{m})}(\mathbf{z},y);
\end{equation}
where $(G',\mathsf{m}')$ is the marked graph resulting from absorbing $e$. Indeed, by Equation \eqref{Mformula1}, we have 
\[ D_{(G/e,\mathsf{m}/e)}(\mathbf{z},y)=\sum_{A\subseteq E(G/e)} D_{\bullet_{w_1,d_1}}D_{\bullet_{w_2,d_2}}\cdots D_{\bullet_{w_{l},d_l}}(y-1)^{|A|-k_{G/e}(A)+|V(G/e)|},\]
where the $(w_i,d_i)$'s are the total marks of the connected components of $(G/e|_A,\mathsf{m}/e)$. 

We may assume without loss of generality that $(w_1,d_1)$ is the total mark of the connected component containing the contracted vertex $v_e$, and hence $w_1>1$ and $d_1>0$.  Thus, applying Lemma~\ref{lemma:Pascal} we get
\begin{align}
\label{proof:absorption}
    D_{(G/e,\mathsf{m}/e)}(\mathbf{z},y)=&
    \sum_{A\subseteq E(G/e)} D_{\bullet_{w_1,d_1-1}}D_{\bullet_{w_2,d_2}}\cdots D_{\bullet_{w_{l},d_l}}(y-1)^{|A|-k_{G/e}(A)+|V(G/e)|} \nonumber \\
   &-
        \sum_{A\subseteq E(G/e)} D_{\bullet_{1,0}}D_{\bullet_{w_1-1,d_1-1}}D_{\bullet_{w_2,d_2}}\cdots D_{\bullet_{w_{l},d_l}}(y-1)^{|A|-k_{G/e}(A)+|V(G/e)|}.
\end{align}

On one hand, by definition of absorption we can see that the first sum in the right-hand side of the last equation equals $D_{(G',\mathsf{m}')}(\mathbf{z},y)$. On the other hand, $G/e\sqcup\bullet$
is isomorphic to $G\setminus e$ since $e$ is a pendant edge, where the isomorphism sends the $\bullet$ vertex to the absorbable endpoint of $e$, the contracted vertex to the other endpoint of $e$ and each other vertex to itself. Moreover, this isomorphism induces a mark-isomorphism between 
$(G/e,\mathsf{m}'')\sqcup \bullet_{1,0}$ where $\mathsf{m}''$ gives mark $(w_1-1,d_1-1)$ to the contracted vertex and leaves all the other marks unchanged (from $\mathsf{m}/e$). Finally, observe that the second sum in the right-hand side of Equation \eqref{proof:absorption} is the $D$-polynomial of $(G/e,\mathsf{m}'')\sqcup \bullet_{1,0}$. Combining all these observations we obtain \eqref{lem:absortion}.
On the other hand, by deletion-contraction, we have $D_{(G,\mathsf{m})}(\mathbf{z},y) = D_{(G\setminus e,\mathsf{m})}(\mathbf{z},y)+D_{(G/e,\mathsf{m}/e)}(\mathbf{z},y)$. Combining this with \eqref{lem:absortion} yields that $D_{(G,\mathsf{m})}(\mathbf{z},y)=D_{(G',\mathsf{m}')}(\mathbf{z},y)$. By the induction hypothesis, $D_{(G',\mathsf{m}')}(\mathbf{z},y)=D_{(K',\mathsf{m}_{K'})}(\mathbf{z},y)$. The conclusion follows after observing that the cores of $(G,\mathsf{m})$ and $(G',\mathsf{m}')$ are mark-isomorphic.
\end{proof}
\begin{example}
Let us consider the unmarked graph $(G,\mathbbm{1})$ from Example~\ref{ex:ejemplo8} and its core $(K,\mathsf{m}_K)$. Theorem~\ref{thm:Dcore} allows for a large simplification in the computation of the $D$-polynomial of $(G,\mathbbm{1})$ by computing the $M$-polynomial of its core and then applying the undotting substitution. Indeed, if we compute the $M$-polynomial of $(G,\mathbbm{1})$ by its states model representation, we will have $64$ terms before the substitution $z_{w,k}= D_{\bullet_{w,k}}$. Instead, the states model representation of the $M$-polynomial of $(K,\mathsf{m}_K)$
has only $4$ terms. After performing the substitution $z_{w,k}= D_{\bullet_{w,k}}$ we get 
\[
D_{(G,\mathbbm{1})}(\mathbf{z},y)
=z_{4,0}z_{3,0}z_{2,0}+z_{5,0}z_{4,0}-z_{4,0}^2z_{1,0}+ z_{7,0}z_{2,0}-z_{6,0}z_{2,0}z_{1,0} + z_{9,0}-2z_{8,0}z_{1,0} + z_{7,0}z_{1,0}^2.
\]

\end{example}

Theorem \ref{thm:Dcore} implies that we will avoid many cancellations in the computation of the $D$-polynomial of a marked graph by first obtaining its core. However, some cancellations may still occur at some point if the core has some vertices with mark $(1,0)$. This motivates the following definition.

\begin{definition}
Consider $\MM^\circ :=\{(w,d)\in \MM\mid w>d+1\}$. We call marks in $\MM^\circ$ \defn{strict}. We say that a marked graph $(G,\mathsf{m})$ is \defn{strictly marked} if the mark of every vertex belongs to $\MM^\circ$. In particular, no vertex has mark $(1,0)$. We also say that $(G,\omega)$ is \defn{strictly weighted} if all weights are larger than $1$. Clearly, if $(G,\omega)$ is strictly weighted, it is also strictly marked when being considered as a marked graph.
\end{definition}
Observe that the dot-sum of two marks is strict if at least one of the marks is strict. It follows that $(\MM^\circ,\dotplus)$ is also a semigroup and that the total mark of a given set $U$ of vertices is strict if there is a vertex $v$ in $U$ with a strict mark. It follows that if $(G,\mathsf{m})$ is strictly marked, then all the variables appearing in the states model representation of $M_{(G,\mathsf{m})}(\mathbf{z})$ come from marks in $\MM^\circ$. Moreover, we have the following result.
\begin{prop}\label{prop:MDproper}
Let $(T,\mathsf{m})$ be a marked tree. The substitution $z_{w,d} = D_{\bullet_{w,d}}$ in the expression of the $M$-polynomial of $(T,\mathsf{m})$ given by Proposition~\ref{prop:smr} yields a cancellation-free expression for the $D$-polynomial of $(T,\mathsf{m})$ if and only if $(T,\mathsf{m})$ is strictly marked. 
\end{prop}

\begin{proof}
Since $T$ is a tree, each term in the expansion of the $M$-polynomial  given by Equation \eqref{Mformula1} is of the form $\mathbf{z}_{\lambda(T,\mathsf{m},A)}$, where $A\subseteq E(T)$. 
Moreover, each term of $D_{\bullet_{w,d}}$ is of the form $(-1)^{i}\binom{d}{i}z_{1,0}^{i}z_{w-i,0}$ for $0\leq i\leq d$. It follows that by substituting $z_{w,d} = D_{\bullet_{w,d}}$ into $\mathbf{z}_{\lambda(T,\mathsf{m},A)}$ yields a polynomial where each monomial is of the form
\begin{equation}
\label{e:term}
 C (-1)^\ell  z_{1,0}^\ell z_{w_1-i_1,0}z_{w_2-i_2,0} \cdots z_{w_k-i_k,0},  
\end{equation}
where $C$ is a positive constant, $\lambda(T,\mathsf{m},A)= (w_1,d_1), (w_2,d_2) ,\ldots , (w_k,d_k),$ for each $j \in \{1,2,\ldots, k\}$, $0 \leq i_j \leq d_j$ and  $\ell := i_1 + i_2 + \cdots +  i_k$. Note that there is a term of this form for every choice of $(i_1,\ldots,i_k)$ satisfying $0\leq i_j\leq d_j$ for all $i\in\{1,\ldots,k\}$. 

Suppose $(T,\mathsf{m})$ is strictly marked. Then every variable $z_{w,d}$ appearing in $M_{(T,\mathsf{m})}(\mathbf{z})$ satisfies $w > d+1$.  It follows that the variables $z_{w_j-i_j,0}$ in Equation  \eqref{e:term} are all different than $z_{1,0}$ and thus all terms with the same power of $z_{1,0}$ will have the same sign and thus they cannot cancel. 

Suppose now that $(T, \mathsf{m})$ is not strictly marked. If $T$ has some absorbable leaf, then by Theorem \ref{thm:Dcore} and the discussion preceding it, there will be many cancellations. Thus, we may assume that $(T,\mathsf{m})$ has an internal  vertex $v_1$ with  mark $(w_1,d_1)$ such that $w_1 = d_1 + 1$ and let $e=v_1v_2$ be any edge incident with $v_1$. Now consider the terms $\mathbf{z}_{\lambda(T,\mathsf{m},\emptyset)}$ and $\mathbf{z}_{\lambda(T,\mathsf{m},\{e\})}$ appearing in Equation \eqref{Mformula1}. We will show that undotting these terms produces a cancellation. Indeed these terms have the form 
\[z_{w_1,d_1}z_{w_2,d_2}\ldots z_{w_n,d_n}\quad\text{and}\quad
z_{w_1+w_2,d_1+d_2+1}z_{w_3,d_3}\ldots z_{w_n,d_n},\]
where $(w_1,d_1),(w_2,d_2),\ldots,(w_n,d_n)$ is the multiset of marks of the vertices of $(T,\mathsf{m})$ and we assume that $(w_i,d_i)$ is the mark of $v_i$ for $i\in\{1,2\}$. By choosing $i_1=d_1$ and $i_j = 0$ for $1 < j \leq n$  in Equation \eqref{e:term} undotting the first term yields  a term of the form 
$$ C' (-1)^{d_1}  z_{1,0}^{d_1} z_{w_1-d_1, 0}z_{w_2,0} \ldots z_{w_n,0}= C' (-1)^{d_1}  z_{1,0}^{d_1+1} z_{w_2,0} \ldots z_{w_n,0}.$$    

On the other hand, by choosing $i_1=d_1+1$ and  $i_j=0$ for all $3\leq j\leq n$ in Equation \eqref{e:term} undotting the second term yields a term of the form: 
$$C (-1)^{d_1+1} z_{1,0}^{d_1+1} z_{w_1+w_2-d_1-1,0}z_{w_3,0} \ldots z_{w_{n},0} = C (-1)^{d_1+1} z_{1,0}^{d_1+1} z_{w_2,0}\ldots z_{w_{n},0}.$$
The conclusion now follows by comparing both terms.

\end{proof}
For the following two results, we identify the variables $z_{w}$ with $z_{w,0}$.
\begin{coro}\label{coro:3n1}
Let $(T,\omega)$ be a weighted tree  of total weight $N$ and $n$ vertices and $(T,\mathsf{m}_\omega)$ its marked version. Suppose that we write 
\[D_{(T,\mathsf{m}_\omega)}(\mathbf{z},y)=\sum_{\lambda\vdash N} c_\lambda \mathbf{z_\lambda}.\]
Then, $\sum_{\lambda\vdash N} |c_\lambda|=3^{n-1}$
if and only if $(T,\omega)$ is strictly weighted. 
\end{coro}

\begin{proof}
Given an edge subset $A$ of $(T,\mathsf{m})$ of size $l$, the spanning subgraph $T|_A$ has $n-l$ connected components. Since  $\lambda(T,\mathsf{m},A)$ is a $\MM$-partition of $(N,n-1)$, this means that the induced term has $l=n-1-(n-l)+1$ dots. Moreover, by the Binomial theorem, after applying the undotting substitution we get $2^l$ terms. It follows that the expression obtained after applying the undotting substitution in \eqref{Mformula1} has 
\[  \sum_{l=0}^{n-1}\binom{n}{l}2^l = 3^{n-1}\]
terms. The conclusion now follows from Proposition \ref{prop:MDproper}.
\end{proof}

\begin{coro}
\label{WfromD}
Let $(G,\omega)$ be a strictly weighted graph and $(G,\mathsf{m}_{\omega})$ be its marked version. Then the $W$-polynomial of $(G,\omega)$ does not depend on $z_1$ and 
\[W_{(G,\omega)}(\mathbf{z},y)=D_{(G,\mathsf{m}_\omega)}(z_{1,0}=0,z_{w,0}=z_w,y).\]
\end{coro}
\begin{proof}
We already observed that the variable $z_{1,0}$ does not appear in the $M$-polynomial of a strictly weighted graph $(G,\mathsf{m})$. By Proposition \ref{prop:WfromM} it follows that $z_1$ does not appear in the $W$-polynomial of $(G,\omega)$. Moreover, if $(w,d)$ is different from $(1,0)$ then $D_{\bullet_{w,d}}(z_{1,0}=0,z_{w,0}=z_w)=z_{w}.$ The conclusion now follows.
\end{proof}

Now we turn our attention back to the  expansion of chromatic symmetric functions in the star basis. 
\begin{theorem}\label{thm:dyx}
Let $G$ be a simple graph and $(K,\mathsf{m}_K)$ be the core of $(G,\mathbbm{1})$. Then, 
\[X_G = D_{(K,\mathsf{m}_K)}(z_{w,0} = \st_w,y=0),\]
where $z_{w,0}=\st_w$ means that we substitute each variable $z_{w,0}$ with $\st_w$ for all $w\in\PP$. In other words, the expansion of $D$ into monomials encodes the chromatic star expansion of $X_G$.
\end{theorem}
\begin{proof}
We will show 
\[X_{G} = D_{(G,\mathbbm{1})}(z_{i,0}=\st_{i},y=0).\] Then, the assertion will follow from 
Theorem~\ref{thm:Dcore}. The proof is very similar to the proof of Theorem~\ref{thm:stfree} and we proceed by induction on the number of internal edges of $G$.  
For the base case, if $G$ has no internal edges and it is connected, then $G$ is a star $St_w$, and  by Theorem~\ref{thm:Dcore} we have $D_{(G,\mathbbm{1})}(\mathbf{z},y=0)=z_{w,0}$ since every leaf in $G$ is absorbable.  Since both the chromatic symmetric function and the $D$-polynomial are multiplicative over connected components, the disconnected case follows. 

 For the inductive step, we assume that the assertion holds for all graphs with at most $n$ internal edges, and we let $G$ be a graph with $n+1$ internal  edges.   By deletion-near-contraction applied to an internal  edge $e$, we have 
\[X_G = X_{G\setminus e} + X_{(G\odot e)^s}  - X_{(G\odot e)^s \setminus \ell_e}.\]
Since each graph on the right-hand side  of the last equation is simple and has at most $n$ internal edges, it follows from the induction hypothesis  that 
\[ X_G = D_{(G\setminus e,\mathbbm{1})}({z}_{w,0} = \st_w,y=0) + D_{((G\odot  e)^s,\mathbbm{1})}({z}_{w,0} =  \st_w,y=0)-D_{((G\odot  e)^s\setminus \ell_e,\mathbbm{1})}({z}_{w,0} =  \st_w,y=0).\]
Recall that by Remark~\ref{rem:My0}, the $D$-polynomial at $y=0$ does not distinguish between a graph and its simplification. Thus,
\[ X_G = D_{(G\setminus e,\mathbbm{1})}({z}_{w,0} = \st_w,y=0) + D_{(G\odot  e,\mathbbm{1})}({z}_{w,0} =  \st_w,y=0)-D_{(G\odot  e\setminus \ell_e,\mathbbm{1})}({z}_{w,0} =  \st_w,y=0).\]
Since $D$ also satisfies the deletion-near-contraction recurrence, we have 
\[ X_G = D_{(G,\mathbbm{1})}({z}_{w,0} =  \st_w,y=0),\]
which concludes the inductive step. Now the assertion follows by induction.

\end{proof}

Recall that a tree is proper if every internal vertex is adjacent to at least one leaf.

\begin{coro}
\label{MtoX}
Let $G$ be a simple graph and $(K,\mathsf{m}_K)$ be the core of $(G,\mathbbm{1})$. 
Then 
 \[X_G = \sum_{A\subseteq E(K)} (-1)^{|A|-r_K(A)}\prod_{i=1}^{k(A)}\left( \sum_{j=0}^{k_i}\binom{k_i}{j}(-1)^j\st_{(w_i-j,1^j)}\right),\]
 where $\{(w_i,k_i)\}_{1\leq i \leq k(A)}$ are the total marks of the components of 
 $(K|_A,\mathsf{m}_K)$.
Moreover, the expression given above is cancellation-free when $G$ is a proper tree.
\end{coro}
\begin{proof}
The first assertion follows directly from Theorem~\ref{thm:dyx}, the definition of the $D$-polynomial and the states model representation~\eqref{Mformula1} for the $M$-polynomial.

The fact that the expression is cancellation free whenever $G$ is a proper tree follows from Proposition~\ref{prop:MDproper} since the core of a proper tree is an strictly marked graph.
\end{proof}

\begin{remark}
In an independent work and using ideas similar to the deletion-near-contraction formula, I. Shah found another expression for the $\st$-expansion of $X_G$ \cite{chmutovshah}.
\end{remark}

\section{Computing the $D$-polynomial of non-strictly marked forests} \label{sec:Mp}

In this section we restrict our attention to forests. Our goal is to introduce a variation on the $M$-polynomial, denoted $M'$, such that we still have the identity $D_{(T,\omega)}(\mathbf{z})= M'_{(T,\omega)}(z_{w,k} = D_{\bullet_{w,k}}(\mathbf{z}))$ whenever $T$ is a weighted forest but no cancellations arise when doing the substitution. 

We start by noting a consequence of the deletion-contraction definition of the $M$-polynomial.  Since for a marked forest $(T, \mathsf{m})$ the $M$-polynomial has no term with $y$, from now on we omit this variable from the $M$-polynomial. For disjoint edge-sets $A_1,A_2$, the graph $(T/A_1\setminus A_2, \mathsf{m}/A_1)$ is the marked forest obtained by contracting the edges in $A_1$ and deleting the edges in $A_2$. It is well known that the resulting graph is independent of the order in which these deletions and contractions are performed.

\begin{lemma} \label{lem:Mstate}
Let $(T,\mathsf{m})$ be a marked forest and let $B$ be any subset of $E(G)$. Then
\[
M_{(T,\mathsf{m})}(\mathbf{z})=\sum_{A_1\sqcup A_2 = B} M_{(T/A_1\setminus A_2, \mathsf{m}/A_1)}(\mathbf{z})
,\]
where the sum ranges over all disjoint ordered pairs $(A_1,A_2)$ such that $A_1\sqcup A_2=B$.
\end{lemma}

Next, we define the $M'$-polynomial. It basically obeys the same deletion-contraction rule as the $M$-polynomial, except in the case that the graph is a star whose center has weight $(1,0)$. In that case, the deletion-contraction rule applies to all but one of the edges of the star. To guarantee that it is well defined, we do all deletions and contractions at the same time. 

\begin{definition}\label{def:Mprime}
Given a total order on $\MM$, the \defn{$M'$-polynomial} of a marked forest $(T,\mathsf{m})$ is the unique polynomial in the set of variables $\{z_{w,d}\}$ such that:
  \begin{enumerate}
  \item if $(T,\mathsf{m})$ is a single vertex with mark $(w,d)$, then $M'_{(T,\mathsf{m})}(\mathbf{z})=z_{w,d}$;
  \item if $(T,\mathsf{m})$ has connected components 
  $(T_1,\mathsf{m}_1),\ldots, (T_k,\mathsf{m}_k)$ with $k>1$, then 
  \[M'_{(T,\mathsf{m})}(\mathbf{z})=M'_{(T_1,\mathsf{m}_1)}(\mathbf{z})\cdots M'_{(T_k,\mathsf{m}_k)}(\mathbf{z});\]
  \item if $(T,\mathsf{m})$ is connected and has absorbable vertices, then $M'_{(T,\mathsf{m})}(\mathbf{z})=M'_{(K, \mathsf{m}_K)}(\mathbf{z})$, where $(K,\mathsf{m}_K)$ denotes the core of $T$;
 \item if $(T,\mathsf{m})$ is connected, has at least one edge and no absorbable vertices, then let $A$ be the set of pendant edges of $T$. We distinguish two subcases:
 \begin{enumerate}
     \item If $T$ is a star whose center has mark $(1,0)$ then let $e$  be the pendant edge incident to the leaf with \emph{largest} mark (resolving tie at random if needed) and let $A'=A\setminus\{e\}$;
     \item Otherwise, let $A'=A$;
\end{enumerate}
Then 
    \begin{equation}
    \label{eq:mprimarec}
M'_{(T,\mathsf{m})}(\mathbf{z}) = 
\sum_{A_1\sqcup A_2 = A'} M'_{(T/A_1\setminus A_2, \mathsf{m}/A_1)}(\mathbf{z});
\end{equation}
 \end{enumerate}
\end{definition}

For a fixed total order on $\MM$, the definition above yields a unique polynomial, since any marked forest falls in just one of the cases, and if there are several leaves with maximum mark in case (d), it is clear that the result is independent of the choice of $e$. In fact, the  $M'$-polynomial depends on the total order chosen; however, Theorem~\ref{thm:Mprime_nocancel} below is independent of the order. In an upcoming article \cite{aliste2022markedII} we plan to discuss this in more detail using elimination theory and Gr\"obner bases to analyze these polynomials. 

\begin{example}
Suppose that $(T,\mathsf{m})$ is the marked star given in Example~\ref{ex:M-poly2}. Recall that $M_{(T,\mathsf{m})}(\mathbf{z})= z_{4,0}z_{1,0}z_{2,0}+z_{5,1}z_{2,0}+z_{4,0}z_{3,1}+z_{7,2}$.
Observe that applying the undotting substitution to any of the terms of degree $2$ in the latter expression will yield a cancellation with the term $z_4z_1z_2$. The choice of a total order determines which cancellation will be avoided by absorbing the corresponding edge. Indeed if $(4,0)$ is larger than $(2,0)$, then in the Definition~\ref{def:Mprime}.(d) we get $A'=\{e'\}$ where $e'$ is the pendant edge incident with the leaf of mark $(2,0)$. Thus, 
\[
M'(\mathbf{z})= M'(
\begin{mygraph2}[baseline=-3pt]\draw (0,0) node[label=below:${(2,0)}$] {};
\draw (1,0) node[label=below:${(1,0)}$] {}  -- (2,0) node[label=below:${(4,0)}$] {};
\end{mygraph2})
+M'(\begin{mygraph2}[baseline=-3pt]\draw (0,0) node[label=below:${(3,1)}$] {}--
(1,0) node[label=below:${(4,0)}$] {};
\end{mygraph2}).\]
The first term of the right-hand side  of the computation is handled  by using  items (a),(b) and (c) in Definition~\ref{def:Mprime}, which gives $z_{2,0}z_{5,0}$. The second term is computed by (d) and (a),  which give $z_{3,1}z_{4,0}+z_{7,2}$.

Instead, if $(2,0)$ is larger than $(4,0)$ a similar computation yields 
\[M'(\mathbf{z})= z_{2,0}z_{5,1}+z_{4,0}z_{3,0}+z_{7,2}.\]
Observe that in both cases after applying the undotting substitution we get $D(\mathbf{z})$.
\end{example}
If a marked tree has no mark of the form $(1,0)$, its $M'$-polynomial will be computed by subcase (ii) of case (d) of Definition~\ref{def:Mprime}, and by Lemma~\ref{lem:Mstate} it will be equal to the $M$-polynomial.
\begin{prop}
Suppose that every mark of $(T,\mathsf{m})$ is different from $(1,0)$ (but no necessarily strict). Then 
\[M'_{(T,\mathsf{m})}(\mathbf{z})=M_{(T,\mathsf{m})}(\mathbf{z}).\]
\end{prop}

We next show, as claimed, that for a marked forest $(T, \mathsf{m})$  the substitution $z_{w,d}= D_{\bullet_{w,d}}$ in $M'_{(T,\mathsf{m})}$ gives $D_{(T,\mathsf{m})}$.

\begin{theorem}\label{thm:Mprime_nocancel}
For every marked forest $(T,\mathsf{m})$, we have $D_{(T,\mathsf{m})}(\mathbf{z})=M'_{(T,\mathsf{m})}(z_{w,d}= D_{\bullet_{w,d}})$. 
\end{theorem}

    \begin{proof}

The proof is by induction on $n$, the number of edges of $T$.  Recall that by definition $D_{(T,\mathsf{m})}(\mathbf{z})= M_{(T,\mathsf{m})}(z_{w,d}= D_{\bullet_{w,d}})$.

If $n=0$, the claim is clear since by definition $M$ and $M'$ agree on edgeless forests.

For the induction step, assume the statement  holds for all forests with fewer than $n$ edges, and let $T$ be a forest with $n$ edges. We consider the following cases.

\begin{itemize}

\item If $(T,\mathsf{m})$ is connected and has some absorbable vertices, let $(K, \mathsf{m}_K)$ be its core. Then
$$M'_{(T,\mathsf{m})}(z_{w,d}= D_{\bullet_{w,d}})=M'_{(K,\mathsf{m}_K)}(z_{w,d}= D_{\bullet_{w,d}})=D_{(K,\mathsf{m}_K)}(\mathbf{z})=D_{(T,\mathsf{m})}(\mathbf{z}),$$
where we have used the induction hypothesis on the second equality and Theorem~\ref{thm:Dcore} on the third.

\item Assume now that $T$ is a star with no absorbable vertices, whose center has mark $(1,0)$ and with set of pendant edges  $L=\{e_1,\ldots,e_n\}$, where  $e_n$ is incident to the leaf of largest mark. We have 
$$M'_{(T,\mathsf{m})}(\mathbf{z})=\sum_{A_1\sqcup A_2 =  \{e_1,\ldots,e_{n-1}\}} M'_{(T/A_1\setminus A_2, \mathsf{m}/A_1)}(\mathbf{z}).$$
  Each of the forests $T/A_1\setminus A_2$ consists of a single edge $e_n$ and several (perhaps none) isolated vertices. Since $T$ has at least two edges (otherwise it would have an absorbable vertex), the inductive hypothesis and Lemma~\ref{lem:Mstate} give
\begin{align*} M'_{(T,\mathsf{m})}(z_{w,d}=D_{\bullet_{w,d}})&=\sum_{A\subseteq L-\{e_n\}} M'_{(T/A\setminus L-A-\{e_n\},\mathsf{m}/A)}(z_{w,d}=D_{\bullet_{w,d}})\\ &= \sum_{A\subseteq L-\{e_n\}} D_{(T/A\setminus L-A-\{e_n\},\mathsf{m}/A)}(\mathbf{z})\\ &= \sum_{A\subseteq L-\{e_n\}} M_{(T/A\setminus L-A-\{e_n\},\mathsf{m}/A)}(z_{w,d}=D_{\bullet_{w,d}})
\\ &=M_{(T,\mathsf{m})}(z_{w,d}=D_{\bullet_{w,d}})=D_{(T,\mathsf{m})}(\mathbf{z}).
\end{align*}

\item If $T$ is connected but it does not fall in the above, we can again use the inductive hypothesis and Lemma~\ref{lem:Mstate} to conclude
  \begin{align*} M'_{(T,\mathsf{m})}(z_{w,d}=D_{\bullet_{w,d}})& =\sum_{A\subseteq L} M'_{(T/A\setminus L-A,\mathsf{m}/A)}(z_{w,d}=D_{\bullet_{w,d}})=\sum_{A\subseteq L} D_{(T/A\setminus L-A,\mathsf{m}/A)}(\mathbf{z})\\ &=\sum_{A\subseteq L} M_{(T/A\setminus L-A,\mathsf{m}/A})(z_{w,d}=D_{\bullet_{w,d}})=M_{(T,\mathsf{m})}(z_{w,d}=D_{\bullet_{w,d}})\\ & =D_{(T,\mathsf{m})}(\mathbf{z}).
  \end{align*}

\item If $(T,\mathsf{m})$ has connected components $(T_1,\mathsf{m}_1),\ldots, (T_k,\mathsf{m}_k)$ with $k>1$, the multiplicative property of $D$ (inherited from that of $M$), gives 
    $D_{(T,\mathsf{m})}(\mathbf{z}) =D_{(T_1,\mathsf{m}_1)}(\mathbf{z})\cdots D_{(T_k,\mathsf{m}_k)}(\mathbf{z})$. Each of the terms $D_{(T_i,\mathsf{m}_i)}(\mathbf{z})$ equals $M'_{(T_i,\mathsf{m}_i)}(z_{w,d}=D_{\bullet_{w,d}})$, either by the inductive hypothesis or because $T_i$ is a single vertex. By the definition of $M'$ for the disconnected case, we conclude $D_{(T,\mathsf{m})}(\mathbf{z}) = M'_{(T,\mathsf{m})}(z_{w,d}=D_{\bullet_{w,d}})$.  
\end{itemize}

  This finishes the proof of the theorem.
  \end{proof}
  We say that $(T,\mathsf{m})$ is \defn{almost strictly marked} if every mark of $(T,\mathsf{m})$ belongs to $\MM^\circ\cup\{(1,0)\}$. In particular, every weighted tree is almost strictly marked. 
 \begin{prop}
 \label{prop:al_str}
 Let $(T,\mathsf{m})$ be an almost strictly marked tree with more than one vertex. Then all the variables $z_{w,d}$ that appear in the monomials of $M'_{(T,\mathsf{m})}(\mathbf{z})$ are strict, that is $w\geq d+2$. 
 \end{prop}
 \begin{proof}
 By the definition of $M'$, it is enough to check that if during the computation of $M'$ a vertex becomes isolated then its mark is of the form $(w,d)$ with $w-d\geq 2$.
  Initially all marks are either $(1,0)$ or $(w,d)$ with $w-d\geq 2$. Observe that marks only change when contracting an edge or when absorbing an absorbable leaf.
  Every time a contraction takes place, two vertices with marks $(w_1,d_1)$ and $(w_2,d_2)$ are contracted into a vertex with mark $(w_1+w_2,d_1+d_2+1)$. Since the definition of $M'$ prevents contracting an edge with both endpoints having mark $(1,0)$, at least one mark has to be strict. Thus we can assume $w_1-d_1\geq 2$ and $w_2-d_2\geq 1$, which implies  $w_1+w_2-(d_1+d_2+1)\geq 2$. On the other hand, when absorbing a leaf, a vertex with mark $(1,0)$ is removed and a vertex with mark $(w,d)$ gets mark $(w+1,d)$. We see that $w+1-d\geq 2$ since $w-d\geq 1$  by the definition of marks. So it only remains to check that a vertex $v$ with mark $(1,0)$ cannot become isolated during the computation of $M'_{(T,\mathsf{m})}$. Indeed, if $v$ is a leaf, then it gets immediately absorbed. Otherwise, as at each step we only may delete leaves, at some point either $v$ becomes a leaf (and thus gets absorbed at the next step) or it becomes the center of a star (and thus gets absorbed after two steps).  This finishes the proof.
 \end{proof}
 \begin{prop}
  If $(T,\mathsf{m})$ is an almost strictly marked tree, then no cancellations arise after applying the undotting substitution to $M'_{(T,\mathsf{m})}(\mathbf{z})$.
\end{prop}
 \begin{proof}
If $T$ has only one vertex with mark $(w,d)$ then applying the undotting substitution to $M'_{(T,\mathsf{m})}(\mathbf{z})$ yields $D_{\bullet_{w,d}}$. From \eqref{def:bullet}, we check there are no cancellations and this finishes the proof if $T$ only has one vertex. 

If $T$ has more than one vertex, then by Proposition~\ref{prop:al_str} when substituting 
$z_{w,d}$ with $D_{\bullet_{w,d}}$
we get monomials  of the form $z_{w-i,0}z_{1,0}^i$ where 
$w-i\geq w-d \neq 1$. Hence,  when doing the undotting substitution in a monomial $z_{w_1,d_1}\cdots z_{w_k,d_k}$ of $M'_{(T,\mathsf{m})}$ all the resulting monomials are of the form $(-1)^{\ell}z_{1,0}^{\ell}P(\mathbf{z})$ for some monomial $P(\mathbf{z})$ in which the variable $z_{1,0}$ does not appear. It follows that no cancellations can occur since  monomials that have a fixed power of $z_{1,0}$ appear always with the same sign, and clearly all coefficients in $M'$ are non-negative.  
\end{proof}

\section{Weighted tree reconstruction from marked polynomials}\label{sec:app}

In this section we introduce the problems of weighted tree reconstruction from the $M$-polynomial and the $D$-polynomial, and show their relationship with Stanley's tree isomorphism problem about distinguishing (unweighted) trees with the chromatic symmetric function.

Recall that whenever necessary, a weighted tree $(T,\omega)$ is considered as a marked tree $(T,\mathsf{m}_{\omega})$, where a vertex of weight $w$ has mark $(w,0)$. Similarly, an unweighted graph is considered weighted with all weights equal to 1. To simplify the notation, when writing terms for the $D$-polynomial we  shall identify the variables $z_{i,0}$ with $z_{i}$ and write the second instead of the first. In this section we only deal with trees, so the $M$ and $D$-polynomials do not depend on  the variable $y$ and thus we omit it from their arguments.
 
As both $M$ and $D$ are polynomials for marked graphs, the first question is whether there exist two marked trees, different up to marked-isomorphism, with the same $M$-polynomial. The answer is positive, as shown by the following example.
\begin{example} The following two trees were given in~\cite{loebl2018isomorphism} as an example of two non-isomorphic
weighted graphs with the same $W$ polynomial (for the general construction of paths with the same $W$-polynomial see \cite{aliniaeifard2020extended,aliste2021vertex,aliste2014proper}).

\begin{center}
\begin{mygraph}[baseline=-3pt,thick,scale=0.6]
\pgftransformshift{\pgfpoint{8cm}{0cm}}
\draw (0,-1) node[label=above:{$2$}] {} -- (1.5,-1)
             node[label=above:{$1$}] {} -- (3.0,-1) 
             node[label=above:{$2$}] {} -- (4.5,-1)
             node[label=above:{$3$}] {} -- (6.0,-1)
             node[label=above:{$1$}] {};
\draw (0,-1);
 \node[draw=none, fill=none] at (3.0,-1.8) {$(P_1,\omega_1)$};
\end{mygraph} \quad \quad
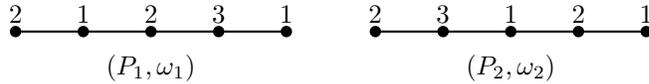
\begin{mygraph}[baseline=-3pt,thick,scale=0.6]
\pgftransformshift{\pgfpoint{8cm}{0cm}}
\draw (0,-1) node[label=above:{$2$}] {} -- (1.5,-1)
             node[label=above:{$3$}] {} -- (3.0,-1) 
             node[label=above:{$1$}] {} -- (4.5,-1)
             node[label=above:{$2$}] {} -- (6.0,-1)
             node[label=above:{$1$}] {};
\draw (0,-1);
 \node[draw=none, fill=none] at (3.0,-1.8) {$(P_2,\omega_2)$};
\end{mygraph}
\end{center}

Note that these two trees can be thought as marked trees with marks $(w,0)$ and as such they have different $M$-polynomial. By changing the vertex labels to marks so that the weights are the same but one mark $(w,d)$ has $d\neq 0$, we turn these weighted paths into non-isomorphic marked paths $(P_1,\mathsf{m}_1)$ and $(P_2,\mathsf{m}_2)$.
\begin{center}
\begin{mygraph}[baseline=-3pt,thick,scale=0.6]
\pgftransformshift{\pgfpoint{8cm}{0cm}}
\draw (0,-1) node[label=above:\tiny{$(2,0)$}] {} -- (1.5,-1);
\draw[dashed]  (1.5,-1) node[label=above:\tiny{$(1,0)$}] {} -- (3.0,-1);
\draw (3.0,-1) node[label=above:\tiny{$(2,0)$}] {} -- (4.5,-1)
             node[label=above:\tiny{$(3,1)$}] {} -- (6.0,-1)
             node[label=above:\tiny{$(1,0)$}] {};
\draw (0,-1);
 \node[draw=none, fill=none] at (3.0,-1.8) {$(P_1,\mathsf{m}_1)$};
\end{mygraph} \quad \quad
\begin{mygraph}[baseline=-3pt,thick,scale=0.6]
\pgftransformshift{\pgfpoint{8cm}{0cm}}
\draw (0,-1) node[label=above:\tiny{$(2,0)$}] {} -- (1.5,-1)
             node[label=above:\tiny{$(3,1)$}] {} -- (3.0,-1) 
             node[label=above:\tiny{$(1,0)$}] {};
\draw[dashed] (3.0,-1) -- (4.5,-1)
             node[label=above:\tiny{$(2,0)$}] {};
\draw (4.5,-1) -- (6.0,-1)
             node[label=above:\tiny{$(1,0)$}] {};
\draw (0,-1);
 \node[draw=none, fill=none] at (3.0,-1.8) {$(P_2,\mathsf{m}_2)$};
\end{mygraph}
\end{center}
The reader can verify that these two marked trees have the same $M$-polynomials by applying deletion-contraction to the dashed edges. It follows that they also have the same $D$-polynomial.

Recall that Theorem~\ref{thm:Dcore} implies that the $D$-polynomial cannot distinguish between a marked tree and its core. Thus, the cores of $(P_1,\mathsf{m}_1)$ and $(P_2,\mathsf{m}_2)$ also have the same $D$-polynomial. However, the $M$-polynomials of the cores are different, since the states model representation implies that the respective highest-degree terms are $z_{4,1}z_{2,0}^2z_{1,0}$ and $z_{3,1}z_{3,0}z_{2,0}z_{1,0}$.  This is an example of two different $M$-polynomials leading to the same $D$-polynomial, the reason that the $D$-polynomials are equal is because of cancellations that occur when we apply the undotting substitution, \emph{i.e.},  $z_{w,d}\mapsto D_{\bullet_{w,d}}$.
\end{example}

We are interested in determining if the $M$ and $D$ polynomials distinguish unweighted graphs, i.e., with all marks of the form $(1,0)$. However, we work with weighted graphs, i.e., graphs with all marks of the form $(w,0)$, because by Theorems~\ref{thm:Dcore} and~\ref{thm:dyx} to compute the $D$-polynomial we use the core of the graph, which is a weighted graph, and from the $D$-polynomial we compute the chromatic symmetric function.  Moreover, as the core of a graph does not have absorbable vertices, we restrict ourselves to the case of weighted trees with no absorbable vertices.

If there existed two non-isomorphic weighted trees $(T_1,\omega_1)$ and $(T_2,\omega_2)$ such that  $D_{(T_1,\omega_1)}=D_{(T_2,\omega_2)}$,  this would produce a pair of unweighted trees with the same chromatic symmetric function just by replacing each vertex of weight $(w,0)$ with an unweighted vertex and $w-1$ leaves adjacent to it. Hence, the problem of distinguishing unweighted trees with the chromatic symmetric function is equivalent to that of distinguishing weighted trees without absorbable leaves with the $D$-polynomial.

It is also unknown to the authors whether the $M$-polynomial distinguishes weighted trees, i.e. with all marks of the form $(w,0)$.  
\begin{question}
Is there a pair of weighted trees $(T_1,\omega_1)$ and $(T_2,\omega_2)$, different up to $\omega$-isomorphism, such that $M_{(T_1,\mathsf{m}_{\omega_1})}(\mathbf{z})=M_{(T_2,\mathsf{m}_{\omega_2})}(\mathbf{z})$?
\end{question}

 We remark that even if it was known that the $M$-polynomial distinguishes weighted trees up to weighted-isomorphism, this would not rule out the possibility of having a pair of non-isomorphic unweighted trees with the same chromatic symmetric function. In order to deduce this, one would need to show that the $M$-polynomial of a weighted tree can be recovered from its $D$-polynomial. This problem will be studied in detail in \cite{aliste2022markedII}. 
\subsection{On $M$-reconstructability}
In the remainder of this section we will see how the above programme can be applied to show that the chromatic symmetric function distinguishes proper trees, i.e. all the internal vertices are adjacent to at least one leaf,  of diameter smaller than $6$ up to isomorphism.  By the results of Section~\ref{sec:Dpoly}, this means that we need to study the $M$ and $D$-reconstructability of weighted trees up to diameter $3$. In other words, we will restrict our attention to weighted trees with up to one internal edge. It is easy to see that if $(T,\omega)$ is such a tree, then either $T$ is a star or is a \defn{$2$-star}, that is a tree formed by taking two stars and connecting their centers by an edge.

\begin{remark}

In what follows, we repeatedly use the states model representation of $M_{(T,\mathsf{m}_{\omega})}(\mathbf{z})$ given by Proposition~\ref{prop:smr} in combination with the undotting substitution $z_{w,d} = D_{\bullet_{w,d}}.$  Recall that given an $\MM$-partition $\lambda$ of length $l$, the term  $\mathbf{z}_\lambda$ is defined as $z_{\lambda_1}z_{\lambda_2}\cdots z_{\lambda_l}$, and if $f(\mathbf{z})$ is a polynomial, then  $[\mathbf{z}_\lambda]{f}(\mathbf{z})$ denotes the coefficient of the term $\mathbf{z}_\lambda$ in $f(\mathbf{z})$. Moreover, every multiset $\mathcal{W}$ of positive integers induces an $\MM$-partition $\{(w,0)\mid w\in\mathcal{W}\}$ and by a slight abuse of notation we let $\mathbf{z}_\mathcal{W}=\prod_{w\in    \mathcal{W}} z_{w,0}$. 
When $(T,\omega)$ is clear from the context, we shall write $M(\mathbf{z})$ and $D(\mathbf{z})$ instead of $M_{(T,\mathsf{m}_{\omega})}(\mathbf{z})$ and $D_{(T,\mathsf{m}_{\omega})}(\mathbf{z})$.
\end{remark}
We start by giving some basic facts that will be needed later:

\begin{prop}
\label{prop:leaves}
Let $(T,\omega)$ be a weighted tree. Then, the following quantities: 
\begin{enumerate}
    \item the number of vertices, edges and the total weight;
    \item the multiset of vertex weights; 
    \item the number of leaves and the multiset of leaf weights;
\end{enumerate}
can be recovered from the terms of degree one, two and the term of largest degree of $M(\mathbf{z})$.
\end{prop}
\begin{proof}
By Equation \eqref{Mformula1},  there is a unique term of degree one in $M(\mathbf{z})$ and this term is of the form $z_{N,n-1}$ where $N$ is the total weight of $(T,\omega)$ and $n$ is the number of vertices of $T$.  This gives us $(a)$.

 We also see that in Equation \eqref{Mformula1} the term of largest degree corresponds to the edge set $A=\emptyset$. Hence, this term has the form 
$\mathbf{z}_\mathcal{W}$, where $\mathcal{W}$ is the multiset of vertex  weights of $T$. This gives us $(b)$.

To get $(c)$ observe that the terms of degree two in $M(\mathbf{z})$ are in one-to-one correspondence with edge sets $A=E\setminus\{e\}$, where $e$ is an edge of $T$. We distinguish two cases: First, if the order of $T$ is at most two,  then every vertex is a leaf and thus the number of leaves and its multiset of weights can be deduced from $(a)$ and $(b)$. Second, if $T$ has more than two vertices, then each pendant edge $e$ is incident to a unique leaf $v$. Then we can set a bijective correspondence between each pendant edge and the monomial $z_{w,0}z_{N-w,n-2}$ in $M(\mathbf{z})$, where $w$ is the weight of $v$ and $N$ is the total weight of $(T,\omega)$. Note that $z_{w,0}z_{N-w,n-2}$ is indeed a monomial in $M(\mathbf{z})$ since we would obtain it by deleting $e$ and contracting all other edges in $T$, and these are the only monomials of this type.
This means that by collecting all the terms of degree two in $M$ which have a factor of the form $z_{i,0}$ we recover the multiset of leaf weights. Since the number of leaves coincides with the cardinality of this multiset, this gives us $(c)$. 
\end{proof}

\begin{prop}
\label{prop:starsrecognizable}
Let $(T,\omega)$ be a weighted tree. Then, we can recover whether $T$ is a star or a $2$-star from its  $M$-polynomial.
\end{prop}
\begin{proof}
By Proposition~\ref{prop:leaves} (a) and (c) we recover $n$  the number of vertices of $T$ and the number of leaves from the $M$-polynomial. Since stars can be characterized as the only trees with $n-1$ leaves and $2$-stars can be characterized as the only trees with $n-2$ leaves, this finishes the proof. 
\end{proof}

\begin{theorem}\label{thm:starsMreconstr}
Let $(T,\omega)$ be a weighted star. Then it can be reconstructed from its $M$-polynomial.
\end{theorem}
\begin{proof}
This follows directly from Proposition~\ref{prop:leaves}, since weighted stars are determined up to $\omega$-isomorphism by the multiset of leaf weights and the total weight.
\end{proof}

Before stating the next lemma, recall that two edges are \defn{independent} if they do not share a vertex.
\begin{lemma}\label{l:mdeg-1}
Let $(T,\omega)$ be a weighted tree. Let $\mathcal{W}$ be the multiset of vertex weights of $(T,\omega)$ and be $w_0, w_1, w_2$ be three (not necessarily distinct) numbers occurring in $\mathcal{W}$. 
Then, the following statements hold: 
\begin{enumerate}
    \item The coefficient
\[\alpha(w_1,w_2):=[z_{w_1+w_2,1}\,\mathbf{z}_{\mathcal{W}\setminus\{w_1,w_2\}}]M(\mathbf{z})\]
counts the number of edges joining a vertex of weight $w_1$ with a vertex of weight $w_2$. Here $\mathcal{W}\setminus\{w_1,w_2\}$ denotes removing one occurrence of each $w_1$ and $w_2$ from $\mathcal{W}$.
\item 
 The coefficient
\[\beta(w_0,w_1,w_2):=[z_{w_0+w_{1},1}\, z_{w_0+w_{2},1}\, \mathbf{z}_{\mathcal{W}\setminus\{w_0,w_0,w_1,w_2\}}]M(\mathbf{z})\]
counts sets of independent edges $\{e_1,e_2\}$ where each $e_i$ connects a vertex of weight $w_0$ with a vertex of weight $w_i$. 
Here $\mathcal{W}\setminus\{w_0,w_0,w_1,w_2\}$ denotes removing two occurrences of $w_0$ and then one occurrence of each $w_1$ and $w_2$ from $\mathcal{W}\setminus\{w_0,w_0\}$. Note that if there are not enough occurrences of either $w_0,w_1$ or $w_2$ then $\beta(w_0,w_1,w_2)$ is defined to be $0$.
\end{enumerate}
\end{lemma}
\begin{proof}
Let $n$ be the number of vertices of $T$. Since $T$ is a tree, every term of degree $n-1$ in Equation \eqref{Mformula1} comes from contracting an edge and deleting all the others. Indeed, if $e$ is an edge whose endpoints have weights $w_1$ and $w_2$, then $e$ induces the monomial 

\begin{equation}\label{eq:mon}
z_{w_1+w_2,1}\,\mathbf{z}_{\mathcal{W}\setminus\{w_1,w_2\}}.
\end{equation}
Let $w_1'$ and $w_2'$ be any other weights in $\mathcal{W}$. Since $\mathcal{W}\setminus{\{w_1, w_2\}}= \mathcal{W}\setminus{\{w_1',w_2'\}}$ if and only if $\{w_1,w_2\}=\{w_1',w_2'\}$, it follows that two edges induce the same term if and only if they have endpoints with the same weights. The first  assertion now follows.

For the second assertion, observe that every term of degree $n-2$ in  \eqref{Mformula1} comes from contracting two (distinct) edges $e_1$ and $e_2$ and deleting  the rest of the edges. Let $a,b$ be the weights of the endpoints of $e_1$ and $c,d$ be the weights of the endpoints of $e_2$. If $e_1$ and $e_2$ are independent, then they must induce a monomial of the form 
$z_{a+b,1}\,z_{c+d,1}\,\mathbf{z}_{\mathcal{W}\setminus\{a,b,c,d\}}$.
Otherwise $e_1$ and $e_2$ share a vertex, in which case they  must induce a monomial of the form 
$z_{a+b+d,2}\,\mathbf{z}_{\mathcal{W}\setminus\{a,b,d\}},$
where we assumed without loss of generality that $b=c$ is the weight of the common vertex. It follows that 
the coefficient $\beta(w_0,w_1,w_2)$ counts sets of independent edges $\{e_1,e_2\}$ such that:
\begin{enumerate}[(i)]
    \item for each $i\in\{1,2\}$, the total weight of $e_i$ is $w_0+w_i$;
    \item the multiset of weights of the vertices of the graph induced by $\{e_1,e_2\}$ is $\{w_0,w_0,w_1,w_2\}$. 
\end{enumerate}
Note that if $w_1=w_0$, then from (ii) it must be the case that $e_1$ is incident to vertices of weight $w_0$ and $w_1$, and $e_2$ is incident to vertices of weight $w_0$ and $w_2$. The same happens if $w_2=w_0$. Finally, if $w_0\not\in\{w_1,w_2\}$, then from (i) and (ii) it follows that each $e_i$ is incident with vertices of weight $w_0$ and $w_i$ since the only other possibilities are that each $e_i$ is incident with vertices of weight $\{w_0,w_0\}$ or $\{w_1,w_2\}$ but these contradict (i) since $w_0\neq w_1$ and $w_0\neq w_2$.
\end{proof}

\begin{theorem}\label{t:m2star}
Let $(T,\omega)$  be a weighted $2$-star. Then it can be reconstructed from its $M$-polynomial.
\end{theorem}
\begin{proof}
Let $(T,\omega)$ be a $2$-star. Let $N$ be the total weight of $(T,\omega)$, $n$ its number of vertices, $\mathcal{W}$ the multiset of vertex weights and 
 $\mathcal{W}_L$ the multiset  of leaf weights of $(T,\omega)$. All of these can be reconstructed from $M(\mathbf{z})$ by Proposition~\ref{prop:leaves}.
 
By comparing $\mathcal{W}$ with $\mathcal{W}_L$ we recover the weights $\{u_0,u_1\}$ of the centers of $T$. Observe that it might happen that $u_0=u_1$. For $k\geq 1$, let  $w_1,\ldots,w_{k}$ denote all the distinct leaf weights and set $\mu_j$ to denote the multiplicity of $w_j$ in $\mathcal{W}_L$ for every $j\in\{1,\ldots,k\}$. Note there may be leaves with the same weights as the weights of the centers.

Assume first that the centers have different weights: $u_0\neq u_1$. In this case, we will be able to reconstruct $(T,\omega)$ from the terms of degree $n-1$ in the states model representation of $M(\mathbf{z})$. Let $v_0$ be the central vertex with weight $u_0$ and $v_1$ the central vertex with weight $u_1$. For each $i\in\{0,1\}$, $j\in\{1,\ldots,k\}$, let $L_{i,j}$ be the number of pendant edges incident with $v_i$ and a leaf of weight $w_j$. Suppose first that there are no leaves with weight $u_0$ or $u_1$. By Lemma~\ref{l:mdeg-1} (a), the coefficient $\alpha(u_i,w_j)$ counts the number of edges joining a vertex of weight $u_i$ with a vertex of weight $w_j$. But in this case $v_i$ is the unique vertex with weight $u_i$ and $w_j\neq u_{1-i}$. It follows that every edge being counted by $\alpha(u_i,w_j)$ must be a pendant edge joining the central vertex $v_i$ with a leaf of weight $w_j$. Hence in this case 
\[ L_{i,j} = \alpha(u_i,w_j)\quad\text{for all }i\in\{0,1\},j\in\{1,\ldots,k\}.\]

For the next case, suppose that there are leaves with weight $u_1$ but there are no leaves with weight $u_0$. That is, we assume that  there is $j_0\in\{1,\ldots,k\}$ such that $w_{j_0}=u_{1}$ but $w_{j}\neq u_{0}$ for all $j$. In this case, by using the same argument as before we obtain 
\[L_{i,j} = \alpha(u_i,w_j)\quad \text{for all } i\in\{0,1\},j\in\{1,\ldots,k\} \text{ such that } (i,j)\neq (0,j_0).\]
If $(i,j)=(0,j_0)$ then the central edge is an edge joining a vertex of weight $u_0$ with a vertex of weight $u_1=w_{j_0}$. It follows from Lemma~\ref{l:mdeg-1} (a) that $\alpha(u_0,w_{j_0}) = 1 + L_{0,j_0}$, from which we deduce that $L_{0,j_0}=\alpha(u_0,w_{j_0})-1$. Since this was the only missing leaf number, this finishes  the reconstruction in this case. The case when there are leaves with weight $u_0$ but no leaf has weight $u_1$ is analogous.

Finally, assume that there are leaves with weight $u_0$ and $u_1$; that is, there exist $j_0$ and $j_1$ such that $w_{j_i}=u_{1-i}$ for $i\in\{0,1\}$. In this case, by the same argument as before we get 
\[L_{i,j} = \alpha(u_i,w_j)\quad \text{for all } i\in\{0,1\},j\in\{1,\ldots,k\} \text{ such that } (i,j)\not\in\{ (0,j_0),(1,j_1)\}.\]
Thus, it remains to deduce the values of $L_{0,j_0}$ and $L_{1,j_1}$. On the one hand, edges joining a vertex of weight $u_0$ with a vertex of weight $u_1$ come in three types: pendant edges joining leaf of weight $w_{j_0}$ with  the center $v_0$; the central edge; and pendant edges connecting  a leaf of weight $w_{j_1}=u_0$  with the center $v_1$, which has weight $u_1=w_{j_0}$. Thus we have  
\[ \alpha(u_0,w_{j_0})= 1 + L_{0,j_0} + L_{1,j_1}.\]
On the other hand, the total number of leaves of weight $w_{j_0}$ is $\mu_{j_0}$. This implies that 
$L_{0,j_0} + L_{1,j_0} =\mu_{j_0}$. But $L_{1,j_0}=\alpha(u_1,j_0)$ so  we can solve these equations, which yields
\[ L_{0,j_0} = \mu_{j_0}-\alpha(u_1,w_{j_0})\quad\text{and}\quad L_{1,j_1} = \alpha(u_0,w_{j_0})-1-\mu_{j_0}+\alpha(u_1,w_{j_0}).\]
This finishes the reconstruction in this case. 

Assume now that $u_0=u_1$. The previous argument does not work since we cannot distinguish the centers by their weights. Nevertheless, we fix (at random) a labeling of the centers as $v_0$ and $v_1$ and let $L_{i,j}$ be the same as before.

By Lemma~\ref{l:mdeg-1} (b), for every $j_1,j_2$ in  $\{1,\ldots,k\}$, the coefficient $\beta(u_0,w_{j_1},w_{j_2})$  counts pairs of independent edges $e_1,e_2$ such that $e_i$ has endpoints with weights $u_0$ and $w_{j_i}$ for $i\in\{1,2\}$. Since $T$ is a 2-star, it follows that $e_1$ and $e_2$ must be  pendant edges that are incident to different centers. Thus, 
\begin{equation}
\label{eq:betaj1j2}
    \beta(u_0,w_{j_1},w_{j_2})=\begin{cases}L_{0,j_1}L_{1,j_1},&\text{if $j_1=j_2$};\\
L_{0,j_1}L_{1,j_2}+L_{0,j_2}L_{1,j_2},&\text{otherwise}.
\end{cases}
\end{equation}

For each weight $w_j$,  $j\in \{1, \ldots k\}$, we will determine the number of leaves of weight $w_j$ attached to each vertex from Equation \eqref{eq:betaj1j2} and the total number $\mu_j$ of  leaves with weight $w_j$.  Let $j_1=j_2=j$. It follows that 
\begin{equation}\label{eq:systemL}
L_{0,j} + L_{1,j} = \mu_j \quad\text{and}\quad L_{0,j}L_{1,j} = \beta(u_0,w_j,w_j)\quad\text{for all $j\in \{1,\ldots,k\}$}.
\end{equation}
This system has a unique solution $L_{0,j}=L_{1,j}=\mu_j/2$ if and only if  $\Delta_j := \mu_j^2-4\beta(u_0,w_j,w_j)=0$. Note that for every $j$, the discriminant $\Delta_j$ can be computed from $M(\mathbf{z})$. Note also that in the case that $\Delta_j\neq 0$ the previous system has two symmetric solutions. The symmetry follows from the fact that interchanging the labeling of the centers interchanges the values of $L_{0,j}$ and $L_{1,j}$ for all $j$.

This means that if $\Delta_j=0$ for all $j\in\{1,\ldots,k\}$ then both centers are adjacent to the same number of leaves of weight $w_j$ for all $j$. In this case we just place $\mu_j/2$ pendant edges of weight $w_j$ on each center and this finishes the reconstruction.

If $\Delta_{j_1}\neq 0$ for some $j_1\in \{1, \ldots, k\}$, then  there are two distinct symmetric solutions for $(L_{0,j_1},L_{1,j_1})$ for the system of equations \eqref{eq:systemL}
(this implies that $L_{0,j_1}\neq L_{1,j_1}$). 
We choose one of these solutions at random and attach leaves of weight $w_{j_1}$ to the centers  according with this solution. Now that  the leaves of weight $w_{j_1}$ have been placed the centers are distinguishable.
To place the other leaves we will use the monomials in Lemma~\ref{l:mdeg-1} (b) for each $j_2\neq j_1$. We know by the previous discussion that $\beta(u_0,w_{j_1},w_{j_2})$ counts the number of leaves of weight $w_{j_1}$ and $w_{j_2}$ attached to different centers. It follows that 
\[
L_{0,j_2} + L_{1,j_2} = \mu_{j_2}\quad\text{and}\quad
L_{0,j_1}L_{1,j_2}+L_{0,j_2}L_{1,j_1} = \beta(u_0,w_{j_1},w_{j_2}).
\]
Since $L_{1,j_1}\neq L_{0,j_1}$ are known, this is now a linear system with a unique solution $(L_{0,j_2},L_{1,j_2})$. Thus we can attach $L_{0,j_2}$  leaves of weight $w_{j_2}$ to the center $v_0$ and $L_{1,j_2}$ leaves of weight $w_{j_2}$ to the center $v_1$. Since we can do this for all $j_2\neq j_1$, this finishes the reconstruction of $(T,\omega)$. To finish the proof we only need to note that by starting with the other solution for $j_1$, we would have obtained the same weighted tree, up to $\omega$-isomorphism. 
\end{proof}

\subsection{On $D$-reconstructability}
Now we turn to the question of reconstructing weighted stars and $2$-stars from the $D$-polynomial. 

\begin{definition}
Let $(T,\omega)$ be a weighted tree and $(T,\mathsf{m}_\omega)$ its marked version. Given $k$, let $M_k(\mathbf{z})$ be the homogeneous component of degree $k$ in $M_{(T,\mathsf{m}_\omega)}(\mathbf{z})$ and $D_k(\mathbf{z}):=M_k(z_{w,d}=D_{\bullet_{w,d}})$.
\end{definition}

\begin{lemma}\label{lem:MkDk}
Let $(T,\omega)$ be a strictly weighted tree of order $n$ and total weight $N$. Then the following assertions hold: 
\begin{enumerate}
    \item Let $0\leq k\leq n$. Then, every monomial in $D_k(\mathbf{z})$ is of the form  
$z_1^i {z}_{\lambda_1
}z_{\lambda_2}\cdots z_{\lambda_k}$ where $0\leq i\leq n-k$ and $\lambda$ is a partition of $N-i$ having all parts different from $1$. In particular, $D(\mathbf{z})$ has degree $n$. 
    \item For every monomial in $D(\mathbf{z})$ of the form $z_{w_1}\cdots z_{w_k}$, where $\sum_i w_i=N$  and $w_j\geq 2$ for all $j$, we have 
    \[ [z_{w_1}\cdots z_{w_k}]D(\mathbf{z})=\sum_{\mathbf{d}}[z_{w_1,d_1}\cdots z_{w_k,d_k}]M(\mathbf{z}),\]
    where the sum ranges over all tuples $\mathbf{d}=(d_1,\ldots,d_k)$ of non-negative integers such that $\sum d_i = n - k$.
\end{enumerate}
\end{lemma}

\begin{proof}
We know that every monomial in $M_k(\mathbf{z})$ is of the form $z_{w_1,d_1}\cdots z_{w_k,d_k}$ where $\sum_j w_j=N$ and $\sum_{j}d_j=n-k$. In addition, each monomial in $D_{\bullet_{w_j,d_j}}$ is of the form $z_{w_j-i_j}z_1^{i_j}$ where $w_j-i_j\geq 2$ and  $0\leq i_j\leq d_j$ since every mark is strict. By combining these we get that each monomial in $D_k(\mathbf{z})$ is of the form 

\[z_1^{i_1+\ldots + i_k}z_{\lambda_1}\cdots z_{\lambda_k}\]
where $\lambda_j= w_j - i_j\geq2$. This shows $(a)$.

To prove $(b)$, by a similar argument as for $(a)$,  a monomial of the form $z_{w_1}\cdots z_{w_k}$ where $\sum_i w_i=N$ and $w_j\geq 2$ can only arise after undotting a term of the form $z_{w_1,d_1}\cdots z_{w_k,d_k}$ in $M(\mathbf{z})$ for some tuple $(d_1,\ldots,d_k)$ of non-negative integers satisfying $\sum_i d_i=n-k$. Since $D_{\bullet_{w,d}}=z_{w}+\text{other terms divisible by $z_1$}$, it follows that each occurrence of $z_{w_1,d_1}\cdots z_{w_k,d_k}$ in $M(\mathbf{z})$ contributes to one occurrence of $z_{w_1}\cdots z_{w_k}$. The conclusion now follows.
\end{proof}
This lemma can be seen as a first step in recovering $M$ from $D$. The idea is to recover all the weights from the terms of $D$ that are not divisible by $z_1$ and then use the terms divisible by $z_1$ to reconstruct the number of dots corresponding to each weight. However, in general, it may not be possible to reconstruct a polynomial $f(\mathbf{z})$ from $f(z_{w,d}=D_{\bullet_{w,d}})$ as the following example shows.
\begin{example}
\label{ex:strict}
Let $f(\mathbf{z}) =  z_{4,1}z_{5,2}+z_{4,2}z_{5,2}$ and $g(\mathbf{z}) = z_{4,1}z_{5,3}+z_{4,2}z_{5,1}$.  A direct computation shows
\begin{align*}
f(z_{w,d}=D_{\bullet_{w,d}})&= 2z_5z_4-3z_5z_3z_1+z_5z_2z_1^2-4z_4^2z_1+8z_4z_3z_1^2-2z_4z_2z_1^3-3z_3^2z_1^3+z_3z_2z_1^4\\
&=g(z_{w,d}=D_{\bullet_{w,d}}).
\end{align*}
Thus, we cannot reconstruct $f(\mathbf{z})$ (or $g(\mathbf{z})$) from $f(z_{w,d}=D_{\bullet_{w,d}})$. The geometry of polynomials that lead to the same polynomial after applying the undotting substitution will be studied in a forthcoming paper \cite{aliste2022markedII}.
\end{example}

\begin{prop}
\label{D:degreeone}
Let $(T,\omega)$ be a strictly weighted tree. Then, the degree-one term and the largest-degree term in $M(\mathbf{z})$ can be recovered from $D(\mathbf{z})$.
\end{prop}
\begin{proof}
By Lemma \ref{lem:MkDk} $(a)$, the largest-degree term in $D(\mathbf{z})$ that is not divisible by $z_1$ has the form $z_{w_1}\cdots z_{w_n}$ where $w_i\geq 2$ for all $i\in\{1,\ldots,n\}$ and $n$ is the order of $T$. By Lemma \ref{lem:MkDk} $(b)$, it is easy to see that 
\[ [z_{w_1}\cdots z_{w_n}]D(\mathbf{z}) = [z_{w_1,0}\cdots z_{w_n,0}]M(\mathbf{z}),\]
which recovers the largest-degree term of $M(\mathbf{z})$ and the order of $T$.

The term of degree one in $M(\mathbf{z})$ is $z_{N,n-1}$, where  $N$ is the total weight of $(T,\omega)$ and $n$ is the order of $T$. By Lemma \ref{lem:MkDk} $(a)$ there is a unique degree-one term in $D(\mathbf{z})$ which must be of the form $z_N$ and thus, using that we already computed the value of $n$ when dealing with the largest-degree term, we recover the degree-one term of $M(\mathbf{z})$.
\end{proof}
For terms of degree two we will use the following notation: Given  positive integers $N,d$, a \defn{$(N,d)-$ strict pair} is a pair of marks $(w_1,k_1), (w_2,k_2)\in\MM^{\circ}$ such that $w_1+w_2=N$, $k_1+k_2+1=d$ and $w_1 > w_2$ or $w_1=w_2$ and $k_1\leq k_2$. If $(N,d)$ is clear from the context, we refer to $(N,d)$-strict pairs just as strict pairs. 
Given a multiset $\mathcal{P}$ of strict pairs, define 
\[ L_\mathcal{P}(\mathbf{z}):=\sum_{((w_1,k_1),(w_2,k_2))\in \mathcal{P}}z_{w_1,k_1}z_{w_2,k_2}.\]
As observed in Example~\ref{ex:strict}, it is not clear whether  we can recover $L_\mathcal{P}(\mathbf{z})$, i.e., $\mathcal{P}$, from  $L_\mathcal{P}(z_{w,d}=D_{\bullet_{w,d}})$.
However the strictness of the pairs in $\mathcal{P}$ will allow us to do the reconstruction as the following lemma shows. 
\begin{lemma}\label{l:starleaf}
Let $N,d$ be integers and $\mathcal{P}$ be a multiset of strict pairs. Then $\mathcal{P}$ can  be reconstructed from the polynomial $L_\mathcal{P}(z_{w,k}=D_{\bullet_{w,k}})$.
\end{lemma}
\begin{proof}
Let us begin by noting that since all pairs are strict, the same argument as in the proof of Proposition~\ref{prop:MDproper} implies that a monomial appears with positive coefficient in $L_\mathcal{P}(z_{w,k}=D_{\bullet_{w,k}})$ only when it contains an even power of $z_1$.

The proof is by induction on the sum of the absolute values of the coefficients of $L_\mathcal{P}(z_{w,k}=D_{\bullet_{w,k}})$, which we denote $S_{\mathcal{P}}$. If $S_{\mathcal{P}}=0$, then the only possibility is that $\mathcal{P}$ is empty. For the inductive step, suppose that $S_{\mathcal{P}}>0$ and that we can reconstruct every $\mathcal{P}'$ from $L_{\mathcal{P}'}(z_{w,k}=D_{\bullet_{w,k}})$ whenever $S_{\mathcal{P}'}<S_{\mathcal{P}}$. Define $Q(\mathbf{z}) := L_\mathcal{P}(z_{w,k}=D_{\bullet_{w,k}})$.

Choose $a_0$ to be the smallest positive integer such that $[z_{N-a_0}z_{a_0}]Q(\mathbf{z})\neq 0 $. Since all marks in the pairs in  $\mathcal{P}$ are strict, it must be the case that $a_0>1$ and we can deduce that $\mathcal{P}$ contains at least one pair of the form $((N-a_0,d-k-1),(a_0,k))$ for some $k<a_0-1$. When undotting the second variable in a term $z_{N-a_0, d-k-1}z_{a_0,k}$ in $L_{\mathcal{P}}(\mathbf{z})$, we obtain in $Q(\mathbf{z})$ terms of the form $z_{N-a_0}z_{a_0-i}z_1^i$ for $0\leq i \leq k$. Moreover, by the choice of $a_0$ all such terms in $Q(\mathbf{z})$ arise after undotting a term of the form $z_{N-a_0, d-\ell-1}z_{a_0,\ell}$ for some $\ell\geq i$. 
Now let $d_0$ be the largest non-negative integer such that $[z_{N-a_0}z_{a_0-d_0}z_1^{d_0}]Q(\mathbf{z})\neq 0$. 
The observations above imply that $\mathcal{P}$ has a strict pair of the form $((N-a_0,d-\ell-1),(a_0,\ell))$ for $\ell\geq d_0$, and by the maximality of $d_0$, it must be $\ell=d_0$. 
Thus, we have proved that $p=((N-a_0,d-d_0-1),(a_0,d_0))$ is a strict pair that belongs to $\mathcal{P}$.

Now let 
\[Q'(\mathbf{z}):=Q(\mathbf{z})-D_{\bullet_{N-a_0,d-d_0-1}}D_{\bullet_{a_0,d_0}}\]
Clearly $Q'(\mathbf{z}) = L_{\mathcal{P}\setminus\{p\}}(z_{w,k}=D_{\bullet_{w,k}})$, and by the observation at the beginning of the proof, $S_{\mathcal{P}\setminus\{p\}}<S_{\mathcal{P}}$. By the induction hypothesis, we can reconstruct $\mathcal{P}\setminus\{p\}$ from $Q'(\mathbf{z})$. But we already know $p$ so this finishes the reconstruction of $\mathcal{P}$.
\end{proof}

 Given a strictly weighted tree $(T,\omega)$ of total weight $N$ and $n$ vertices, there is a unique multiset $\mathcal{P}$ of $(N,n-1)$-strict pairs such that $L_\mathcal{P}(\mathbf{z})=M_2(\mathbf{z})$, since  every term in  $M_2(\mathbf{z})$ is of the form $z_{w_1,d_1}\,z_{w_2,d_2}$ where 
 the marks $(w_1,d_1)$ and $(w_2,d_2)$ are strict.
It follows that  $L_\mathcal{P}(z_{w,k}=D_{\bullet_{w,k}}) = D_2(\mathbf{z})$.

\begin{coro}\label{c:leaves}
Let $(T,\omega)$ be a  strictly weighted tree. 
Then, $M_2(\mathbf{z})$ can be reconstructed  from $D_2(\mathbf{z})$. In particular, from the $D$-polynomial, we can reconstruct the multiset of leaf weights, the number of edges,  and recognize whether $(T,\omega)$ is a star or a $2$-star. 
\end{coro}

\begin{theorem}\label{thm:starsfromD}
Let $(T,\omega)$ be a strictly  weighted star. Then it can be reconstructed from its $D$-polynomial. 
\end{theorem}
\begin{proof}

A weighted star can be reconstructed from  its total weight and its multiset of leaf weights.  The total weight can be recovered by Proposition \ref{D:degreeone} while the multiset of leaf weights can be recovered by Corollary~\ref{c:leaves}. This gives the reconstruction of $(T,\omega)$. 
\end{proof}

\begin{example}
Consider the following  $D$-polynomial of an unknown tree $(T, \omega)$: 
\begin{align*}
D(\mathbf{z})=&-z_1^3z_9 + 3z_1^2z_{10} + 2z_1^2z_3z_7 + z_1^2z_4z_6 - 3z_1z_{11} - 4z_1z_3z_8 - 2z_1z_4z_7 \\&- 2z_1z_3z_4^2 - z_1z_3^2z_5+ z_{12}  + 2z_3z_9  + z_4z_8 + z_3^2z_6 + 2z_3z_4z_5 + z_2z_3^2z_4.
\end{align*}
First, if we sum the absolute value of all the coefficients we get $27=3^3$, which, by Corollary \ref{coro:3n1} implies that $(T,\omega)$ is a strictly weighted tree with $3$ edges. Next, from the term $z_{12}$ we recover the total weight $N=12$ and from the term $z_2z_3^2z_4$, which can be characterized as the term of largest degree that is not divisible by $z_1$, we recover the multiset of vertex weights $\{ 2,3,3,4 \}$.

Finally, we recover $M_2(\mathbf{z})$ from $D_2(\mathbf{z})$.  By Lemma \ref{lem:MkDk} $(a)$,
\begin{equation}
    \label{eq:D2}
 D_2(\mathbf{z}) = 2z_1^2z_3z_7 + z_1^2z_4z_6  - 4z_1z_3z_8 - 2z_1z_4z_7 + 2z_3z_9  + z_4z_8.
 \end{equation} 
We follow the proof of Lemma \ref{l:starleaf}. It is clear from Equation~\eqref{eq:D2} that $a_0=3$. Also, the term $z_3z_9$ is the only term of the form $z_{9}z_{3-i}z_1^{i}$, so $d_0= 0$ and it follows that $(9,3),(3,0)$ is a strict pair in $\mathcal{P}$.  After subtracting $z_{3}z_{9}$ and following the same argument we obtain that the multiplicity of $(9,3),(3,0)$ in $\mathcal{P}$ is $2$. Finally, after subtracting again $z_{3}z_{9}$ we get that $a_0=4$ and that $d_0=0$ since $z_{4}z_{8}$ is the only term of the form $z_{8}z_{4-i}z_1^i$. It follows that  $(8,3),(4,0)$ belongs to $\mathcal{P}$. This finishes the reconstruction of $\mathcal{P}$. We deduce that $T$ is a star and that the multiset of leaf-weights is $\{3,3,4\}$. Since $N=12$ it follows that the center has weight $2$.
\end{example}

\begin{lemma}
\label{lem:degn-1}
Let $(T,\omega)$ be a strictly weighted tree with $n$ vertices. Then the terms of degree $n-1$ of the $M$-polynomial can be reconstructed from the $D$-polynomial.
\end{lemma}

\begin{proof}
Let $N$ be the total weight of $(T,\omega)$, which is known by Proposition~\ref{D:degreeone}. 
Let $\{ \tau_1, \tau_2 , \ldots, \tau_l \}$ be the set of
 partitions of $N$ of length $n-1$ with all parts larger than one and such that
 $a_i:=[\mathbf{z}_{\tau_i}]D(\zz) > 0$ for all $i\in\{1,\ldots,l\}$. Using  Proposition~\ref{D:degreeone} we can also compute the multiset $\mathcal{W}$ of vertex weights of $(T, \omega)$.  

Since the terms of degree $n-1$ without the variable $z_1$ appear after undotting a term of degree $n-1$ of $M$ 
and  in the states model representation the terms of degree $n-1$ of $M$ are induced by an edge, we have that for each $i \in \{1,2,\ldots, k\}$ there exist at least one edge $e = uv$ of $(T,\omega)$ such that 
\[ \mathbf{z}_{\tau_i}z_{\omega(u)}z_{\omega(v)} = {\mathbf{z}_\mathcal{W}}z_{\omega(u)+\omega(v)}.\]

This implies that   the term 
$z_{\omega(u)+\omega(v), 1}\mathbf{z}_{\mathcal{W} \setminus \{\omega(u),\omega(v)\}}$ 
appears in $M(\mathbf{z})$. Moreover, we  deduce that $[\mathbf{z}_{\tau_i}]D(\zz)$ is equal to the number of edges
joining a vertex of  weight  $\omega(u)$ with a vertex of weight $\omega(v)$. Hence, by  Lemma~\ref{l:mdeg-1} we get that

$$[z_{\omega(u)+\omega(v), 1}\mathbf{z}_{\mathcal{W} \setminus \{\omega(u),\omega(v)\}}]M(\zz) = [\mathbf{z}_{\tau_i}]D(\zz),$$
holds  for all $\tau_i$, which conclude the proof.

\end{proof}

\begin{theorem}\label{thm:2starsfromD}
Let $(T,\omega)$ be a strictly weighted $2$-star. 
Then $(T,\omega)$ can be reconstructed from the $D$-polynomial.
\end{theorem}

\begin{proof}
We use the same notation as in the proof of Theorem~\ref{t:m2star}. First, Proposition~\ref{D:degreeone}, combined with Proposition~\ref{prop:leaves} ensure that we can recover the total weight $N$, the number of vertices  $n$ and  the multiset of vertex weights $\mathcal{W}$ from $D(\mathbf{z})$.
By Corollary~\ref{c:leaves} we can also recover the multiset of leaf  weights $\mathcal{W}_L$ from $D(\mathbf{z})$. 
From this we can deduce the weights $u_0$ and $u_1$ of the centers, the distinct leaf weights $w_1,\ldots,w_k$ and the multiplicities $\mu_j$ of each $w_j$ in $\mathcal{W}_L$ for each $j\in\{1,\ldots,k\}$.

Suppose first that  $u_0\neq u_1$. By the proof of Theorem~\ref{t:m2star}, it suffices to recover the coefficients 
\[\alpha({u_i,w_j})=[z_{u_i+w_j,1}\,\mathbf{z}_{\mathcal{W}\setminus\{u_i,w_j\}}]M(\mathbf{z})\]
from $D(\mathbf{z})$, but since the terms involved have degree $n-1$ in $M(\mathbf{z})$, this can be done by Lemma~\ref{lem:degn-1}.  

Suppose now that $u_0=u_1$. We consider the same coefficients 
\[\beta(u_0,w_{j_1}, w_{j_2})=[z_{u_0+w_{j_1},1}\,z_{u_0+w_{j_2},1}\,\mathbf{z}_{\mathcal{W}\setminus\{u_0,u_0,w_{j_1},w_{j_2}\}}]M(\mathbf{z})\] 
as in the proof of Theorem~\ref{t:m2star}. Define 
\[\mathcal{W}(j_1,j_2)=\mathcal{W}\setminus\{u_0,u_0,w_{j_1},w_{j_2}\}\cup\{u_0+w_{j_1},u_0+w_{j_2}\}.\]
We claim that 
\begin{equation}\label{claim}
[\mathbf{z}_{\mathcal{W}(j_1,j_2)}]D(\mathbf{z}) = \beta(u_0,w_{j_1}, w_{j_2}) + K_{j_1,j_2},
\end{equation} 
where $K_{j_1,j_2}$ is defined as follows: $K_{j_1,j_2}=0$ if $w_{j_1}\neq u_{0}+w_{j_2}$ and $w_{j_2}\neq u_{0}+w_{j_1}$; otherwise,   $K_{j_1,j_2}$ counts the number of edge sets $\{e_1,e_2\}$ such that $e_1$ and $e_2$ share a common vertex and the weights of the other vertices are either $\{u_0,w_{j_2}\}$ (when $w_{j_1}=u_0+w_{j_2}$) or $\{u_0,w_{j_1}\}$ (when  $w_{j_2}=u_0+w_{j_1}$).

Indeed, let $\tau$ be the integer partition obtained after reordering the weights in $\mathcal{W}(j_1,j_2)$ so that $\mathbf{z}_\tau = \mathbf{z}_{\mathcal{W}(j_1,j_2)}$. Note that every part of $\tau$ is  different from $1$ since $(T,\omega)$ is strictly weighted. Thus, by Corollary~\ref{WfromD}, the coefficient $[\mathbf{z}_\tau]D(\mathbf{z})$ counts the number of edge sets $A=\{e_1,e_2\}$ such that $\lambda(T,\omega, A)=\tau$. Let $A=\{e_1,e_2\}$ be any set of two distinct edges. The edges $e_1$ and $e_2$ can be either independent or share a common vertex: 
\begin{enumerate}
    \item  If $e_1$ and $e_2$ are independent, then they must be pendant edges incident to different centers; let $a_1$ and $a_2$ be the weights of the leaves incident with $e_1$ and $e_2$ respectively. Then $\lambda(T,\omega,A)=\{u_0+a_1,u_0+a_2\}\cup\mathcal{W}\setminus\{u_0,u_0,a_1,a_2\}$. It follows that a set of independent edges $e_1$ and $e_2$ contributes to the coefficient $[\mathbf{z}_{\tau}]D(\mathbf{z})$ if and only if $\{a_1,a_2\}=\{w_{j_1},w_{j_2}\}$. But from the proof of Theorem~\ref{t:m2star} and Lemma~\ref{l:mdeg-1} (b) we know that this case is counted by $\beta(u_0,w_{j_1},w_{j_2})$;
    \item If $e_1$ and $e_2$ share a common vertex, then they contribute to  the coefficient $[\mathbf{z}_{\tau}]D(\mathbf{z})$ if and only if $\lambda(T,\omega,A)=\tau$. Let $a_1$ and $a_2$ be the weights of the non-common vertices of $e_1$ and $e_2$. It is easy to check 
    that \[\lambda(T,\omega,A)=\{u_0+a_1+a_2\}\cup\mathcal{W}\setminus\{u_0,a_1,a_2\}.\] 
    It follows that if $w_{j_2}\leq w_{j_1}$ then $\lambda(T,\omega,A)=\tau$ if and only if $\{a_1,a_2\}=\{u_0,w_{j_2}\}$ and $w_{j_1}=u_0+w_{j_2}$.
    This proves the claim.
\end{enumerate}
Note that setting $j_1=j_2=j$, for $j\in\{1,\ldots,k\}$,  the value of $K_{j,j}$ is zero and  Equation \eqref{claim} implies that 
\[\Delta_{j}=\mu_j^2-4\beta(u_0,w_j,w_j)=\mu_j^2-4[\mathbf{z}_{\mathcal{W}(w_j,w_j)}]D(\mathbf{z}).\]
That is, we can deduce  the value of $\Delta_j$ for every $j\in\{1,\ldots,k\}$ from $D(\mathbf{z})$. Now we check whether $\Delta_j=0$ for all $j\in\{1,\ldots,k\}$. If it is the case, then the proof follows from the same arguments as the proof of Theorem~\ref{t:m2star}. Thus for the remainder of the proof we assume there exists $j_1\in\{1,\ldots,k\}$ such that $\Delta_{j_1}\neq 0$ and choose a solution $(L_{0,j_1},L_{1,j_1})$ for the leaf numbers of weight $w_{j_1}$. Note that by the proof of Theorem~\ref{t:m2star} we know that  for every $j_2\in\{1,\ldots,k\}\setminus\{j_1\}$,  the system 
\begin{equation}
    \label{eq:the_system}
L_{0,j_2}+L_{1,j_2} = \mu_{j_2}\quad\text{and}\quad L_{0,j_2}L_{1,j_1}+L_{0,j_1}L_{0,j_2}=\beta(u_0,w_{j_1},w_{j_2})
\end{equation}
has a unique solution $(L_{0,j_2},L_{1,j_2})$ since $L_{0,j_1}\neq L_{1,j_1}$. 
But from Equation \eqref{claim} we know how to recover  $\beta(u_0,w_{j_1},w_{j_2})$ 
for every $j_2$ such that $w_{j_2}\not\in\{w_{j_1}+u_0,w_{j_1}-u_0\}$. Thus if there is $j_1\in\{1,\ldots,k\}$ such that $\Delta_{j_1}\neq 0$ and $w_{j_1}+u_0$ and $w_{j_1}-u_0$ do not occur as leaf weights, then we 
have 
\[ L_{0,j_2}+L_{1,j_2} = \mu_{j_2}\quad\text{and}\quad L_{0,j_2}L_{1,j_1}+L_{0,j_1}L_{0,j_2}=[\mathbf{z}_{\mathcal{W}(w_{j_1},w_{j_2})}]D(\mathbf{z})\]
for all $j_2$. Since this system also has a unique solution, this would finish the proof.

It only remains us to compute $\beta(u_0,w_{j_1},w_{j_2})$ if there exists $j_2\in\{1,\ldots,k\}$ such that  $w_{j_2}=w_{j_1}+u_0$ or $w_{j_1}=w_{j_2}+u_0$. We explain the case $w_{j_1}=w_{j_2}+u_0$, the other case being analogous. By Equation \eqref{claim} it suffices to count the number $K_{j_1,j_2}$ of edge sets  $A=\{e_1,e_2\}$ where $e_1$ and $e_2$ share a common vertex (which is a center), and the weights of the other vertices are $u_0$ and $w_{j_2}$. Since $T$ is a $2$-star, there are two possibilities: 
\begin{enumerate}[(i)]
    \item either  $e_1$ is the central edge and $e_2$ is a pendant edge incident with a leaf of weight $w_{j_2}$;
    \item or $e_1$ is a pendant edge incident with a vertex of weight $u_{0}$ and $e_2$ is a pendant edge incident with a vertex of weight $w_{j_2}$.
\end{enumerate}
If $u_0$ does not occur as a leaf weight, then only case (i) is possible and thence $K_{j_1,j_2}=\mu_{j_2}$. Otherwise, let $j_3$ be the unique $j\in\{1,\ldots,k\}$ such that $u_0=w_{j}$. Then, case (ii) is also  possible, and it can be counted indirectly as the total number of (unordered) pairs of leaves of weight $u_0$ and $w_{j_2}$ minus the number of leaves of weight $u_0$ and $w_{j_2}$ that are adjacent to different centers. By Lemma~\ref{l:mdeg-1} (b), this last number equals 
$\beta(u_0, w_{j_2},u_0)$. 
It follows that 
\begin{equation*}
K_{j_1,j_2} = \begin{cases}
\mu_{j_3} + \binom{\mu_{j_3}}{2} - \beta(u_0,u_0,u_0),&\text{if $j_2=j_3$};
\\
\mu_{j_2} + \mu_{j_3}\mu_{j_2} - \beta (u_0, w_{j_2},u_0),& \text{otherwise}.
\end{cases}
\end{equation*}
Note that  $\gamma:=\mu_{j_3}+\binom{\mu_{j_3}}{2}-\beta(u_0,u_0,u_0)=\mu_{j_3}$ if $\mu_{j_3}<2$ and equals $\mu_{j_3}+\binom{\mu_{j_3}}{2}-[\mathbf{z}_{\mathcal{W}(u_0,u_0)}]D(\mathbf{z})$ when $\mu_{j_3}\geq 2$. 
On the other hand, when $j_3\neq j_2$, from Equation \eqref{claim} we get
\[\beta(u_0, w_{j_2},u_{0})= [\mathbf{z}_{\mathcal{W}(u_0,w_{j_2})}]D(\mathbf{z})-\delta_{w_{j_2}-2u_0}K_{j_2,j_3},
\]
where $\delta_{a}=1$ if $a=0$ and $0$ otherwise;  and $K_{j_2,j_3}$ is the number of edge sets $A = \{e_1,e_2\}$ where $e_1$ and $e_2$ share a common vertex and the weights of the non-common vertices are $\{u_0,u_0\}$. Following the same argument as the one for computing $K_{j_1,j_2}$ in the case $j_2=j_3$ we get that $K_{j_2,j_3}=\gamma$.
Hence, by combining these relations we recover
\begin{align}
\beta(u_0,w_{j_1},w_{j_2})=&[\mathbf{z}_{\mathcal{W}(w_{j_1},w_{j_2})}]D(\mathbf{z})-
\gamma\delta_{j_2-j_3} \nonumber \\&- (1-\delta_{j_2-j_3})\Big(\delta_{w_{j_2},2u_0}\gamma +
\mu_{j_2} + \mu_{j_3}\mu_{j_2} - [\mathbf{z}_{\mathcal{W}(u_0,w_{j_2})}]D(\mathbf{z})\Big). \label{eq:finalb}
\end{align}
Thus, by plugging this last relation into the system \eqref{eq:the_system}, we can solve for $L_{0,j_2}$ and $L_{1,j_2}$. Since these were the only remaining leaf numbers to recover, we have reconstructed $(T,\omega)$ and this finishes the proof.
\end{proof}

We now give an example to illustrate the ideas of Theorem \ref{thm:2starsfromD}.

\begin{example}
In this example we show how to reconstruct a strictly weighted $2$-star $(T,\omega)$ from $D(\mathbf{z})$. The complete $D$-polynomial and all the computations are done in Sagemath and they are available in \cite{alistecodeI}. The term of degree one and the unique term of largest degree that is not divisible by $z_1$  in $D(\mathbf{z})$ are
\[ z_{28} + z_{2}^3z_{3}^2z_4z_5z_6.\]
From this information, following Proposition \ref{D:degreeone} we  deduce that the total weight $N$ is $28$, that the multiset of vertex weights is $\{2,2,2,3,3,4,5,6\}$ that the order $n$ of $T$ is $8$ and thus $T$ has $7$ edges.
\begin{align*}
    D_2(\mathbf{z})=&z_{1}^6z_{8}z_{13} + z_{1}^6z_{6}z_{15} +
    z_{1}^6z_{5}z_{16} + z_{1}^6z_{4}z_{17} + 2z_{1}^6z_{3}z_{18} +
    z_{1}^6z_{2}z_{19} - 3z_{1}^5z_{9}z_{13} - 3z_{1}^5z_{8}z_{14} \\ 
    &- 6z_{1}^5z_{6}z_{16} - 6z_{1}^5z_{5}z_{17} - 6z_{1}^5z_{4}z_{18}
    - 12z_{1}^5z_{3}z_{19} - 6z_{1}^5z_{2}z_{20} + 3z_{1}^4z_{10}z_{13} +
    9z_{1}^4z_{9}z_{14} \\
    &+ 3z_{1}^4z_{8}z_{15} + 15z_{1}^4z_{6}z_{17} 
    +15z_{1}^4z_{5}z_{18} + 15z_{1}^4z_{4}z_{19} + 30z_{1}^4z_{3}z_{20} 
    + 15z_{1}^4z_{2}z_{21} - z_{1}^3z_{11}z_{13} \\
    &- 9z_{1}^3z_{10}z_{14}
    - 9z_{1}^3z_{9}z_{15} - z_{1}^3z_{8}z_{16} - 20z_{1}^3z_{6}z_{18}
    - 20z_{1}^3z_{5}z_{19} - 20z_{1}^3z_{4}z_{20}- 40z_{1}^3z_{3}z_{21}\\
    &- 20z_{1}^3z_{2}z_{22} + 3z_{1}^2z_{11}z_{14} + 9z_{1}^2z_{10}z_{15}
    + 3z_{1}^2z_{9}z_{16} + 15z_{1}^2z_{6}z_{19} + 15z_{1}^2z_{5}z_{20}+ 15z_{1}^2z_{4}z_{21}\\
    & + 30z_{1}^2z_{3}z_{22} + 15z_{1}^2z_{2}z_{23}
    - 3z_{1}z_{11}z_{15} - 3z_{1}z_{10}z_{16} - 6z_{1}z_{6}z_{20} -
    6z_{1}z_{5}z_{21} - 6z_{1}z_{4}z_{22}\\
    &- 12z_{1}z_{3}z_{23} -
    6z_{1}z_{2}z_{24} + z_{11}z_{16} + z_{6}z_{21} + z_{5}z_{22} +
    z_{4}z_{23} + 2z_{3}z_{24} + z_{2}z_{25}.
\end{align*}
Following the proof of Lemma \ref{l:starleaf} we recover the multiset of leaf-weights $\{2,3,3,4,5,6\}$. Thus, $w_1=2, w_2=3,w_3=4,w_4=5,w_5=6$ ,$\mu_1=1$, $\mu_2=2$, $\mu_3=1$ , $\mu_4=1$ and $\mu_5=1$. By comparing the multiset of leaf weights with the multiset of vertex weights we deduce that both centers have weight $2$. 

Since both centers have equal weight, we need to compute the $\Delta_i$'s. From Lemma \ref{l:mdeg-1} $(b)$, it is easy to see that $\Delta_i=1$ if $\mu_i=1$ because in this case $\beta(u_0,w_i,w_i)$ is necessarily $0$. It follows
that $\Delta_1=\Delta_3=\Delta_4=\Delta_5=1$. On the other hand, 
$\Delta_2 = 4 - \beta(2,3,3) = 4-4[z_2z_4z_5^2z_6^2]D(\mathbf{z})=0$ by Equation \eqref{claim} in the proof of the Theorem.

 We also have $j_3=1$ and 
$\gamma = \mu_1 = 1$. We also choose $j_1=5$,  so $w_{j_1}=6$ and we fix $L_{0,5} = 1$ and $L_{1,5} = 0$. 
To compute $L_{0,2}$ and $L_{1,2}$ we notice that $2 \notin \{ 8, 4\} $ then $\beta(2,6,2) = [z_{\mathcal{W}(2,6)}]D(\zz) = [z_3^2z_4^2z_5z_8]D(\zz) = 1$. From Equation~\eqref{eq:the_system} we deduce $L_{0,2} = 0$ and $L_{1,2} = 1$. In the same way we can compute $L_{0,5} = 1$ and $L_{1,5} = 0$.

The case of $L_{0,4}$ and $L_{1,4}$ is different because $w_3 = 4 = w_{j_1} - u_0 = 6-2$. In this case $j_2 = 3$ and $j_3=1$. From Equation~\eqref{eq:finalb} of the proof we get that:
\begin{eqnarray*}
 \beta(2,6,4) &=& [\zz_{\mathcal{W}(6,4)}]D(\zz) - (\gamma + \mu_3 + \mu_1\mu_3 - [\zz_{\mathcal{W}(2,4)}]D(\zz))\\
 &=& 3 -(1 +2+2 -1) = 1
\end{eqnarray*}
which implies $L_{0,4} = 0$ and $L_{1,4} = 1$.

Hence, we can reconstruct the 2-star which have a center of weight 2 with attached leaves of weight 6, 5 and 3; and another center of weight 2 and attached leaves of weights 2,3 and 4. 
\end{example}

From Theorems~\ref{thm:starsfromD} and~\ref{thm:2starsfromD} and the discussion at the beginning of this section we immediately deduce the following result on Stanley's tree isomorphism question. Recall that a tree is proper if every internal vertex is adjacent to at least one leaf.

\begin{coro}
Proper trees of diameter less or equal than $5$ can be reconstructed from their chromatic symmetric functions.
\end{coro}


\begin{thebibliography}{10}

\bibitem{aliniaeifard2020extended}
Farid Aliniaeifard, Victor Wang, and Stephanie van Willigenburg.
\newblock Extended chromatic symmetric functions and equality of ribbon schur
  functions.
\newblock {\em arXiv preprint arXiv:2010.00147}, 2020.

\bibitem{alistecodeI}
Jos\'e Aliste-Prieto.
\newblock Reconstructing trees from their marked graph polynomials.
\newblock Code available at
  \texttt{https://github.com/jaliste/marked\_graphs\_one}, 2022.

\bibitem{aliste2021vertex}
Jos{\'e} Aliste-Prieto, Logan Crew, Sophie Spirkl, and Jos{\'e} Zamora.
\newblock A vertex-weighted {T}utte symmetric function, and constructing graphs
  with equal chromatic symmetric function.
\newblock {\em Electronic journal of combinatorics}, 28(2), 2021.

\bibitem{aliste2022markedII}
Jos{\'e} Aliste-Prieto, Anna de~Mier, Rosa Orellana, and Jos{\'e} Zamora.
\newblock Marked graphs and the chromatic symmetric function ii.
\newblock {\em in preparation}, 2021.

\bibitem{aliste2017trees}
Jos{\'e} Aliste-Prieto, Anna de~Mier, and Jos{\'e} Zamora.
\newblock On trees with the same restricted {${U}$}-polynomial and the
  {P}rouhet--{T}arry--{E}scott problem.
\newblock {\em Discrete Mathematics}, 340(6):1435--1441, 2017.

\bibitem{aliste2014proper}
Jos{\'e} Aliste-Prieto and Jos{\'e} Zamora.
\newblock Proper caterpillars are distinguished by their chromatic symmetric
  function.
\newblock {\em Discrete Mathematics}, 315:158--164, 2014.

\bibitem{chan2021tree}
William Chan and Logan Crew.
\newblock Tree bases of chromatic symmetric functions.
\newblock {\em arXiv preprint arXiv:2110.15291}, 2021.

\bibitem{chmutovshah}
Chmutov.
\newblock Private communication, 2021.

\bibitem{cho2016chromatic}
Soojin Cho and Stephanie van Willigenburg.
\newblock Chromatic bases for symmetric functions.
\newblock {\em The Electronic Journal of Combinatorics}, 23(1):P1--15, 2016.

\bibitem{crew2020deletion}
Logan Crew and Sophie Spirkl.
\newblock A deletion--contraction relation for the chromatic symmetric
  function.
\newblock {\em European Journal of Combinatorics}, 89:103143, 2020.

\bibitem{dahlberg2017lollipop}
Samantha Dahlberg and Stephanie van Willigenburg.
\newblock Lollipop and lariat symmetric functions.
\newblock {\em arXiv preprint arXiv:1702.06974}, 2017.

\bibitem{ellis2011tutte}
Joanna Ellis-Monaghan and Iain Moffatt.
\newblock The {T}utte--{P}otts connection in the presence of an external
  magnetic field.
\newblock {\em Advances in Applied Mathematics}, 47(4):772--782, 2011.

\bibitem{gasharov1999stanley}
Vesselin Gasharov.
\newblock On stanley's chromatic symmetric function and clawfree graphs.
\newblock {\em Discrete mathematics}, 205(1-3):229--234, 1999.

\bibitem{grinberg2020hopf}
Darij Grinberg and Victor Reiner.
\newblock {\em Hopf algebras in combinatorics}.
\newblock Mathematisches Forschungsinstitut Oberwolfach gGmbH, 2020.

\bibitem{guay2013modular}
Mathieu Guay-Paquet.
\newblock A modular relation for the chromatic symmetric functions of (3+
  1)-free posets.
\newblock {\em arXiv preprint arXiv:1306.2400}, 2013.

\bibitem{heil2019algorithm}
S~Heil and C~Ji.
\newblock On an algorithm for comparing the chromatic symmetric functions of
  trees.
\newblock {\em Australasian Journal of Combinatorics}, 75(2):210--222, 2019.

\bibitem{OEIS}
OEIS~Foundation Inc.
\newblock The on-line encyclopedia of integer sequences, 2019.
\newblock [Online].

\bibitem{loebl2018isomorphism}
Martin Loebl and Jean-S{\'e}bastien Sereni.
\newblock Isomorphism of weighted trees and {S}tanley's isomorphism conjecture
  for caterpillars.
\newblock {\em Annales de l’Institut Henri Poincar{\'e} D}, 01 2018.

\bibitem{macdonald1998symmetric}
Ian~Grant Macdonald.
\newblock {\em Symmetric functions and Hall polynomials}.
\newblock Oxford university press, 1998.

\bibitem{martin2008distinguishing}
Jeremy~L Martin, Matthew Morin, and Jennifer~D Wagner.
\newblock On distinguishing trees by their chromatic symmetric functions.
\newblock {\em Journal of Combinatorial Theory, Series A}, 115(2):237--253,
  2008.

\bibitem{noble99weighted}
Steven Noble and Dominic Welsh.
\newblock A weighted graph polynomial from chromatic invariants of knots.
\newblock {\em Ann. Inst. Fourier (Grenoble)}, 49(3):1057--1087, 1999.
\newblock Symposium \`a la M\'emoire de Fran\c cois Jaeger (Grenoble, 1998).

\bibitem{orellana2014graphs}
Rosa Orellana and Geoffrey Scott.
\newblock Graphs with equal chromatic symmetric functions.
\newblock {\em Discrete Mathematics}, 320:1--14, 2014.

\bibitem{penaguiao2018kernel}
Raul Penaguiao.
\newblock The kernel of chromatic quasisymmetric functions on graphs and
  nestohedra.
\newblock {\em arXiv preprint arXiv:1803.08824}, 2018.

\bibitem{sagan2001symmetric}
Bruce~E Sagan.
\newblock {\em The symmetric group: representations, combinatorial algorithms,
  and symmetric functions}, volume 203.
\newblock Springer, 2001.

\bibitem{Scott-thesis}
Geoffrey Scott.
\newblock Characterizing graphs with equal chromatic functions.
\newblock Senior Thesis, Dartmouth College, 2008.

\bibitem{shareshian2016chromatic}
John Shareshian and Michelle~L Wachs.
\newblock Chromatic quasisymmetric functions.
\newblock {\em Advances in Mathematics}, 295:497--551, 2016.

\bibitem{stanley95symmetric}
Richard~P. Stanley.
\newblock A symmetric function generalization of the chromatic polynomial of a
  graph.
\newblock {\em Adv. Math.}, 111(1):166--194, 1995.

\bibitem{stanley1999enumerative}
Richard~P Stanley.
\newblock {\em Enumerative combinatorics. Vol. 2, volume 62 of Cambridge
  Studies in Advanced Mathematics}.
\newblock Cambridge University Press, Cambridge, 1999.

\bibitem{thatte2020connected}
Bhalchandra~D Thatte.
\newblock The connected partition lattice of a graph and the reconstruction
  conjecture.
\newblock {\em Journal of Graph Theory}, 93(2):181--202, 2020.

\bibitem{west2001introduction}
Douglas West.
\newblock {\em Introduction to graph theory}.
\newblock Prentice hall, 2001.

\end{thebibliography}
\end{document}